\documentclass[a4paper, notitlepage, 12pt]{article}
\usepackage{ucs}                % Unicode support in LaTeX
\usepackage[utf8x]{inputenc} 	% We use UTF-8 encoded files only 
\usepackage[T1]{fontenc}     % for the right font encoding
\usepackage{lmodern}         % fonts for good pdf files

\usepackage{color}
\usepackage{latexsym}
\usepackage{amsmath,amssymb,amsbsy,amsthm}
\usepackage{wasysym}
\usepackage[sfdefault=cmbr,OMLmathsans]{isomath} % has everything slanted, also
\usepackage{upgreek}   % upright Greek letters

\usepackage{float}
\usepackage{algpseudocode}
\usepackage[section]{algorithm}
\usepackage{tensor}
\usepackage{hyperref}
\usepackage{authblk}

\usepackage{multirow} 
\usepackage{array}
\usepackage{graphicx}
\usepackage{caption}

\usepackage{rcs} \RCS $Revision: 3.9 $ \RCS $Date: 2015/10/28 15:22:19 $ \RCS $Author: matthies $ \RCS $RCSfile: NonLinBU.tex,v $ \RCS $Id: NonLinBU.tex,v 3.9 2015/10/28 15:22:19 matthies Exp $

\usepackage[british]{babel}

\newcommand{\citep}[1]{\cite{#1}}

\usepackage{refdef}
\usepackage{ifontdef}

\newtheorem{thm}{Theorem}

\newtheorem{prop}[thm]{Proposition}

\newtheorem{defi}[thm]{Definition}

\newtheorem{lem}[thm]{Lemma}

\newcommand{\ignore}[1]{}

\newcommand{\vrho}{\varrho}
\newcommand{\vepsilon}{\varepsilon}

\newcommand{\vphi}{\varphi}

\newcommand{\vek}[1]{\mathchoice{\displaystyle\boldsymbol{#1}}
{\textstyle\boldsymbol{#1}}{\scriptstyle\boldsymbol{#1}}
{\scriptscriptstyle\boldsymbol{#1}}}

\newcommand{\tnb}[1]{\mathchoice{\displaystyle\mathboldsans{#1}}
{\textstyle\mathboldsans{#1}}{\scriptstyle\mathboldsans{#1}}
{\scriptscriptstyle\mathboldsans{#1}}}

\newcommand{\tns}[1]{\mathchoice{\displaystyle\mathsans{#1}}
{\textstyle\mathsans{#1}}{\scriptstyle\mathsans{#1}}
{\scriptscriptstyle\mathsans{#1}}}

\newcommand{\vtil}[1]{\vek{\tilde{#1}}}
\newcommand{\vhat}[1]{\vek{\hat{#1}}}

\newcommand{\That}[1]{\tnb{\hat{#1}}}

\newcommand{\EXP}[1]{\mathbb{E}\left(#1\right)}
 
\newcommand{\dlangle}{\langle\negthinspace\langle}
\newcommand{\drangle}{\rangle\negthinspace\rangle}

\newcommand{\divg}{\mathop{\mathrm{div}}\nolimits}

\newcommand{\spn}{\mathop{\mathrm{span}}\nolimits}

\newcommand{\dd}{\partial}
\newcommand{\di}{\mathrm{d}}

\newcommand{\ii}{\mathchoice{\displaystyle\mathrm i}
{\textstyle\mathrm i}{\scriptstyle\mathrm i}
{\scriptscriptstyle\mathrm i}}

\newcommand{\ip}[2]{\langle #1, #2 \rangle}
\newcommand{\bkt}[2]{\langle #1 | #2 \rangle}
\newcommand{\ipj}[1]{\langle #1 \rangle}
\newcommand{\ipd}[2]{\dlangle #1, #2 \drangle}
\newcommand{\bkd}[2]{\dlangle #1 | #2 \drangle}
\newcommand{\ns}[1]{| #1 |}
\newcommand{\nd}[1]{\| #1 \|}

\newcommand{\Hf}[2]{\tensor[^#1]{#2}{}}

\newcommand{\Lperp}{\mathop{\underline{\perp}}\nolimits}
\newcommand{\stindep}{\mathop{\perp\negthickspace\negthickspace\perp}\nolimits}

\definecolor{myred}{rgb}{1, 0.2, 0.2}

\newcommand{\authorhgm}{Hermann G. Matthies}
\newcommand{\authoral}{Alexander Litvinenko}
\newcommand{\authorbr}{Bojana V. Rosi\'c}
\newcommand{\authorez}{Elmar Zander}
\newcommand{\authorop}{Oliver Pajonk}

\newcommand{\affilwire}{Institute of Scientific Computing \authorcr
                        Technische Universit\"at Braunschweig}

\newcommand{\thetitle}{Inverse Problems in a Bayesian Setting}

\newcommand{\theauthor}{\authorhgm\thanks{corresponding author}, \authorez,
            \authorbr, \authorcr \authoral, \authorop}
\newcommand{\thesubject}{---(MSC2010) 62F15, 65N21, 62P30, 60H15, 60H25, 74G75,
            80A23, 74C05\\ 
            ---(PACS2010) 46.65.+g, 46.35.+z, 44.10.+i\\
            ---(ACM1998) G.1.8, G.3, J.2}
\newcommand{\thekeywords}{inverse identification, uncertainty
  quantification, Bayesian update, parameter identification,
  conditional expectation, filters, functional and spectral approximation}
\newcommand{\textdate}{20th October 2015}

\begin{document}

% ============================================================================
% connect to default LaTeX values
\title{\thetitle\thanks{Partly supported by the Deutsche
          Forschungsgemeinschaft (DFG) through SFB 880.}}

\author{\theauthor}

\affil{\affilwire}

%\date{\today}
%\date{}
\date{\textdate}

%%%% END IGNORE
%%-%%

\maketitle

% RCSID:       $Id: abstract.tex,v 3.2 2015/10/22 00:45:43 matthies Exp $
% Author:      $Author: matthies $
% Contact:     wire@tu-bs.de
% ============================================================================
%% texfile{
%%  AUTHOR    = "$Author: matthies $",
%%  VERSION   = "$Revision: 3.2 $",
%%  DATE      = "$Date: 2015/10/22 00:45:43 $",
%%  FILENAME  = "$RCSfile"}
%
% =================================

\begin{abstract}
  In a Bayesian setting, inverse problems and uncertainty
  quantification (UQ) --- the propagation of uncertainty through a
  computational (forward) model --- are strongly connected.  In the
  form of conditional expectation the Bayesian update becomes
  computationally attractive.  We give a detailed account of this
  approach via conditional approximation, various approximations, and
  the construction of filters.  Together with a functional or spectral
  approach for the forward UQ there is no need for time-consuming and
  slowly convergent Monte Carlo sampling.  The developed sampling-free
  non-linear Bayesian update in form of a filter is derived from the
  variational problem associated with conditional expectation.  This
  formulation in general calls for further discretisation to make the
  computation possible, and we choose a polynomial approximation.
  After giving details on the actual computation in the framework of
  functional or spectral approximations, we demonstrate the workings
  of the algorithm on a number of examples of increasing complexity.
  At last, we compare the linear and nonlinear Bayesian update in form
  of a filter on some examples.

%\ignore{            %%%% BEGIN IGNORE
\vspace{5mm}

\vspace{5mm}
{\noindent\textbf{Classification:} \thesubject}
%}            %%%% END IGNORE

\end{abstract}

%  $Log: abstract.tex,v $
%  Revision 3.2  2015/10/22 00:45:43  matthies
%  final corrections
%
%  Revision 3.1  2015/06/12 04:23:04  hgm
%  check-in for new paper Sarajevo 2015
%
%  Revision 2.1  2013/12/17 22:30:32  hgm
%  arXiv version
%
%  Revision 1.2  2013/12/15 16:43:27  hgm
%  final
%
%  Revision 1.1  2013/12/11 21:57:24  hgm
%  initial check in, not final
%
%
%
%

%%% Local Variables: 
%%% mode: latex
%%% TeX-master: "../NonLinBU"
%%% End: 

%\pagenumbering{roman}

% table of contents if desired
%\cleardoublepage

%\tableofcontents
%\cleardoublepage

% list of figures if desired
%\listoffigures
%\cleardoublepage

%list of symbols if desired
%\listofsymbols
%\cleardoublepage

%\pagenumbering{arabic}

% use include to start each section on a new page.

% RCSID:       $Id: introduction.tex,v 3.7 2015/10/28 15:26:06 matthies Exp $
% Author:      $Author: matthies $
% Contact:     wire@tu-bs.de
% =================================
%% texfile{
%%  AUTHOR    = "$Author: matthies $",
%%  VERSION   = "$Revision",
%%  DATE      = "$Date: 2015/10/28 15:26:06 $",
%%  FILENAME  = "$RCSfile: introduction.tex,v $"}
%
% =================================

\section{Introduction}  \label{S:intro}
Inverse problems deal with the determination of parameters in
computational models, by comparing the prediction of these models with
either real measurements or observations, or other, presumably more
accurate, computations.  These parameters can typically not be
observed or measured directly, only other quantities which are somehow
connected to the one for which the information is sought.  But it is
typical that we can compute what the observed response should be,
under the assumption that the unknown parameters have a certain
value.  And the difference between predicted or forecast response is
obviously a measure for how well these parameters were identified.

There are different ways of attacking the problem of parameter
identification theoretically and numerically.  One way is to define
some measure of discrepancy between predicted observation and the
actual observation.  Then one might use optimisation algorithms to
make this measure of discrepancy as small as possible by changing the
unknown parameters.  Classical least squares approaches start from this
point.  The parameter values where a minimum is attained is then
usually taken as the `best' value and regarded as close to the `true'
value.

One of the problems is that for one the measure of discrepancy crops
up pretty arbitrarily, and on the other hand the minimum is often not
unique.  This means that there are many parameter values which explain
the observations in a `best' way.  To obtain a unique solution, some
kind of `niceness' of the optimal solution is required, or
mathematically speaking, for the optimal solution some regularity is
enforced, typically in competition with discrepancy measure to be
minimised.  This optimisation approach hence leads to regularisation
procedures, a good overview of which is given by \citep{Engl2000}.
 
Here we take another tack, and base our approach on the Bayesian idea
of updating the knowledge about something like the unknown parameters
in a probabilistic fashion according to Bayes's theorem.  In order to
apply this, the knowledge about the parameters has to be described in
a \emph{Bayesian} way through a probabilistic model \citep{jaynes03},
\citep{Tarantola2004}, \citep{Stuart2010}.  As it turns out, such a
probabilistic description of our previous knowledge can often be
interpreted as a regularisation, thus tying these differing approaches
together.

The Bayesian way is on one hand difficult to tackle, i.e.\ finding a
computational way of doing it; and on the other hand often becomes
computationally very demanding.  One way the Bayesian update may be
achieved computationally is through sampling. On the other hand, we
shall here use a functional approximation setting to address such
stochastic problems.  See \citep{boulder:2011} for a synopsis on our
approach to such parametric problems.

% Testing:
% \[ \ipj{a}, \ip{a}{b}, \ipd{\tns{a}}{\tns{b}},  \bkt{a}{b}, \bkd{\tns{a}}{\tns{b}},
%    \ns{a}, \nd{\tns{a}}, \nt{A} \]

It is well-known that such a Bayesian update is in fact closely
related to \emph{conditional expectation} \citep{Bobrowski2006/087},
\citep{Goldstein2007}, and this will be the basis of the method
presented.  For these and other probabilistic notions see for example
\citep{Papoulis1998/107} and the references therein.

The functional approximation approach towards stochastic problems is
explained e.g.\ in \citep{matthies6}.  These approximations are in the
simplest case known as Wiener's so-called \emph{homogeneous} or
\emph{polynomial chaos} expansion \citep{Wiener1938}, which are
polynomials in independent Gaussian RVs --- the `chaos' --- and which
can also be used numerically in a Galerkin procedure
\citep{ghanemSpanos91}, \citep{matthiesKeese05cmame},
\citep{matthies6}.  This approach has been generalised to other types
of RVs \citep{xiuKarniadakis02a}.  It is a computational variant of
\emph{white noise analysis}, which means analysis in terms of
independent RVs, hence the term `white noise' \citep{holdenEtAl96},
\citep{Janson1997}, \citep{hida}, see also \citep{matthiesKeese05cmame},
\citep{Roman_Sarkis_06}, and \citep{GalvisSarkis:2012} for here relevant
results on stochastic regularity.  Here we describe a computational
extension of this approach to the inverse problem of Bayesian
updating, see also \citep{opBvrAlHgm12}, \citep{bvrAlOpHgm12-a}, \citep{OpBrHgm12},
\citep{BvrAkJsOpHgm11}.

To be more specific, let us consider the following situation:
we are investigating some physical system which
is modelled by an evolution equation for its state: 
\begin{equation}  \label{eq:I}
      \frac{\di}{\di t}u = A(q;u(t)) + \eta(q;t); \qquad u(0) = u_a\text{ given} .
\end{equation}
where $u(t) \in \C{U}$ describes the state of the system
at time $t \in [0,T]$ lying in a Hilbert space $\C{U}$ (for the sake of
simplicity), $A$ is a---possibly non-linear---operator modelling
the physics of the system, and $\eta\in\C{U}^*$ is some external
influence (action / excitation / loading).  Both $A$ and $\ell$ may
involve some \emph{noise} --- i.e.\ a random process --- so that
\eqref{eq:I} is a stochastic evolution equation.

Assume that the model depends on some parameters $q \in \C{Q}$, which
are uncertain.  These may actually include the initial conditions for
the state, $u_a$.  To have a concrete example of \refeq{eq:I},
consider the diffusion equation
\begin{equation} \label{eq:I-c}
\frac{\dd}{\dd t}u(x,t) - \divg (\kappa(x) \nabla u(x,t)) = \eta(x,t),
\quad x\in\C{G},
\end{equation}
with appropriate boundary and initial conditions, where $\C{G} \subset
\D{R}^n$ is a suitable domain.  The diffusing quantity is $u(x,t)$
(heat, concentration) and the term $\eta(x,t)$ models sinks and
sources.  Similar examples will be used for the numerical experiments
in \refS{bayes-lin} and \refS{bayes-non-lin}.  Here $\C{U} =
H^1_E(\C{G})$, the subspace of the Sobolev space $H^1(\C{G})$
satisfying the essential boundary conditions, and we assume that the
diffusion coefficient $\kappa(x)$ is uncertain.  The parameters could
be the positive diffusion coefficient field $\kappa(x)$, but for
reasons to be explained fully later, we prefer to take $q(x) =
\log(\kappa(x))$, and assume $q \in \C{Q} = L_2(\C{G})$.

The updating methods have to be well defined and stable in a
continuous setting, as otherwise one can not guarantee numerical
stability with respect to the PDE discretisation refinement, see
\citep{Stuart2010} for a discussion of related questions.  Due to this
we describe the update before any possible discretisation in the
simplest Hilbert space setting.

On the other hand, no harm will result for the basic understanding if
the reader wants to view the occurring spaces as finite dimensional
Euclidean spaces.  Now assume that we observe a function of the state
$Y(u(q),q)$, and from this observation we would like to identify the
corresponding $q$.  In the concrete example \refeq{eq:I-c} this could
be the value of $u(x_j,t)$ at some points $x_j \in \C{G}$.  This is
called the \emph{inverse} problem, and as the mapping $q\mapsto Y(q)$
is usually not invertible, the inverse problem is \emph{ill-posed}.
Embedding this problem of finding the best $q$ in a larger class by
modelling our knowledge about it with the help of probability theory,
then in a Bayesian manner the task becomes to estimate conditional
expectations, e.g.\ see \citep{jaynes03}, \citep{Tarantola2004},
\citep{Stuart2010}, and the references therein.  The problem now is
\emph{well-posed}, but at the price of `only' obtaining probability
distributions on the possible values of $q$, which now is modelled as
a $\C{Q}$-valued random variable (RV).  On the other hand one
naturally also obtains information about the remaining uncertainty.
Predicting what the measurement $Y(q)$ should be from some assumed $q$
is computing the \emph{forward} problem.  The \emph{inverse} problem
is then approached by comparing the forecast from the forward problem
with the actual information.

Since the parameters of the model to be estimated are uncertain, all
relevant information may be obtained via their stochastic description.
In order to extract information from the posterior, most estimates
take the form of expectations w.r.t.\ the posterior.  These
expectations --- mathematically integrals, numerically to be evaluated
by some quadrature rule --- may be computed via asymptotic,
deterministic, or sampling methods.  In our review of current work we
follow our recent publications \citep{opBvrAlHgm12}, \citep{bvrAlOpHgm12-a},
\citep{OpBrHgm12}, \citep{BvrAkJsOpHgm11}.

One often used technique is a Markov chain Monte Carlo (MCMC) method
\citep{Madras-Fields:2002}, \citep{Gamerman06}, constructed such that the
asymptotic distribution of the Markov chain is the Bayesian posterior
distribution; for further information see \citep{BvrAkJsOpHgm11} and
the references therein.
% This can be then sampled by letting the Markov chain run for a
% sufficiently long time, although the samples are not independent in
% this case.  With the intention of accelerating the MCMC method some
% authors \citep{Marzouk2009a, Kucherova10, Pence2010,
% kucerovaSykBvrHgm:2012} have introduced stochastic spectral methods
% into the computation.  Expanding the prior random process into a
% polynomial chaos (PCE) or a \KL{} expansion (KLE) (e.g.\
% \citep{matthies6}), the inverse problem becomes an inference on the
% weights of the \KL{} modes.  Pence et al.\ \citep{Pence2010} combine
% polynomial chaos theory with maximum likelihood estimation, where
% the parameter estimates are calculated in a recursive or iterative
% manner.  Christen and Fox \citep{Christen2005} have applied a local
% linearisation of the forward model to improve the acceptance
% probability of proposed moves, while in \citep{Balakrishnan2003,
% Ma2009, Li2009, zengZhang:2010} collocation methods are employed as
% a more efficient sampling technique.

These approaches require a large number of samples in order to obtain
satisfactory results.  Here the main idea here is to perform the
Bayesian update directly on the polynomial chaos expansion (PCE)
without any sampling \citep{opBvrAlHgm12}, \citep{bvrAlOpHgm12-a},
\citep{boulder:2011}, \citep{OpBrHgm12}, \citep{BvrAkJsOpHgm11}.  This
idea has appeared independently in \citep{Blanchard2010a} in a simpler
context, whereas in \citep{saadGhn:2009} it appears as a variant of
the Kalman filter (e.g.\ \citep{Kalman}).  A PCE for a push-forward of
the posterior measure is constructed in \citep{moselhyYMarz:2011}.

From this short overview it may already have become apparent that the
update may be seen abstractly in two different ways.  Regarding the
uncertain parameters
\begin{equation}  \label{eq:RVq}
q: \Omega \to \C{Q} \text{  as a RV on a probability space   }
  (\Omega, \F{A}, \D{P})
\end{equation}
where the set of elementary events is $\Omega$, $\F{A}$ a
$\sigma$-algebra of events, and $\D{P}$ a probability measure, one set
of methods performs the update by changing the probability measure
$\D{P}$ and leaving the mapping $q(\omega)$ as it is, whereas the
other set of methods leaves the probability measure unchanged and
updates the function $q(\omega)$.  In any case, the push forward
measure $q_* \D{P}$ on $\C{Q}$ defined by $q_* \D{P}(\C{R}) :=
\D{P}(q^{-1}(\C{R}))$ for a measurable subset $\C{R} \subset \C{Q}$ is
changed from prior to posterior.  For the sake of simplicity we assume
here that $\C{Q}$ --- the set containing possible realisations of
$q$ --- is a Hilbert space.  If the parameter $q$ is a RV, then so is
the state $u$ of the system \refeq{eq:I}.  In order to avoid a
profusion of notation, unless there is a possibility of confusion, we
will denote the random variables $q, f, u$ which now take values in
the respective spaces $\C{Q}, \C{U}^*$ and $\C{U}$ with the same
symbol as the previously deterministic quantities in \refeq{eq:I}.

In our overview \citep{BvrAkJsOpHgm11} on spectral methods in
identification problems, we show that Bayesian identification methods
\citep{jaynes03}, \citep{Tarantola2004}, \citep{Goldstein2007},
\citep{Stuart2010} are a good way to tackle the identification
problem, especially when these latest developments in functional
approximation methods are used.  In the series of papers
\citep{bvrAlOpHgm12-a}, \citep{boulder:2011}, \citep{OpBrHgm12},
\citep{BvrAkJsOpHgm11}, Bayesian updating has been used in a
linearised form, strongly related to the Gauss-Markov theorem
\citep{Luenberger1969}, in ways very similar to the well-known Kalman
filter \citep{Kalman}.  These similarities ill be used to construct an
abstract linear filter, which we term the \tbf{Gauss-Markov-Kalman}
filter (GMKF).  This turns out to be a linearised version of
\emph{conditional expectation}.  Here we want to extend this to a
non-linear form, and show some examples of linear (LBU) and non-linear
(QBU) Bayesian updates.

The organisation of the remainder of the paper is as follows: in
\refS{bayes} we review the Bayesian update---classically defined via
conditional probabilities---and recall the link between conditional
probability measures and conditional expectation.  In \refS{char-rv},
first we point out in which way --- through the conditional
expectation --- the posterior measure is characterised by Bayes's
theorem, and we point out different possibilities.  Often, one does
not only want a characterisation of the posterior measure, but
actually an RV which has the posterior measure as push-forward or
distribution measure.  Some of the well-known filtering algorithms
start from this idea.  Again by means of the conditional expectation,
some possibilities of construction such an RV are explored, leading to
`filtering' algorithms.

In most cases, the conditional expectation can not be computed
exactly.  We show how the abstract version of the conditional
expectation is translated into the possibility of real computational
procedures, and how this leads to various approximations, also in
connection with the previously introduced filters.

We show how to approximate the conditional expectation up to any
desired polynomial degree, not only the linearised version
\citep{Luenberger1969}, \citep{Kalman} which was used in
\citep{opBvrAlHgm12}, \citep{bvrAlOpHgm12-a}, \citep{boulder:2011},
\citep{OpBrHgm12}, \citep{BvrAkJsOpHgm11}.  This representation in
monomials is probably numerically not very advantageous, so we
additionally show a version which uses general function bases for
approximation.

The numerical realisation in terms of a functional or spectral
approximations --- here we use the well known Wiener-Hermite chaos ---
is shortly sketched in \refS{num-real}.  In \refS{bayes-lin} we then
show some computational examples with the \emph{linear version (LBU)},
whereas in \refS{bayes-non-lin} we show how to compute with the
non-linear or quadratic (QBU) version.  Some concluding remarks are
offered in \refS{concl}.
%%

%  $Log: introduction.tex,v $
%  Revision 3.7  2015/10/28 15:26:06  matthies
%  content described better
%
%  Revision 3.6  2015/10/22 00:46:13  matthies
%  final corrections
%
%  Revision 3.5  2015/10/22 00:36:20  matthies
%  spell-check finished for Adnan, Sarajevo, book
%
%  Revision 3.4  2015/10/22 00:30:35  matthies
%  really finished for Adnan, Sarajevo, book
%
%  Revision 3.3  2015/10/22 00:27:52  matthies
%  essentially finished for Adnan, Sarajevo, book
%
%  Revision 3.2  2015/10/21 21:17:38  matthies
%  some corrections for new paper
%
%  Revision 3.1  2015/10/05 11:22:51  matthies
%  started changes for new chapter for Adnan, Sarajevo, book
%
%  Revision 2.1  2013/12/17 22:30:42  hgm
%  arXiv version
%
%  Revision 1.5  2013/12/15 00:44:41  hgm
%  reference to section corrected
%
%  Revision 1.4  2013/12/14 16:44:01  hgm
%  corrected citations
%
%  Revision 1.3  2013/12/12 01:08:24  hgm
%  Wed night
%
%  Revision 1.2  2013/12/11 23:28:16  hgm
%  practically finished
%
%  Revision 1.1  2013/12/11 22:03:54  hgm
%  initial check in, not final
%
%
%
%

%%% Local Variables: 
%%% mode: latex
%%% TeX-master: "../NonLinBU"
%%% End: 

% RCSID:       $Id: bayes.tex,v 3.11 2015/10/28 15:26:24 matthies Exp $
% Author:      $Author: matthies $
% Contact:     wire@tu-bs.de
% =================================
%% texfile{
%%  AUTHOR    = "$Author",
%%  VERSION   = "$Revision: 3.11 $",
%%  DATE      = "$Date: 2015/10/28 15:26:24 $",
%%  FILENAME  = "$RCSfile"}
%
% =================================

\section{Bayesian Updating} \label{S:bayes}
Here we shall describe the frame in which we want to treat the problem
of Bayesian updating, namely a dynamical system with time-discrete
observations and updates.  After introducing the setting in
\refSS{bayes-setting}, we recall Bayes's theorem in
\refSS{bayes-laplace-thm} in the formulation of Laplace, as well as
its formulation in the special case where densities exist, e.g.\
\citep{Bobrowski2006/087}.  The next \refSS{cond-expect} treats the
more general case and its connection with the notion of
\emph{conditional expectation}, as it was established by Kolmogorov,
e.g.\ \citep{Bobrowski2006/087}.  This notion will be the basis of our
approach to characterise a RV which corresponds to the posterior measure.

\subsection{Setting} \label{SS:bayes-setting}
In the setting of \refeq{eq:I} consider the following problem:
one makes observations $y_n$ at times
$0 < t_1 < \dots < t_n \dots \in [0,T]$, and from these one would like to
infer what $q$ (and possibly $u(q;t)$) is.  In order to include a
possible identification of the state $u(q;t_n)$, we shall define a new
variable $x = (u, q)$, which we would thus like to identify:

Assume that $U:\C{U}\times\C{Q}\times[0,T] \ni (u_a,q,t) \mapsto
u(q;t) \in\C{U}$ is the flow or solution operator of \refeq{eq:I},
i.e.\ $u(q;t)=U(u_a,t_a,q,t)$, where $u_a$ is the initial condition at
time $t_a$.  We then look at the operator which advances the variable
$x=(u, q) \in \C{X} = \C{U}\times\C{Q}$ from $x_n=(u_n,q)$ at time
$t_n$ to $x_{n+1}=(u_{n+1},q)$ at $t_{n+1}$, where the Hilbert space
$\C{X}$ carries the natural inner product implied from $\C{U}$ and
$\C{Q}$,
\[ x_n=(u_n,q) \mapsto x_{n+1}=(u_{n+1},q) = (U(u_n,t_n,q,t_{n+1}),q) \in \C{X},\]
 or a bit more generally encoded in an operator $\hat{f}$:
\begin{equation}  \label{eq:dyn}
  \forall n\in\D{N}_0: \qquad 
      x_{n+1} = \hat{f}(x_n, w_n, n); \qquad x_0 = x_a \in \C{X} \text{ given} .
\end{equation}
This is a discrete time step advance map, for example of the dynamical
system \refeq{eq:I}, where a random `error' term $w_n$ is included,
which may be used to model randomness in the dynamical system per se,
or possible discretisation errors, or both, or similar things.  Most
dynamical --- and also quasi-static and stationary systems,
considering different loadings as a sequence in some pseudo-time ---
can be put in the form \refeq{eq:dyn} when observed at discrete points
in time.  Obviously, for fixed model parameters like $q$ in
\refeq{eq:I} the evolution is trivial and does not change anything,
but the \refeq{eq:dyn} allows to model everything in one formulation.

Often the dependence on the random term is assumed to be linear, so
that one has
\begin{equation}  \label{eq:dyn-l}
  \forall n\in\D{N}_0: \qquad 
      x_{n+1} = f(x_n) + \vepsilon S_x(x_n) w_n; \qquad x_0 = x_a  \text{ given},
\end{equation}
where the scalar $\vepsilon \geq 0$ explicitly measures the size of
the random term $w_n$, which is now assumed to be discrete white noise
of unit variance and zero mean, and possible correlations are
introduced via the linear operator $S_x(x_n)$.

But one can not observe the entity $q$ or $u(q; t)$, i.e.\ $x=(q,u)$
directly---like in Plato's cave allegory we can only see a `shadow'
--- here denoted by a vector $y\in\C{Y}$ --- of
it, formally given by a `measurement operator'
\begin{equation}  \label{eq:iI}
Y: \C{X}=\C{Q} \times \C{U} \ni (q,u(t_n)) \mapsto y_{n+1} = Y(q; u(t_n)) \in \C{Y},
\end{equation}
where for the sake of simplicity we assume $\C{Y}$ to be a Hilbert
space.

Typically one considers also some observational `error'
$\vepsilon v_n$, so that the observation may be expressed as 
\begin{equation} \label{eq:iiI}
y_{n+1} = H(Y(q; u(t_n)),\vepsilon v_n) = \hat{h}(x_n,\vepsilon v_n),
\end{equation}
where similarly as before $v_n$ is a discrete white noise process, and
the observer map $H$ resp.\ $\hat{h}$ combines the `true' quantity
$Y(q; u(t_n))$ to be measured with the error, to give the observation
$y_n$..

Translating this into the notation of the discrete dynamical system
\refeq{eq:dyn}, one writes
\begin{equation}  \label{eq:dyn-m}
  y_{n+1} = \hat{h}(x_n, \vepsilon v_n) \in \C{Y},
\end{equation}
where again the operator $\hat{h}$ is often assumed to be linear in
the noise term, so that one has similarly to \refeq{eq:dyn-l}
\begin{equation} \label{eq:dyn-ml}
  y_{n+1} = h(x_n) + \vepsilon S_y(x_n) w_n \in \C{Y}.  
\end{equation}

The mappings $Y$ in \refeq{eq:iI}, $H$ in \refeq{eq:iiI}, $\hat{h}$ in
\refeq{eq:dyn-m}, resp.\ $h$ \refeq{eq:dyn-ml} are usually not
invertible and hence the problem is called \emph{ill-posed}.  One way to
address this is via regularisation (see e.g.\ \cite{Engl2000}), but
here we follow a different track.  Modelling our lack of knowledge
about $q$ and $u(t_n)$ in a Bayesian way \cite{Tarantola2004} by
replacing them with a $\C{Q}$- resp.\ $\C{U}$-valued random variable
(RV), the problem becomes well-posed \cite{Stuart2010}.  But of course
one is looking now at the problem of finding a probability
distribution that best fits the data; and one also obtains a
probability distribution, not just \emph{one} pair $x_n=(q, u(t_n))$.

We shall allow for $\C{X}$ to be an infinite-dimensional space, as
well as for $\C{Y}$; although practically in any real situation only
finitely many components are measured.  But by allowing for the
infinite-dimensional case, we can treat the case of partial
differential equations --- PDE models --- like \refeq{eq:I} directly
and not just their discretisations, as its often done, and we only use
arguments which are independent on the number of observations.  In
particular this prevents hidden dependencies on local compactness, the
dimension of the model, or the number of measurements, and the
possible break-down of computational procedures as these dimensions
grow, as they will be designed for the infinite-dimensional case.  The
procedure practically performed in a real computation on a
finite-dimensional model and a finite-dimensional observation may then
be seen as an approximation of the infinite-dimensional case, and
analysed as such.

Here we focus on the use of a Bayesian approach inspired by the
`linear Bayesian' approach of \citep{Goldstein2007} in the framework of
`white noise' analysis \citep{hida}, \citep{holdenEtAl96},
\citep{Janson1997}, \citep{malliavin97}, \citep{Wiener1938},
\citep{xiuKarniadakis02a}.  Please observe that although the unknown
`truth' $x_n$ may be a deterministic quantity, the model for the
observed quantity $y_{n+1}$ involves randomness, and it therefore
becomes a RV as well.

To complete the mathematical setup we assume that $\Omega$ is a
measure space with $\sigma$-algebra $\F{A}$ and with a probability
measure $\D{P}$, and that $x: \Omega \to \C{X}$ and similarly $q, u$,
and $y$ are random variables (RVs).  The corresponding
\emph{expectation} will be denoted by $\bar{x} = \EXP{x} =
\int_{\Omega} x(\omega)\; \D{P}(\di \omega)$, giving the mean
$\bar{x}$ of the random variable, also denoted by $\ipj{x} :=
\bar{x}$.  The quantity $\tilde{x} := x - \bar{x}$ is the zero-mean or
fluctuating part of the RV $x$.

The space of vector valued RVs, say $x:\Omega \to \C{X}$, will for
simplicity only be considered in the form $\E{X} = \C{X} \otimes
\C{S}$, where $\C{X}$ is a Hilbert space with inner product
$\ip{\cdot}{\cdot}_{\C{X}}$, $\C{S}$ is a Hilbert space of scalar RVs ---
here we shall simply take $\C{S}=L_2(\Omega,\F{A},\D{P})$ --- with
inner product $\ip{\cdot}{\cdot}_{\C{S}}$, and the tensor product
signifies the Hilbert space completion with the scalar product as
usually defined for elementary tensors $x_1 \otimes s_1, x_2 \otimes
s_2 \in\E{X}$ with $x_1, x_2 \in \C{X}$ and $s_1, s_2 \in \C{S}$ by
\[ \ipd{x_1 \otimes s_1}{x_2 \otimes s_2}_{\E{X}} := 
           \ip{x_1}{x_2}_{\C{X}}\ip{s_1}{s_2}_{\C{S}} ,\]
and extended to all of $\E{X}$ by linearity.

Obviously, we may also consider the expectation not only as a linear
operator $\D{E}:\E{X} \to \C{X}$, but, as $\C{X}$ is isomorphic to the
subspace of constants $\E{X}_c := \C{X}\otimes \spn\{1\} \subset
\E{X}$, also as an orthogonal projection onto that subspace $\D{E} =
P_{\E{X}_c}$, and we have the orthogonal decomposition
\[ \E{X} = \E{X}_c \oplus \E{X}_c^\perp, \text{ with } \E{X}_c^\perp
=: \E{X}_0 , \]
where $\E{X}_0$ is the zero-mean subspace, so that
\[  \forall x\in\E{X}:\quad \bar{x} = \EXP{x} = P_{\E{X}_c} x \in
\E{X}_c,\; \tilde{x} = (I-P_{\E{X}_c}) x \in \E{X}_0 . \]

Later, the covariance operator between two Hilbert-space valued RVs
will be needed.  The covariance operator between two RVs $x$ and $y$
is denoted by
\[ C_{x y} : \C{Y} \ni v \mapsto \EXP{\tilde{x}\,
  \ip{\tilde{y}}{v}_{\E{Y}}} \in \C{X} \cong \E{X}_c. \]
For $x \in \C{X}\otimes \C{S}$ and $y \in \C{Y}\otimes \C{S}$ it is
also often written as $C_{x y}=\EXP{\tilde{x} \otimes \tilde{y}}$.

\subsection{Recollection of Bayes's theorem} \label{SS:bayes-laplace-thm}
Bayes's theorem is commonly accepted as a consistent way to incorporate
new knowledge into a probabilistic description \citep{jaynes03},
\citep{Tarantola2004}, and its present mathematical form is due to
Laplace, so that a better denomination would be the
\emph{Bayes-Laplace} theorem.

The elementary textbook statement of the theorem is about
conditional probabilities
\begin{equation}  \label{eq:iII}
 \D{P}(\C{I}_x|\C{M}_y) = \frac{\D{P}(\C{M}_y|\C{I}_x)}{\D{P}(\C{M}_y)}\D{P}(\C{I}_x),
 \quad \D{P}(\C{M}_y) > 0,
\end{equation}
where $\C{I}_x \subseteq \C{X}$ is some measurable subset of possible
$x$'s, and the measurable subset $\C{M}_z \subseteq \C{Y}$ is the
information provided by the measurement. Here the conditional
probability $\D{P}(\C{I}_x|\C{M}_y)$ is called the \emph{posterior}
probability, $\D{P}(\C{I}_x)$ is called the \emph{prior} probability,
the conditional probability $\D{P}(\C{M}_y|\C{I}_x)$ is called the
\emph{likelihood}, and $\D{P}(\C{M}_y)$ is called the evidence.  The
\refeq{eq:iII} is only valid when the set $\C{M}_y$ has non-vanishing
probability measure, and becomes problematic when $\D{P}(\C{M}_y)$
approaches zero, cf.\ \citep{jaynes03}, \citep{rao2005}.  This arises
often when $\C{M}_y = \{y_m\}$ is a one-point set representing a
measured value $y_m \in \C{Y}$, as such sets have typically vanishing
probability measure.  In fact the well-known \emph{Borel-Kolmogorov
  paradox} has led to numerous controversies and shows the possible
ambiguities \citep{jaynes03}.  Typically the posterior measure is
singular w.r.t.\ the prior measure, precluding a formulation in
densities.  Kolmogorov's resolution of this situation shall be
sketched later.

One well-known very special case where the formulation in densities is
possible, which has particular requirements on the likelihood, is
when $\C{X}$---as here---is a metric space, and there is a background
measure $\mu$ on $(\C{X}, \F{B}_{\C{X}})$ --- $\F{B}_{\C{X}}$ is the
Borel-$\sigma$-algebra of $\C{X}$ --- and similarly with $\nu$ and
$(\C{Y}, \F{B}_{\C{Y}})$, and the RVs $x$ and $y$ have probability
density functions (pdf) $\pi_x(x)$ w.r.t.\ $\mu$ and $\pi_y(y)$
w.r.t.\ $\nu$ resp.{}, and a joint density $\pi_{xy}(x,y)$ w.r.t.\
$\mu\otimes\nu$.  Then the theorem may be formulated as
(\citep{Tarantola2004} Ch.\ 1.5, \citep{rao2005}, \citep{jaynes03})
\begin{equation}  \label{eq:iIIa}
 \pi_{(x|y)}(x|y) = \frac{\pi_{xy}(x,y)}{\pi_y(y)} =
 \frac{\pi_{(y|x)}(y|x)}{Z_y} \pi_x(x),
\end{equation}
where naturally the marginal density $Z_y := \pi_y(y) = \int_{\C{X}}
\pi_{xy}(x,y)\;\mu(\di x)$ (from German \emph{Zustandssumme}) is a
normalising factor such that the conditional density
$\pi_{(x|y)}(\cdot|y)$ integrates to unity w.r.t\ $x$.  In this case
the limiting case where $\D{P}(\C{M}_y)$ vanishes may be captured via
the metric \citep{rao2005} \citep{jaynes03}.  The joint density
\[ \pi_{xy}(x,y) = \pi_{(y|x)}(y|x) \pi_x(x) \]
may be factored into the likelihood function $\pi_{(y|x)}(y|x)$ and
the prior density $\pi_x(x)$, like $\pi_y(y)$ a marginal density, $\pi_x(x) =
\int_{\C{Y}} \pi_{xy}(x,y)\;\nu(\di y)$.  These terms in the second
equality in \refeq{eq:iIIa} are in direct correspondence with those in
\refeq{eq:iII}.  Please observe that the model for the RV representing
the error in \refeq{eq:dyn-m} determines the likelihood functions
$\D{P}(\C{M}_y|\C{I}_x)$ resp.\ $\pi_{(y|x)}(y|x)$.  To require the
existence of the joint density is quite restrictive.  As
\refeq{eq:dyn-m} shows, $y$ is a function of $x$, and a joint density
on $\C{X}\times\C{Y}$ will generally not be possible as $(x,y) \in
\C{X} \times \C{Y}$ are most likely on a sub-manifold; but the
situation of \refeq{eq:dyn-ml} is one possibility where a joint
density may be established.  The background densities are typically in
finite dimensions the Lebesgue measure on $\D{R}^d$, or more general
Haar measures on locally compact Lie-groups \citep{Segal1978}.  Most
computational approaches determine the pdfs \citep{Marzouk2007},
\citep{Stuart2010}, \citep{Kucherova10}.

However, to avoid the critical cases alluded to above, Kolmogorov
already defined conditional probabilities via conditional expectation,
e.g.\ see \citep{Bobrowski2006/087}.  Given the conditional
expectation operator $\EXP{\cdot|\C{M}_y}$, the conditional
probability is easily recovered as $\D{P}(\C{I}_x|\C{M}_y) =
\EXP{\chi_{\C{I}_x}|\C{M}_y}$, where $\chi_{\C{I}_x}$ is the
characteristic function of the subset $\C{I}_x$.  It may be shown that
this extends the simpler formulation described by \refeq{eq:iII} or
\refeq{eq:iIIa} and is the more fundamental notion, which we examine
next.  Its definition will lead directly to practical computational
procedures.

\subsection{Conditional expectation} \label{SS:cond-expect}
The easiest point of departure for conditional expectation
\citep{Bobrowski2006/087} in our setting is to define it not just for
one piece of measurement $\C{M}_y$---which may not even be possible
unambiguously---but for sub-$\sigma$-algebras $\F{S} \subset \F{A}$ on
$\Omega$.  A sub-$\sigma$-algebra $\F{S}$ is a mathematical
description of a reduced possibility of randomness --- the smallest
sub-$\sigma$-algebra $\{\emptyset, \Omega\}$ allows only the constants
in $\E{X}_c$ --- as it contains fewer events than the full algebra
$\F{A}$.  The connection with a measurement $\C{M}_y$ is to take
$\F{S}:=\sigma(y)$, the $\sigma$-algebra generated by the measurement
$y=\hat{h}(x, \vepsilon v)$ from \refeq{eq:dyn-m}.  These are all
events which are consistent with possible observations of some value
for $y$.  This means that the observation of $y$ allows only a certain
`fineness' of information to be obtained, and this is encoded in the
sub-$\sigma$-algebra $\F{S}$.

\subsubsection{Scalar random variables} \label{SSS:cond-expect-scalar}
For scalar RVs ---functions $r(x)$ of $x$ with finite variance, i.e.\ elements
of $\C{S}:=L_2(\Omega, \F{A}, \D{P})$---the subspace corresponding to
the sub-$\sigma$-algebra $\C{S}_\infty :=L_2(\Omega, \F{S},\D{P})$ is
a closed subspace \citep{Bobrowski2006/087} of the full space $\C{S}$.
One example of such a scalar RV is the function 
\[ r(x) := \chi_{\C{I}_x}(x) = \begin{cases} 1 \; \text{ if } x \in \C{I}_x ,  \\
                                           0 \; \text{ otherwise},
                         \end{cases}   \]
mentioned at the end of \refSS{bayes-laplace-thm} used to define
conditional probability of the subset $\C{I}_x \subseteq \C{X}$ once a
conditional expectation operator is defined: $\D{P}(\C{I}_x|\F{S}) =
\EXP{\chi_{\C{I}_x}|\F{S}}$.
\begin{defi} \label{D:scalar-cond-exp} For scalar functions of $x$ ---
  scalar RVs $r(x)$ --- in $\C{S}$, the conditional expectation
  $\EXP{\cdot | \F{S}}$ is defined as the orthogonal projection onto
  the closed subspace $\C{S}_\infty$, so that $\EXP{r(x) | \F{S}} \in
  \C{S}_\infty$, e.g.\ see \citep{Bobrowski2006/087}.
\end{defi}
The question is now on how to characterise this subspace
$\C{S}_\infty$, in order to make it more accessible for possible
numerical computations.  In this regard, note that the
\emph{Doob-Dynkin} lemma \citep{Bobrowski2006/087} assures us that if a
RV $s(x)$ --- like $\EXP{r(x)|\F{S}}$ --- is in the subspace
$\C{S}_\infty$, then $s(x) = \vphi(y)$ for some $\vphi\in
L_0(\C{Y})$, the space of measurable scalar functions on
$\C{Y}$.  We state this key fact and the resulting
new characterisation of the conditional expectation in
\begin{prop}  \label{prop:Doob-Dynkin}
The subspace $\C{S}_\infty$ is given by
\begin{equation}  \label{eq:iIVxx}
\C{S}_\infty =  \overline{\spn} \{\vphi \; | \; 
     \vphi(\hat{h}(x,\vepsilon v));\; \vphi \in L_0(\C{Y})
      \; \text{ and } \vphi \in \C{S}\}.      
\end{equation}
The conditional expectation of a scalar RV $r(x)\in\C{S}$, being the
orthogonal projection, minimises the distance to the original RV over
the whole subspace:
\begin{equation} \label{eq:iIII}
  \EXP{r(x) |\F{S}} := P_{\C{S}_\infty}(r(x)) := \textup{arg min}_{\tilde{r}\in\C{S}_\infty}
  \; \| r(x) - \tilde{r} \|_{\C{S}},
\end{equation}
where $P_{\C{S}_\infty}$ is the orthogonal projector onto
$\C{S}_\infty$.  The \refeq{eq:iIVxx} and \refeq{eq:iIII} imply the
existence of a optimal map $\phi \in L_0(\C{Y})$ such that
\begin{equation}  \label{eq:opt-condex}
  \EXP{r(x) |\F{S}} = P_{\C{S}_\infty}(r(x)) = \phi(\hat{h}(x,\vepsilon v)).
\end{equation}
In \refeq{eq:iIII}, one may equally well minimise the square of the
distance, the \emph{loss-function}
\begin{equation} \label{eq:iIII-l}
\beta_{r(x)}(\tilde{r})= \frac{1}{2}  \; \| r(x) - \tilde{r} \|^2_{\C{S}}. 
\end{equation}
Taking the vanishing of the first variation / G\^ateaux derivative of
the loss-function \refeq{eq:iIII-l} as a necessary condition for a minimum
leads to a simple geometrical interpretation: the difference
between the original scalar RV $r(x)$ and its projection has to be
perpendicular to the subspace:
\begin{equation}  \label{eq:iIV}
  \forall \tilde{r} \in \C{S}_\infty: \;
  \ip{r(x)  - \EXP{r(x)|\F{S}}}{\tilde{r}}_{\C{S}} = 0, \; 
  \text{ i.e. } r(x)  - \EXP{r(x)|\F{S}} \in \C{S}_\infty^\perp .
\end{equation}
Rephrasing \refeq{eq:iIII} with account to \refeq{eq:iIV} and
\refeq{eq:iIII-l} leads for the optimal map $\phi \in L_0(\C{Y})$ to
\begin{equation} \label{eq:iV-DD}
  \EXP{r(x) |\sigma(y)} =  \phi(\hat{h}(x,\vepsilon v)) :=
  \textup{arg min}_{\vphi\in L_0(\C{Y})} \; \beta_{r(x)}(\vphi(\hat{h}(x,\vepsilon v))),
\end{equation}
and the orthogonality condition of \refeq{eq:iV-DD} which corresponds to
\refeq{eq:iIV} leads to
\begin{equation}  \label{eq:iIV-DDl}
  \forall \vphi \in L_0(\C{Y}): \;
  \ip{r(x)  - \phi(\hat{h}(x,\vepsilon v))}{\vphi(\hat{h}(x,\vepsilon v))}_{\C{S}} = 0 .
\end{equation}
\end{prop}
\begin{proof}
 The \refeq{eq:iIVxx} is a direct statement of the
 Doob-Dynkin lemma \citep{Bobrowski2006/087}, and the \refeq{eq:iIII}
 is equivalent to the definition of the conditional expectation being
 an orthogonal projection in $L_2(\Omega,\F{A},\D{P})$ --- actually an
 elementary fact of Euclidean geometry.

 The existence of the optimal map $\phi$ in \refeq{eq:opt-condex} is a
 consequence of the minimisation of a continuous, coercive, and
 strictly convex function --- the norm \refeq{eq:iIII} --- over the
 closed set $\C{S}_\infty$ in the complete space $\C{S}$.  The
 equivalence of minimising the norm \refeq{eq:iIII} and
 \refeq{eq:iIII-l} is elementary, which is re-stated in \refeq{eq:iV-DD}. 

 The two equivalents statements --- the `Galerkin orthogonality'
 conditions --- \refeq{eq:iIV} and \refeq{eq:iIV-DDl}
 follow not only from requiring the G\^ateaux derivative of
 \refeq{eq:iIII-l} to vanish, but also express an elementary fact of
 Euclidean geometry.
\end{proof}

% Then $q_a := P_{\E{Q}_\infty}(q)$ is called the \emph{updated}, \emph{analysis},
% \emph{assimilated}, or \emph{posterior} value,
% incorporating the new information.  This is the Bayesian update
% expressed in terms of RVs instead of measures.  It is the estimate of
% the unknown parameters $q$ after the measurement has been performed.

The square of the distance $r(x)-\phi(y)$ may be interpreted as a
difference in variance, tying conditional expectation with variance
minimisation; see for example \citep{Papoulis1998/107},
\citep{Bobrowski2006/087}, and the references therein for basic
descriptions of conditional expectation.  See also
\citep{Luenberger1969}.

\subsubsection{Vector valued random variables} \label{SSS:cond-expect-vector}
Now assume that $R(x)$ is a function of $x$ which takes values in a
vector space $\C{R}$, i.e.\ a $\C{R}$-valued RV, where $\C{R}$ is a
Hilbert space.  Two simple examples are given by the conditional mean
where $R(x):=x\in\C{X}$ with $\C{R}=\C{X}$, and by the conditional
variance where one takes $R(x) := (x-\bar{x})\otimes(x-\bar{x}) =
(\tilde{x})\otimes(\tilde{x})$, where $\C{R} = \E{L}(\C{X})$.  The
Hilbert tensor product $\E{R}=\C{R} \otimes \C{S}$ is again needed for
such vector valued RVs, where a bit more formalism is required, as we
later want to take linear combinations of RVs, but with linear
operators as `coefficients' \citep{Luenberger1969}, and this is most
clearly expressed in a component-free fashion in terms of
$L$-invariance, where we essentially follow \citep{bosq2000},
\citep{bosq2007}:
\begin{defi} \label{D:L-inv} Let $\E{V}$ be a subspace of $\E{R}=\C{R}
  \otimes \C{S}$.  The subspace is called \emph{linearly closed},
  $L$-\emph{closed}, or $L$-\emph{invariant}, iff $\E{V}$ is closed,
  and $\forall v \in \E{V}$ and $\forall L \in \E{L}(\C{R})$ it holds
  that $L v \in \E{V}$.
\end{defi}
In finite dimensional spaces one can just apply the notions for the
scalar case in \refSSS{cond-expect-scalar} component by component, but
this is not possible in the infinite dimensional case.  Of course the
vectorial description here collapses to the scalar case upon taking
$\C{R} = \D{R}$.  From \citep{bosq2000} one has the following
\begin{prop} \label{prop:fst-res-lcl} It is obvious that the whole
  space $\E{R}=\C{R} \otimes \C{S}$ is linearly closed, and that for a
  linearly closed subspace $\E{V} \subseteq \E{R}$ its orthogonal
  complement $\E{V}^\perp$ is also linearly closed.  Clearly, for a
  closed subspace $\C{S}_a \subseteq \C{S}$, the tensor space $\C{R}
  \otimes \C{S}_a$ is linearly closed, and hence the space of
  constants $\E{R}_c = \C{R}\otimes\spn\{\chi_\Omega\} \cong \C{R}$ is
  linearly closed, as well as its orthogonal complement $\E{R}_0 =
  \E{R}_c^\perp$, the subspace of zero-mean RVs.
\end{prop}
Let $v \in \E{R}$ be a RV, and denote by 
\begin{equation} \label{eq:sub-sp-gen}
\C{R}_v   := \overline{\spn}\; v(\Omega), \quad \sigma(v) := \{
v^{-1}(B) \; : \; B \in \F{B}_{\C{R}}\}
\end{equation}
the closure of the span of the image of $v$ and the $\sigma$-algebra
generated by $v$, where $\F{B}_{\C{R}}$ is the Borel-$\sigma$-algebra
of $\C{R}$.  Denote the closed subspace generated by $\sigma(v)$ by
$\C{S}_v := L_2(\Omega, \sigma(v), \D{P}) \subseteq \C{S}$.  Let
$\E{R}_{Lv} := \overline{\spn}\; \{ Lv\; : \; L \in \E{L}(\C{R}) \}
\subseteq \E{R}$, the linearly closed subspace generated by $v$, and
finally denote by $\E{R}_{v} := \spn \{v\} \subseteq \E{R}$, the
one-dimensional ray and hence closed subspace generated by $v$.
Obviously it holds that
\[ v \in \E{R}_{v} \subseteq \E{R}_{Lv} \subseteq \C{R}
\otimes \C{S}_{v} \subseteq \E{R}, \; \text{ and } \bar{v}\in\C{R}_v , \]
and $\C{R}\otimes\C{S}_{v}$ is linearly closed according to
Proposition~\ref{prop:fst-res-lcl}.

\begin{defi} \label{D:L-orth}
Let $\E{V}$ and $\E{W}$ be subspaces of $\E{R}$, and $v, w \in \E{R}$ two RVs.
\begin{itemize}
\item  The two subspaces are \emph{weakly orthogonal} or simply just
  \emph{orthogonal},  denoted by

  $\E{V} \perp \E{W}$, iff $\forall v
  \in \E{V}, \forall w \in \E{W}$ it holds that $\ipd{v}{w}_{\E{R}}=0$.

\item  A RV $v \in \E{R}$ is \emph{weakly orthogonal} or simply just
  \emph{orthogonal} to the subspace $\E{W}$, denoted by

  $v \perp \E{W}$, iff $\E{R}_{v} \perp \E{W}$, i.e.\ $\forall w \in
  \E{W}$ it holds that $\ipd{v}{w}_{\E{R}} = 0$.

\item Two RVs $v, w \in \E{R}$ are \emph{weakly orthogonal} or
  as usual simply just \emph{orthogonal}, denoted by

  $v \perp w$, iff $\ipd{v}{w}_{\E{R}} = 0$, i.e.\ $\E{R}_{v} \perp \E{R}_{w}$.

\item  The two subspaces $\E{V}$ and $\E{W}$ are 
  \emph{strongly orthogonal} or $L$-\emph{orthogonal},  iff they are
  linearly closed---\refD{L-inv}---and it holds that
  $\ipd{Lv}{w}_{\E{R}} = 0$, $\forall v \in \E{V}, \forall w \in 
  \E{W}$ and $\forall L \in \E{L}(\C{R})$.  This is denoted by

  $\E{V} \Lperp \E{W}$, and in other words $\E{L}(\C{R})
  \ni C_{vw} = \EXP{v\otimes w}=0$.

\item  The RV $v$ is \emph{strongly orthogonal} to a
  linearly closed subspace $\E{W} \subseteq \E{R}$, denoted by 

  $v \Lperp \E{W}$, iff $\E{R}_{Lv} \Lperp \E{W}$,
  i.e.\ $\forall w \in \E{W}$ it holds that $C_{vw} = 0$.  

\item The two RVs $v, w$ are \emph{strongly orthogonal}
  or simply just \emph{uncorrelated}, denoted by 

  $v \Lperp w$, iff $C_{vw} = 0$, i.e.\
  $\E{R}_{Lv} \Lperp \E{R}_{Lw}$.

\item Let $\F{C}_1, \F{C}_2 \subseteq \F{A}$ be two
  sub-$\sigma$-algebras.  They are \emph{independent}, denoted by

  $\F{C}_1 \stindep \F{C}_2$, iff the closed subspaces of $\C{S}$
  generated by them are orthogonal in $\C{S}$:\\
   $L_2(\Omega,\F{C}_1,\D{P}) \perp  L_2(\Omega,\F{C}_2,\D{P})$.

\item The two subspaces $\E{V}$ and $\E{W}$ are
  \emph{stochastically independent}, denoted by

  $\E{V} \stindep \E{W}$, iff the sub-$\sigma$-algebras generated are:
  $\sigma(\E{V}) \stindep \sigma(\E{W})$.

\item The two RVs $v, w$ are \emph{stochastically independent}, denoted by 

  $v \stindep  w$, iff $\sigma(v) \stindep \sigma(w)$, i.e.\
  $\C{S}_v \perp \C{S}_w$.
\end{itemize}
\end{defi}
\begin{prop} \label{prop:scnd-res-lcl} Obviously $\E{R}_c \Lperp
  \E{R}_0$.  It is equally obvious that for any two closed
  subspaces $\C{S}_a, \C{S}_b \subseteq \C{S}$, the condition
  $\C{S}_a \perp \C{S}_b$ implies
  that the tensor product subspaces are strongly orthogonal:
  \[ \C{R} \otimes \C{S}_a \Lperp \C{R} \otimes \C{S}_b . \]
  This implies that for a closed subspace $\C{S}_s \subseteq \C{S}$
  the subspaces $\E{R}_s=\C{R} \otimes \C{S}_s \subseteq \E{R}$ and
  its orthogonal complement $\E{R}_s^\perp=\C{R} \otimes
  \C{S}_s^\perp$ are linearly closed and strongly orthogonal.
\end{prop}
We note from \citep{bosq2000}, \citep{bosq2007} the following results
which we collect in
\begin{prop}  \label{prop:indep-ortho}
   Let $v, w \in \E{R}_0$ be two zero-mean RVs.  Then
  \[ v \stindep w \Rightarrow v \Lperp w \Rightarrow v \perp w . \]
  Strong orthogonality in general does not imply independence, and
  orthogonality does not imply strong orthogonality, unless $\C{R}$ is
  one-dimensional.

If $\E{S}  \subseteq \E{R}$ is linearly closed, then
  \[ v \perp \E{S} \Rightarrow  v \Lperp \E{S}, \text{ i.e. }
    \E{R}_v \perp \E{S} \Rightarrow  \E{R}_v \Lperp \E{S} \Rightarrow
    \E{R}_{Lv} \Lperp \E{S} . \]
\end{prop}
From this we obtain the following:
\begin{lem} \label{L:ortho_L-ortho} 
  Set $\E{R}_\infty := \C{R} \otimes \C{S}_\infty$ for the
  $\C{R}$-valued RV $R(x)$ with finite variance on the
  sub-$\sigma$-algebra $\F{S}$, representing the new information.
 
  Then $\E{R}_\infty$ is $L$-invariant or strongly closed, and for any
  zero mean RV $v \in \E{R}$:
\begin{equation} \label{eq:perp-Lperp}
 v \in \E{R}^\perp_\infty \Leftrightarrow v \perp \E{R}_\infty
 \Rightarrow v \Lperp \E{R}_\infty .
\end{equation}
In addition, it holds --- even if $v \in \E{R}$ is not zero mean ---
that
\begin{equation} \label{eq:perp-corr-perp}
 v \in \E{R}^\perp_\infty \Leftrightarrow v \perp \E{R}_\infty
 \Rightarrow \forall w \in \E{R}_\infty : \; \EXP{v \otimes w} = 0 .
\end{equation}
\end{lem}
\begin{proof}
  $\E{R}_\infty$ is of the type $\C{R} \otimes \C{S}_\infty$ where
  $\C{S}_\infty$ is a closed subspace, and $\C{R}$ is obviously
  closed.  From the remarks above it follows that $\E{R}_\infty$ is
  $L$-invariant or linearly resp.\ strongly closed.  The
  \refeq{eq:perp-Lperp} is a direct consequence of
  Proposition~\ref{prop:indep-ortho}.

  To prove \refeq{eq:perp-corr-perp}, take any $w \in \E{R}_\infty$
  and any $L\in\E{L}(\C{R})$, then
\[ v \in \E{R}^\perp_\infty \Rightarrow 0=\ipd{v}{w}_{\E{R}}=\ipd{v}{Lw}_{\E{R}}=
\EXP{\ip{v}{Lw}_{\C{R}}} . \]
  Now, for any $r_1, r_2 \in \C{R}$, take the mapping $L: r_*
  \mapsto \ip{r_2}{r_*}_{\C{R}}\,r_1$, yielding
\begin{multline*}
  0=\EXP{\ip{v}{Lw}_{\C{R}}} = \EXP{\ip{v}{\ip{r_2}{w}_{\C{R}}\,r_1}_{\C{R}}} =\\
\EXP{\ip{v}{r_1}_{\C{R}}\ip{r_2}{w}_{\C{R}}} = 
\ip{r_1}{\EXP{v\otimes w}\,r_2}_{\C{R}} \Leftrightarrow \EXP{v\otimes w}\equiv 0. 
\end{multline*}
\end{proof}

Extending the scalar case described in \refSSS{cond-expect-scalar},
instead of 
\[ \C{S} = L_2(\Omega,\D{P},\F{A}) =
L_2(\Omega,\D{P},\F{A};\D{R}) \cong \D{R} \otimes
L_2(\Omega,\D{P},\F{A}) = \D{R} \otimes \C{S} \]
and its subspace generated by the measurement
\[ \C{S}_\infty = L_2(\Omega,\D{P},\F{S}) =
L_2(\Omega,\D{P},\F{S};\D{R}) \cong \D{R} \otimes
L_2(\Omega,\D{P},\F{S}) = \D{R} \otimes \C{S}_\infty \]
one now considers the space \refeq{eq:space-def1} and its subspace \refeq{eq:space-def2}
\begin{eqnarray} \label{eq:space-def1}
L_2(\Omega,\D{P},\F{A};\C{R}) &\cong& \C{R} \otimes
L_2(\Omega,\D{P},\F{A})  = \C{R} \otimes \C{S} := \E{R} \; \text{ and} \\
\label{eq:space-def2}
L_2(\Omega,\D{P},\F{S};\C{R}) &\cong& \C{R} \otimes
L_2(\Omega,\D{P},\F{S})  = \C{R} \otimes \C{S}_\infty  := \E{R}_\infty
\subseteq \E{R}. 
\end{eqnarray}
The conditional expectation in the vector-valued case is defined
completely analogous to the scalar case, see \refD{scalar-cond-exp}:
\begin{defi} \label{D:vector-cond-exp} 
  For $\C{R}$-valued functions of $x$ --- vectorial RVs $R(x)$ --- in
  the Hilbert-space $\E{R}$ \refeq{eq:space-def1}, the conditional
  expectation $\EXP{\cdot | \F{S}}: \E{R}\to\E{R}$ is defined as the
  orthogonal projection onto the closed subspace $\E{R}_\infty$
  \refeq{eq:space-def2}, denoted by $P_{\E{R}_\infty}$, so that
  $\EXP{R(x) | \F{S}} = P_{\E{R}_\infty}(R(x)) \in \E{R}_\infty$, e.g.\
  see \citep{bosq2000}, \citep{Bobrowski2006/087}.
\end{defi}
From this one may derive a characterisation of the conditional
expectation similar to Proposition~\ref{prop:Doob-Dynkin}.
\begin{thm}  \label{T:Doob-Dynkin-v}
The subspace $\E{R}_\infty$ is given by
\begin{equation}  \label{eq:iIVxx-v}
\E{R}_\infty =   \{\vphi \; | \; 
     \vphi(\hat{h}(x,\vepsilon v))\in \E{R};\; \vphi \in L_0(\C{Y},\C{R})\}.      
\end{equation}
The conditional expectation of a vector-valued RV $R(x)\in\E{R}$,
being the orthogonal projection, minimises the distance to the
original RV over the whole subspace:
\begin{equation} \label{eq:iIII-v}
  \EXP{R(x) |\F{S}} := P_{\E{R}_\infty}(R(x)) := \textup{arg min}_{\tilde{R}\in\E{R}_\infty}
  \; \| R(x) - \tilde{R} \|_{\E{R}},
\end{equation}
where $P_{\E{R}_\infty}$ is the orthogonal projector onto
$\E{R}_\infty$.  The \refeq{eq:iIVxx-v} and \refeq{eq:iIII-v} imply the
existence of a optimal map $\Phi \in L_0(\C{Y},\C{R})$ such that
\begin{equation}  \label{eq:opt-condex-v}
  \EXP{R(x) |\F{S}} = P_{\E{R}_\infty}(R(x)) = \Phi(\hat{h}(x,\vepsilon v)).
\end{equation}
In \refeq{eq:iIII-v}, one may equally well minimise the square of the
distance, the \emph{loss-function}
\begin{equation} \label{eq:iIII-l-v}
\beta_{R(x)}(\tilde{R})= \frac{1}{2}  \; \| R(x) - \tilde{R} \|^2_{\E{R}}. 
\end{equation}
Taking the vanishing of the first variation / G\^ateaux derivative of
the loss-function \refeq{eq:iIII-l-v} as a necessary condition for a
minimum leads to a simple geometrical interpretation: the difference
between the original vector-valued RV $R(x)$ and its projection has to
be perpendicular to the subspace $\E{R}_\infty$: 
$ \forall \tilde{R} \in \E{R}_\infty:$
\begin{equation}  \label{eq:iIV-v}
  \ipd{R(x)  - \EXP{R(x)|\F{S}}}{\tilde{R}}_{\E{R}} = 0, \; 
  \text{ i.e. } R(x)  - \EXP{R(x)|\F{S}} \in \E{R}_\infty^\perp .
\end{equation}
Rephrasing \refeq{eq:iIII-v} with account to \refeq{eq:iIV-v} and
\refeq{eq:iIII-l-v} leads for the optimal map $\Phi \in L_0(\C{Y},\C{R})$ to
\begin{equation} \label{eq:iV-DD-v}
  \EXP{R(x) |\sigma(y)} =  \Phi(\hat{h}(x,\vepsilon v)) :=
  \textup{arg min}_{\vphi\in L_0(\C{Y},\C{R})} \; \beta_{R(x)}(\vphi(\hat{h}(x,\vepsilon v))),
\end{equation}
and the orthogonality condition of \refeq{eq:iV-DD-v} which corresponds to
\refeq{eq:iIV-v} leads to
\begin{equation}  \label{eq:iIV-DDl-v}
  \forall \vphi \in L_0(\C{Y},\C{R}): \;
  \ipd{R(x)  - \Phi(\hat{h}(x,\vepsilon v))}{\vphi(\hat{h}(x,\vepsilon v))}_{\E{R}} = 0 .
\end{equation}
In addition, as $\E{R}_\infty$ is linearly closed, one obtains the
useful statement 
\begin{equation}  \label{eq:iIV-L-t}
  \forall \tilde{R} \in \E{R}_\infty: \;
  \E{L}(\C{R}) \ni \EXP{(R(x)  - \EXP{R(x)|\F{S})}\otimes\tilde{R}} = 0.
\end{equation}
or rephrased $\forall \vphi \in L_0(\C{Y},\C{R})$:
\begin{equation}  \label{eq:iIV-DDl-t}
    \E{L}(\C{R}) \ni \EXP{(R(x)  - \Phi(\hat{h}(x,\vepsilon
    v)))\otimes\vphi(\hat{h}(x,\vepsilon v))} = 0 .
\end{equation}
\end{thm}
\begin{proof}
  The \refeq{eq:iIVxx-v} is just a version of the Doob-Dynkin lemma
  again \citep{Bobrowski2006/087}, this time for vector-valued
  functions.  The \refeq{eq:iIII-v}, \refeq{eq:opt-condex-v},
  \refeq{eq:iIII-l-v}, \refeq{eq:iIV-v}, \refeq{eq:iV-DD-v}, and
  \refeq{eq:iIV-DDl-v} follow just as in the scalar case of
  Proposition~\ref{prop:Doob-Dynkin}.
  
  As $\E{R}_\infty$ is linearly closed according to
  \refL{ortho_L-ortho}, the \refeq{eq:perp-Lperp} causes
  \refeq{eq:iIV-v} to imply \refeq{eq:iIV-L-t}, and
  \refeq{eq:iIV-DDl-v} together with \refeq{eq:perp-corr-perp} from
  Lemma~\ref{L:ortho_L-ortho} to imply \refeq{eq:iIV-DDl-t}.
\end{proof}

Already in \citep{Kalman} it was noted that the conditional
expectation is the best estimate not only for the \emph{loss function}
`distance squared', as in \refeq{eq:iIII-l} and \refeq{eq:iIII-l-v},
but for a much larger class of loss functions under certain
distributional constraints.  However for the quadratic loss function this
is valid without any restrictions.

Requiring the derivative of the quadratic loss function in
\refeq{eq:iIII-l} and \refeq{eq:iIII-l-v} to vanish may also be
characterised by the \emph{Lax-Milgram} lemma, as one is minimising a
quadratic functional over the vector space $\E{R}_\infty$, which is
closed and hence a \emph{Hilbert} space.  For later reference, this
result is recollected in
\begin{thm} \label{T:cond-expect-orthog} 
  In the scalar case, there is a unique minimiser $\EXP{r(x) |\F{S}} =
  P_{\C{S}_\infty}(r(x)) \in \C{S}_\infty$ to the problem in
  \refeq{eq:iIII}, and it is characterised by the \emph{orthogonality}
  condition \refeq{eq:iIV}
\begin{equation}  \label{eq:iIV-L}
  \forall \tilde{r} \in \C{S}_\infty: \quad
  \ip{r(x)  - \EXP{r(x)|\F{S}}}{\tilde{r}}_{\C{S}} = 0.
\end{equation}
The minimiser is unique as an element of $\C{S}_\infty$, but the
mapping $\phi \in L_0(\C{Y})$ in \refeq{eq:iV-DD} may not necessarily
be.  It also holds that
\begin{equation}  \label{eq:iIV-P}
\|P_{\C{S}_\infty}(r(x))\|_{\C{S}}^2 = \|r(x)\|_{\C{S}}^2 -  
             \|r(x) - P_{\C{S}_\infty}(r(x))\|_{\C{S}}^2 .
\end{equation}
As in the scalar case, in the vector-valued case there is a unique
minimiser $\EXP{R(x) |\F{S}} = P_{\E{R}_\infty}(R(x)) \in
\E{R}_\infty$ to the problem in \refeq{eq:iIII-v}, which satisfies the
\emph{orthogonality} condition \refeq{eq:iIV-v}
\begin{equation}  \label{eq:iIV-L-v}
  \forall \tilde{R} \in \E{R}_\infty: \quad
  \ipd{R(x)  - \EXP{R(x)|\F{S}}}{\tilde{R}}_{\E{R}} = 0,
\end{equation}
which is equivalent to the the \emph{strong orthogonality condition}
\refeq{eq:iIV-L-t}
\begin{equation}  \label{eq:iIV-L-v-s}
  \forall \tilde{R} \in \E{R}_\infty: \quad
  \EXP{R(x)  - \EXP{R(x)|\F{S}}\otimes\tilde{R}} = 0.
\end{equation}
The minimiser is unique as an element of $\E{R}_\infty$, but the
mapping $\Phi \in L_0(\C{Y},\C{R})$ in \refeq{eq:iV-DD-v} may not
necessarily be.  It also holds that
\begin{equation}  \label{eq:iIV-P-v}
\|P_{\E{R}_\infty}(R(x))\|_{\E{R}}^2 = \|R(x)\|_{\E{R}}^2 - 
             \|R(x) - P_{\E{R}_\infty}(R(x))\|_{\E{R}}^2 .
\end{equation}
\end{thm}
\begin{proof}
  It is all already contained in Proposition~\ref{prop:Doob-Dynkin}
  resp.\ \refT{Doob-Dynkin-v}.  Except for \refeq{eq:iIV-L-v-s}, this
  is just a re-phrasing of the \emph{Lax-Milgram} lemma, as the
  bi-linear functional---in this case the inner product---is naturally
  coercive and continuous on the subspace $\E{R}_\infty$, which is
  closed and hence a \emph{Hilbert} space.  The only novelty here are
  the \refeq{eq:iIV-P} and \refeq{eq:iIV-P-v} which follow from
  \emph{Pythagoras's} theorem.
\end{proof}

\section{Characterising the posterior} \label{S:char-rv}
The information contained in the Bayesian update is encoded in the
conditional expectation.  And it only characterises the distribution
of the posterior.  A few different ways of characterising the
distribution via the conditional expectation are sketched in
\refSS{char-dist}.  But in many situations, notably in the setting of
\refeq{eq:dyn} or \refeq{eq:dyn-l}, with the observations according to
\refeq{eq:dyn-m} or \refeq{eq:dyn-ml}, we want to construct a new RV
$z \in \E{X}$ to serve as an approximation to the solution of
\refeq{eq:dyn} or \refeq{eq:dyn-l}.  This then is a \emph{filter}, and
a few possibilities will be given in \refSS{char-p_rv-filt}.

\subsection{The posterior distribution measure} \label{SS:char-dist}
It was already mentioned at the beginning of
\refSSS{cond-expect-scalar}, that the scalar function $r_{\C{I}_x}(x)
= \chi_{\C{I}_x}(x)$ may be used to characterise the conditional
probability distribution of $x\in \E{X}$.  Indeed, if for a RV $R(x)
\in \E{R}$ one defines:
\begin{equation}  \label{eq:cond-prob-E}
\forall \C{E}\in\F{B}_{\C{R}}:\; \D{P}(\C{E}|\F{S}) := \EXP{\chi_{\C{E}}(R)|\F{S}},
\end{equation}
one has completely characterised the posterior distribution, a version
of which is under certain conditions---\citep{Bobrowski2006/087},
\citep{rao2005}, \citep{jaynes03}---a measure on $\C{R}$, the image
space of the RV $R$.

One may also recall that the \emph{characteristic function} in the
sense of stochastics of a RV $R \in \E{R}$, namely
\[ \vphi_R:\; \C{R}^* \ni r^* \mapsto \vphi_R(r^*) := \EXP{\exp(\ii \,
  \ip{r^*}{R}_{\C{R}})}, \] 
completely characterises the distribution of the RV $R$.  As we assume
that $\C{R}$ is a Hilbert space, we may identify $\C{R}$ with its dual
space $\C{R}^*$, and in this case take $\vphi_R$ as defined on
$\C{R}$.  If now a conditional expectation operator $\EXP{\cdot|\F{S}}$
is given, it may be used to define the \emph{conditional}
characteristic function $\vphi_{R|\F{S}}$:
\begin{equation}  \label{eq:cond-char-fct}
\forall r\in\C{R}:\; \vphi_{R|\F{S}}(r) := \EXP{\exp(\ii \,
  \ip{r}{R}_{\C{R}})|\F{S}}.
\end{equation}
This again completely characterises the posterior distribution.

Another possible way, actually encompassing the previous two, is to
look at all functions $\psi:\C{R} \to \D{R}$, and compute---when they
are defined and finite---the quantities
\begin{equation}  \label{eq:cond-all-fct}
 \mu_\psi := \EXP{\psi(R)|\F{S}},
\end{equation}
again completely characterising the posterior distribution.
The two previous examples show that not \emph{all} functions of $R$
with finite conditional expectation are needed.  The first example
uses the set of functions
\[ \{ \psi \; | \; \psi(R) = \chi_{\C{E}}(R), \C{E} \in \F{B}_{\C{R}}\}, \]
whereas the second example uses the set
\[ \{ \psi \; | \; \psi(R) = \exp(\ii \, \ip{r}{R}_{\C{R}}), 
  r \in \C{R} \}. \]

\subsection{A posterior random variable --- filtering} \label{SS:char-p_rv-filt}
In the context of a situation like in \refeq{eq:dyn} resp.\
\refeq{eq:dyn-l}, which represents the \emph{unknown} system and state
vector $x_n$, and where one observes $y_n$ according to
\refeq{eq:dyn-m} resp.\ \refeq{eq:dyn-ml}, one wants to have an
estimating or \emph{tracking} model system, with a state estimate
$z_n$ for $x_n$ which would in principle obey \refeq{eq:dyn} resp.\
\refeq{eq:dyn-l} with the noise $w_n$ set to zero---as one only knows
the structure of the system as given by the maps $\hat{f}$ resp.\ $f$
but not the initial condition $x_0$ nor the noise.  The observations
$y_n$ can be used to correct the state estimate $z_n$, as will be
shown shortly.  The state estimate will be computed via Bayesian
updating.  But the Bayesian theory, as explained above, only
characterises the posterior \emph{distribution}; and there are many
random variables which might have a given distribution.  To obtain a
RV $z_n$ which can be used to predict the next state $x_{n+1}$ through
the estimate $z_{n+1}$ one may use a filter based on Bayesian theory.
The mean vehicle for this will be the notion of conditional
expectation as described in the previous \refS{bayes}.  As we will
first consider only one update step, the time index $n$ will be
dropped for the sake of ease of notation:  The true state is $x$, its
\emph{forecast} is $x_f$, and the forecast of the measurement is
$y_f(x_f)$, whereas the observation is $\hat{y}$.

To recall, according to \refD{vector-cond-exp}, the Bayesian update is
defined via the conditional expectation $\EXP{R(x) |\sigma(y(x))}$
through a measurement $y(x)$---which will for the sake of simplicity
be denoted just by $\EXP{R(x) |y}$---of a $\C{R}$-valued RV $R(x)$ is
simply the orthogonal projection onto the subspace $\E{R}_\infty$ in
\refeq{eq:iIVxx-v},
\[ \EXP{R(x) |y} = P_{\E{R}_\infty}(R(x)) = \Phi_R(y(x)), \]
which is given by the optimal map $\Phi_R$ from \refeq{eq:opt-condex-v},
characterised by \refeq{eq:iIV-DDl-t}, where we have added an index
$R$ to signify that this is the optimal map for the conditional
expectation of the RV $R \in \E{R}$.

The linearly closed subspace $\E{R}_\infty$ induces a orthogonal
decomposition decomposition
\[\E{R} = \E{R}_\infty \oplus \E{R}_\infty^\perp,\]
where the orthogonal projection onto $\E{R}_\infty^\perp$ is given by
$I-P_{\E{R}_\infty}$.  Hence a RV in $\E{R}$ like $R(x)$ can
be decomposed accordingly as
\begin{multline}  \label{eq:orth-decomp-P}
  R(x) = P_{\E{R}_\infty}(R(x)) +  \left(I-P_{\E{R}_\infty}\right)(R(x)) =\\
  \Phi_R(y(x)) + (R(x) - \Phi_R(y(x))).
\end{multline}
This \refeq{eq:orth-decomp-P} is the starting point for the updating.
A measurement $\hat{y}$ will inform us about the component in
$\E{R}_\infty$, namely $\Phi_R(\hat{y})$, while we leave the component
orthogonal to it unchanged: $R(x_f) - \Phi_R(y(x_f))$.  Adding these two
terms then gives an \emph{updated} or \emph{assimilated} RV $R_a \in \E{R}$:
\begin{multline}  \label{eq:orth-upd}
  R_a = \Phi_R(\hat{y}) + (R(x_f) - \Phi_R(y(x_f)))  
      = \bar{R}_a^{|\hat{y}} + \tilde{R}_a = \\ 
  R(x_f) + (\Phi_R(\hat{y})-\Phi_R(y(x_f))) = R_f + R_\infty ,
\end{multline}
where $R_f = R(x_f) \in \E{R}$ is the \emph{forecast} and $R_\infty =
(\Phi_R(\hat{y})-\Phi_R(y(x_f))) \in \E{R}$ is the \emph{innovation}.
For $\bar{R}_a^{|\hat{y}} = \Phi_R(\hat{y})$  and $\tilde{R}_a =
R(x_f) - \Phi_R(y(x_f))$ one has the following result:
\begin{prop} \label{prop:corr-CE}
The assimilated RV $R_a$ from \refeq{eq:orth-upd} has the correct
conditional expectation
\begin{equation}  \label{eq:corr-CE}
  \EXP{R_a |y} = \Phi_R(\hat{y}) + \EXP{(R(x_f) - \Phi_R(y(x_f)))|y} = 
   \EXP{R(x_f) |\hat{y}} = \bar{R}_a^{|\hat{y}},
\end{equation}
better would be \emph{posterior} expectation---after the observation $\hat{y}$.
\end{prop}
\begin{proof}
  Observe that that the conditional expectation of the second term
  $\tilde{R}_a$ in \refeq{eq:orth-upd} vanishes:
\begin{multline*} 
  \EXP{\tilde{R}_a|y} = \EXP{(R(x) - \Phi_R(y(x)))|y} = \\
                P_{\E{R}_\infty}(R(x) - P_{\E{R}_\infty}(R(x))) =
                P_{\E{R}_\infty}(R(x)) - P_{\E{R}_\infty}(R(x)) = 0.
\end{multline*}
This means that the conditional expectation of the second term in
\refeq{eq:corr-CE} is nought, whereas the remaining term
$\Phi_R(\hat{y})$ is just $\EXP{R(x_f) |\hat{y}}$.
\end{proof}
From \refeq{eq:orth-upd} one can now construct filters.  As often the
optimal map $\Phi_R$ is often not easy to compute, one may even want
to replace it by an approximation, say $g_R$, so that the update
equation is
\begin{equation}  \label{eq:orth-upd-g}
    \tilde{R}_a = R(x_f) + (g_R(\hat{y})-g_R(y(x_f))).
\end{equation}
Whichever way, either the \refeq{eq:orth-upd} or
\refeq{eq:orth-upd-g}, they are composed of the following elements,
the prior knowledge, which gives the prediction or forecast
$R_f=R(x_f)$ for the RV $R$, and the correction, innovation, or update
\[ R_\infty = (\Phi_R(\hat{y})-\Phi_R(y(x_f))) \approx (g_R(\hat{y})-g_R(y(x_f))),\]
which is the update difference between the actual observation
$\hat{y}$ and the predicted or forecast observation $y(x_f)$.

\subsubsection{Getting the mean right} \label{SSS:filt-mean}
The simplest function $R(x)$ to think of is the identity $R(x):=x$.
This gives an update---a filter---for the RV $x$ itself.  The optimal
map will be denoted by $\Phi_x$ in this case.  From
\refeq{eq:orth-upd} one has:
\begin{equation}  \label{eq:orth-upd-x}
  x_a = x_f + (\Phi_x(\hat{y})-\Phi_x(y(x_f))) = x_f + x_\infty ,
\end{equation}
and Proposition~\ref{prop:corr-CE} ensures that the assimilated RV
$x_a$ has the correct conditional mean
\begin{equation}  \label{eq:corr-mean-x}
\EXP{x_a|y} = \EXP{x_f |\hat{y}} = \Phi_x(\hat{y}) =: \bar{x}^{|\hat{y}}.
\end{equation}

The \refeq{eq:orth-upd-x} is the basis for many filtering algorithms,
and many variations on the basic prescription are possible.  Often
they will be such that the property according to
Proposition~\ref{prop:corr-CE}, the correct conditional mean, is only
approximately satisfied.  This is due to the fact that for one the
\refeq{eq:orth-upd-x} is an equation for RVs, which in their entirety
can not be easily handled, they are typically infinite dimensional
objects and thus have to be discretised for numerical purposes. 

It was also already pointed out that the optimal map $\Phi_x$ is not
easy to compute, and thus approximations are used, $\Phi_x \approx
g_x$, the simplest one being where $g_x = G_x \in \E{L}(\C{Y},\C{X})$
is taken as a linear map, leading to linear filters \citep{bosq2000},
\citep{Goldstein2007}.  The well-known Kalman filter (KF)
\citep{Kalman} and its many variants and extensions --- e.g.\ extended
KF, Gauss-Markov-Kalman filter, square root KF, etc. --- and
simplifications --- e.g.\ 3DVar, 4DVar, Kriging, Gaussian process
emulation (GPE) --- arise in this way (e.g.\ \citep{Papoulis1998/107},
\citep{Tarantola2004}, \citep{Evensen2009}, \citep{saadGhn:2009},
\citep{Blanchard2010a}, \citep{BvrAkJsOpHgm11},
\citep{bvrAlOpHgm12-a}, \citep{opBvrAlHgm12}, \citep{OpBrHgm12}).

As the conditional expectation of $x_a$ in \refeq{eq:orth-upd-x} is
\refeq{eq:corr-mean-x} $\EXP{x_a|\hat{y}} = \Phi_x(\hat{y}) =
\bar{x}^{|\hat{y}}$, the zero-mean part of $x_a$ is $\tilde{x}_a=x_f -
\Phi_x(y(x_f))$.  The posterior variance of the RV $x_a$ is thus
\begin{equation} \label{eq:orth-upd-x-var}
  C_{x_a x_a|\hat{y}} =  \EXP{\tilde{x}_a \otimes \tilde{x}_a|\hat{y}} 
  = \EXP{(x_f - \Phi_x(y(x_f))) \otimes (x_f - \Phi_x(y(x_f))) | \hat{y}},
\end{equation}
and it has been noted many times that this does not depend
on the observation $\hat{y}$.  Still, one may note (e.g.\
\citep{Tarantola2004}) 
\begin{prop} \label{prop:corr-var-Gauss} 
  Assume that $x_f$ is a Gaussian RV, that the observation
  $y(x_f)=\hat{h}(x_f, v)$ --- absorbing the scaling $\vepsilon$ into
  $v$ --- is affine in $x_f$ and in the uncorrelated Gaussian
  observational noise $v$, i.e.\ $v \Lperp x_f$ and $C_{v x_f} =0$.
  Then the optimal map $\Phi_x = K_x \in \E{L}(\C{Y},\C{X})$ is
  linear, and the updated or assimilated RV $x_a$ from
  \refeq{eq:orth-upd-x} is also Gaussian, and has the \emph{correct}
  posterior distribution, characterised by the mean
  \refeq{eq:corr-mean-x}, $\bar{x}^{|\hat{y}}$, and the covariance
  \refeq{eq:orth-upd-x-var}.  Setting $w = \hat{h}(x_f, v) := H(x_f) +
  v$ with $H \in \E{L}(\C{X},\C{Y})$, one obtains from
  \refeq{eq:orth-upd-x-var}
  \begin{multline} \label{eq:Gauss-cov} C_{x_a x_a|\hat{y}} =
    \EXP{\tilde{x}_a \otimes \tilde{x}_a|\hat{y}} = C_{x_f x_f} - K_x
    C_{w x_f} - C_{w x_f}^T K_x^T + K_x C_{ww} K_x^T = \\
    (I - K_xH) C_{x_f x_f} (I-K_xH)^T + K_x C_{vv} K_x^T
  \end{multline}
  for the covariance, and for the mean
  \begin{equation}    \label{eq:Gauss-mean}
    \bar{x}^{|\hat{y}} = \EXP{x_a|\hat{y}} = K_x \hat{y}.
  \end{equation}
\end{prop}
\begin{proof}
  As this is a well known result, we only show the connection of
  \refeq{eq:Gauss-cov} with \refeq{eq:orth-upd-x-var}.  Note that
  \begin{multline*}
    \tilde{x}_a = x_f -  \Phi_x(y(x_f)) = x_f - K_x(\hat{h}(x_f, v))\\
        = x_f -  K_x(w) = (I-K_x H)x_f - K_x v.
  \end{multline*}
  This gives the \refeq{eq:Gauss-cov}, and the \refeq{eq:Gauss-mean}
  follows directly from \refeq{eq:corr-mean-x}.
\end{proof}
This means that in the purely linear Gaussian case described in
Proposition~\ref{prop:corr-var-Gauss} a RV with the correct posterior
distribution is given simply by the process of projection.

In the context of the dynamical system \refeq{eq:dyn} resp.\
\refeq{eq:dyn-l}, where the measurement is denoted by $y_{n+1}$, the
update for the tracking equation is
\begin{eqnarray}  \label{eq:dyn-tr1}
  z_{n+1} &=& \hat{f}(z_n,0,n) + (\Phi_x(y_{n+1}) - \Phi_x(\hat{h}(\hat{f}(z_n,0,n),0))) \\
   & & \quad \text{for the case \refeq{eq:dyn}, resp.\ for \refeq{eq:dyn-l}}
   \notag \\ 
  \label{eq:dyn-tr2}
  z_{n+1} &=& f(z_n) + (\Phi_x(y_{n+1}) - \Phi_x(h(f(z_n)))).
\end{eqnarray}
Again the assimilated or updated state estimate $x_a:=z_{n+1}$ is composed
of two components, the prediction or forecast $x_f := \hat{f}(z_n,0,n)$
resp.\ $x_f := f(z_n)$, and the correction, innovation, or update
\[ x_\infty := (\Phi_x(y_{n+1}) -\Phi_x(\hat{h}(\hat{f}(z_n,0,n),0))) \] 
resp.\ $ x_\infty := (\Phi_x(y_{n+1}) - \Phi_x(h(f(z_n))))$, which takes into
account the difference resulting from the actual measurement
$\hat{y}:= y_{n+1}$ and the forecast measurement
$\hat{h}(\hat{f}(z_n,0,n),0)$ resp. $h(f(z_n))$. 

If the optimal map $\Phi_x$ is not easy to compute, one may want to
replace it by an approximation as in \refeq{eq:orth-upd-g}, say $g$,
so that for example the \refeq{eq:dyn-tr2} would read
\begin{equation}  \label{eq:dyn-tr3}
  z_{n+1} =  f(z_n) + (g(y_{n+1}) - g(h(f(z_n)))) = (f - g\circ h
  \circ f)(z_n) + g(y_{n+1}) ,
\end{equation}
where one may hope to show that if the map $(f - g\circ h \circ f)$ is
a contraction, the difference $x_n - z_n$ will decrease as
$n\to\infty$ \citep{LawLitHGM15}.  Many variations on this theme
exist, especially in the case where both the observation map $h$ and
the update operator $g$ are linear \citep{Sanz-Alonso_Stuart:2014},
\citep{Tarn-Rasis:1976}.

\subsubsection{Getting also the covariance right} \label{SSS:filt-var}
An approximation to a RV which has the required posterior distribution
was constructed in \refSSS{filt-mean}, where at least the mean was
correct.  One may now go a step further and also get the correct
posterior covariance.  As a starting point take the \emph{assimilated}
RV $x_a$ from \refeq{eq:orth-upd-x} that has the correct conditional
mean $\bar{x}^{|\hat{y}}$ from \refeq{eq:corr-mean-x}, but the
covariance, from \refeq{eq:orth-upd-x-var}, is $C_{x_a x_a|\hat{y}} =
\EXP{\tilde{x}_a \otimes \tilde{x}_a|\hat{y}}$.  To get the covariance
and the mean right, we compute what the correct posterior covariance
should be, by computing the optimal map for $R(x) := x \otimes x$.
This gives for the posterior correlation
\begin{equation}  \label{eq:corr-corr}
 \hat{C}_p := \EXP{R(x_f)|\hat{y}} = \EXP{(x_f \otimes x_f)|\hat{y}} =
 \Phi_{x\otimes x}(\hat{y}),
\end{equation}
so that the posterior covariance is
\begin{equation}  \label{eq:upd-cov}
 C_p := \hat{C}_p - \bar{x}^{|\hat{y}} \otimes \bar{x}^{|\hat{y}} =
 \Phi_{x\otimes x}(\hat{y}) - \Phi_x(\hat{y}) \otimes \Phi_x(\hat{y}).
\end{equation}
\begin{prop} \label{prop:upd-mean-cov}
A new RV $x_c$ with the correct posterior covariance \refeq{eq:upd-cov} is
built from $x_a = \bar{x}^{|\hat{y}} + \tilde{x}_a$ in
\refeq{eq:orth-upd-x} by taking
\begin{equation}  \label{eq:upd-mean-cov-x}
  x_c := \bar{x}^{|\hat{y}} + C_p^{1/2} C_{x_a x_a|\hat{y}}^{-1/2} \tilde{x}_a .
\end{equation}
\end{prop}
\begin{proof}
  As $\EXP{x_c|\hat{y}} = \bar{x}^{|\hat{y}}$, one has
  \begin{multline*}
   C_{x_c x_c} =  \EXP{(C_p^{1/2}\,C_{x_a x_a|\hat{y}}^{-1/2}\,\tilde{x}_a) 
   \otimes (C_p^{1/2}\,C_{x_a x_a|\hat{y}}^{-1/2}\,\tilde{x}_a)|\hat{y}}=\\
   C_p^{1/2}\,C_{x_a x_a|\hat{y}}^{-1/2}\,\EXP{\tilde{x}_a\otimes\tilde{x}_a|\hat{y}}
   \,C_{x_a x_a|\hat{y}}^{-1/2}\,C_p^{1/2} = \\ C_p^{1/2}\,C_{x_a x_a|\hat{y}}^{-1/2}\,  
   C_{x_a x_a|\hat{y}}\,C_{x_a x_a|\hat{y}}^{-1/2}\,C_p^{1/2} = 
   C_p^{1/2}\,C_p^{1/2} = C_p,  
  \end{multline*}
  proving that the RV $x_c$ in \refeq{eq:upd-mean-cov-x} has the
  correct posterior covariance.
\end{proof}

Having achieved a RV which has the correct posterior mean and
covariance, it is conceivable to continue in this fashion, building RVs
which match the posterior better and better.  A similar idea, but from
a different starting point, is used in \citep{moselhyYMarz:2011} and
\citep{ParnoTMYM:2015-arxiv}.  In the future, it is planned to combine
these two approaches.

\subsection{Approximations} \label{SS:char-apprx}
In any actual inverse computations several kinds of approximations are
usually necessary.  Should one pursue an approach of sampling form the
posterior distribution \refeq{eq:iIIa} in \refSS{bayes-laplace-thm}
for example, then typically a sampling and a quantisation or binning
approximation is performed, often together with some kernel-estimate
of the density.  All of these processes usually introduce
approximation errors.  Here we want to use methods based on the conditional
expectation, which were detailed in \refSS{cond-expect}.

Looking for example at \refT{Doob-Dynkin-v} and
\refT{cond-expect-orthog} in \refSSS{cond-expect-vector}, one has to
work in the usually infinite dimensional space $\E{R} = \C{R} \otimes
\C{S}$ from \refeq{eq:space-def1} and its subspace $\E{R}_\infty =
\C{R} \otimes \C{S}_\infty$ from \refeq{eq:space-def2}, to minimise
the functional in \refeq{eq:iIII-l-v} to find the optimal map
representing the conditional expectation for a desired function
$R(x)$, \refeq{eq:opt-condex-v} and \refeq{eq:iV-DD-v}, $\Phi$ in the
space $L_0(\C{Y},\C{R})$.  Then one has to construct a RV whose
distribution my be characterised by the conditional expectation, to
represent the posterior measure.  Approximations in this latter
respect were discussed in \refSS{char-p_rv-filt}.  The space
$\E{R}_\infty$ is computationally accessible via $L_0(\C{Y},\C{R})$,
which has to be approximated by some finite dimensional subspace.
This will be discussed in this \refSS{char-apprx}.  Furthermore, the
component spaces of $\E{R} = \C{R} \otimes \C{S}$ are also typically
infinite dimensional, and have in actual computations to be replaced
by finite dimensional subspaces.  This topic will be sketched in
\refS{num-real}.

Computationally we will not be able to deal with the \emph{whole}
space $\E{R}_\infty$, so we look at the effect of approximations.
Assume that $L_0(\C{Y},\C{R})$ in \refeq{eq:iV-DD-v} or
\refeq{eq:iIV-DDl-v} is approximated by subspaces $L_{0,m} \subset
L_0(\C{Y},\C{R})$ with $\E{L}(\C{Y},\C{R}) \subseteq L_{0,m}$, where
$m\in\D{N}$ is a parameter describing the level of approximation and
$L_{0,m} \subset L_{0,k}$ if $m < k$, such that the subspaces
\begin{equation}  \label{eq:iIVxxN}
\E{R}_m =   \{\vphi(y) \; | \; \vphi \in L_{0,m}; \;
\vphi(\hat{h}(x,\vepsilon v)) \in \E{R}\}  \subseteq \E{R}_\infty   
\end{equation}
are linearly closed and their union is dense 
\begin{equation}  \label{eq:clos-L0}
\overline{\bigcup_m \E{R}_m} = \E{R}_\infty,   
\end{equation}
a consistency condition.

To obtain results for the situation where the projection
$P_{\E{R}_\infty}$ is replaced by the orthogonal projection
$P_{\E{R}_m}$ of $\E{R}$ onto $\E{R}_m$, all that is necessary is to
reformulate the \refT{Doob-Dynkin-v} and \refT{cond-expect-orthog}.
\begin{thm} \label{T:cond-expect-orthog-n}
The orthogonal projection $P_{\E{R}_m}$ of the RV $R(x) \in \E{R}$ is
characterised by:
\begin{equation} \label{eq:iIII-n-v}
  P_{\E{R}_m}(R(x)) := \textup{arg min}_{\tilde{R}\in\E{R}_m}
  \; \frac{1}{2}  \;\| R(x) - \tilde{R} \|^2_{\E{R}},
\end{equation}
The \refeq{eq:iIVxxN} implies the existence of a optimal map 
$\Phi_m \in L_{0,m}(\C{Y},\C{R})$ such that
\begin{equation}  \label{eq:opt-condex-n}
  P_{\E{R}_m}(R(x)) = \Phi_m(\hat{h}(x,\vepsilon v)).
\end{equation}
Taking the vanishing of the first variation / G\^ateaux derivative of
the loss-function as a necessary condition for a
minimum leads to a simple geometrical interpretation: the difference
between the original vector-valued RV $R(x)$ and its projection has to
be perpendicular to the subspace $\E{R}_m$: $ \forall \tilde{R} \in \E{R}_m:$
\begin{equation}  \label{eq:iIV-n}
  \ipd{R(x)  - P_{\E{R}_m}(R(x))}{\tilde{R}}_{\E{R}} = 0, \; 
  \text{ i.e. } R(x)  - P_{\E{R}_m}(R(x)) \in \E{R}_m^\perp .
\end{equation}
Rephrasing \refeq{eq:iIII-n-v} with account to \refeq{eq:iIV-n} leads
for the optimal map $\Phi_m \in L_{0,n}(\C{Y},\C{R})$ to
\begin{equation} \label{eq:iV-DD-n}
  P_{\E{R}_m}(R(x)) =  \Phi_m(\hat{h}(x,\vepsilon v)) :=
  \textup{arg min}_{\vphi\in L_{0,m}(\C{Y},\C{R})} \; \|
  R(x)-\vphi(\hat{h}(x,\vepsilon v)) \|^2_{\E{R}},
\end{equation}
and the orthogonality condition of \refeq{eq:iIV-n} leads to
\begin{equation}  \label{eq:iIV-DDl-nR}
  \forall \vphi \in L_{0,m}(\C{Y},\C{R}): \;
  \ipd{R(x)  - \Phi_m(\hat{h}(x,\vepsilon v))}{\vphi(\hat{h}(x,\vepsilon v))}_{\E{R}} = 0 .
\end{equation}
In addition, as $\E{R}_m$ is linearly closed, one obtains the
useful statement 
\begin{equation}  \label{eq:iIV-L-t-n}
  \forall \tilde{R} \in \E{R}_m: \;
  \E{L}(\C{R}) \ni \EXP{(R(x)  - P_{\E{R}_m}(R(x)))\otimes \tilde{R}} = 0.
\end{equation}
or rephrased $\forall \vphi \in L_{0,m}(\C{Y},\C{R})$:
\begin{equation}  \label{eq:iIV-DDl-n}
  \E{L}(\C{R}) \ni \EXP{(R(x)  - \Phi_m(\hat{h}(x,\vepsilon
    v)))\otimes\vphi(\hat{h}(x,\vepsilon v))} = 0 .
\end{equation}
There is a unique minimiser $P_{\E{R}_m}(R(x)) \in \E{R}_m$ to the
problem in \refeq{eq:iIII-n-v}, which satisfies the
\emph{orthogonality} condition \refeq{eq:iIV-n}, which is equivalent
to the the \emph{strong orthogonality condition} \refeq{eq:iIV-L-t-n}.
The minimiser is unique as an element of $\E{R}_m$, but the
mapping $\Phi_m \in L_{0,m}(\C{Y},\C{R})$ in \refeq{eq:iV-DD-n} may not
necessarily be.  It also holds that
\begin{equation}  \label{eq:iIV-P-n}
\|P_{\E{R}_m}(R(x))\|_{\E{R}}^2 = \|R(x)\|_{\E{R}}^2 - 
             \|R(x) - P_{\E{R}_m}(R(x))\|_{\E{R}}^2 .
\end{equation}
Additionally, one has
\begin{equation}  \label{eq:stab-Pn}
  \|P_{\E{R}_m}(R(x))\|_{\E{R}}^2 \leq  \|P_{\E{R}_\infty}(R(x))\|_{\E{R}}^2 .
\end{equation}
\end{thm}
\begin{proof}
  It is all already contained in \refT{Doob-Dynkin-v} and
  \refT{cond-expect-orthog} when applied to $\E{R}_m$.  The stability
  condition \refeq{eq:stab-Pn} is due to the simple fact that $\E{R}_m
  \subseteq \E{R}_\infty$.
\end{proof}

From the consistency condition, the stability \refeq{eq:stab-Pn} as
shown in \refT{cond-expect-orthog-n}, and \emph{C\'ea's} lemma, one
immediately obtains:
\begin{thm}  \label{T:Cea-Ln}
For all RVs $R(x) \in \E{R}$, 
the sequence $R_m := P_{\E{R}_m}(R(x))$ converges to
$R_\infty := P_{\E{R}_\infty}(R(x))$:
\begin{equation}  \label{eq:iIVxxC}
\lim_{m\to\infty} \|R_\infty - R_m \|^2_{\E{R}} = 0.      
\end{equation}
\end{thm}
\begin{proof}
Well-posedness is a direct consequence of \refT{cond-expect-orthog}.
As the $P_{\E{R}_m}$ are orthogonal projections
onto the subspaces $\E{R}_m$, their norms are hence all equal to 
unity --- a stability condition, as shown in \refeq{eq:stab-Pn}.
Application of C\'ea's lemma then directly yields \refeq{eq:iIVxxC}.
\end{proof}

\subsubsection{Approximation by polynomials} \label{SSS:approx-Phi-poly}
Here we choose the subspaces of polynomials up to degree $m$ for the purpose
of approximation, i.e.\
\[ \E{R}_m := \overline{\spn}\;\{ \vphi \; | \; \vphi(\hat{h}(x,
\vepsilon v)) \in \E{R},\;  \; \vphi \in \C{P}_m(\C{Y},\C{X}) \}, \]
where $\C{P}_m(\C{Y},\C{X}) \subset L_0(\C{Y},\C{X})$ are the
polynomials of degree at most $m$ on $\C{Y}$ with values in $\C{X}$.
We may write $\psi_m \in \C{P}_m(\C{Y},\C{X})$ as
\begin{equation}  \label{eq:n-deg-pol}
 \psi_m(y) := \Hf{0}{H} + \Hf{1}{H}y + \dotsb 
+ \Hf{k}{H}y^{\vee k}+\dotsb + \Hf{m}{H}y^{\vee m},
\end{equation}
where $\Hf{k}{H} \in \E{L}^k_s(\C{Y},\C{R})$ is symmetric
and $k$-linear; and $y^{\vee k} := \overbrace{y \vee \ldots \vee y}^{k}
:= \text{Sym}(y^{\otimes k})$ is the symmetric tensor product of
the $y$'s taken $k$ times with itself.  Let us remark here that the form of
\refeq{eq:n-deg-pol}, given in monomials, is numerically not a good form---except
for very low $m$---and straightforward use in computations is not recommended.
The relation \refeq{eq:n-deg-pol} could be re-written in some orthogonal 
polynomials---or in fact any other system of multi-variate functions; this
generalisation will be considered in \refSSS{approx-Phi-gen}.
For the sake of conceptual simplicity, we stay with \refeq{eq:n-deg-pol} and
then have that for any RV $R(x) \in \E{R}$
\begin{equation} \label{eq:n-deg-pol-q}
  \Phi_{R,m}(R(x)) :=
  \psi_{R,m}(y) := \Hf{0}{H} + \dotsb + \dotsb + \Hf{m}{H}y^{\vee m}
  =: \Psi_{R,m}(\Hf{0}{H},\dots,\Hf{m}{H})
\end{equation}
the optimal map in \refeq{eq:opt-condex-n} from
\refT{cond-expect-orthog-n} --- where we have added an index $R$ to
indicate that it depends on the RV $R(x)$, but for simplicity omitted
this index on the coefficient maps $\Hf{k}{H}$ --- is a function
$\Psi_{R,m}$ of the coefficient maps $\Hf{k}{H}$.  The stationarity or
orthogonality condition \refeq{eq:iIV-DDl-n} can then be written in
terms of the $\Hf{k}{H}$.  We need the following abbreviations for any
$k,\ell\in\D{N}_0$ and $p \in \E{R}, v \in \E{Y}$:
\[\ipj{p \otimes v^{\vee k}} := 
     \EXP{p \otimes v^{\vee k}} = \int_{\Omega} p(\omega)\otimes
  v(\omega)^{\vee k} \, \D{P}(\di \omega) \]
and
\[ \Hf{k}{H}\ipj{y^{\vee(\ell+k)}} := 
\ipj{y^{\vee\ell}\vee (\Hf{k}{H}y^{\vee k})}=
\EXP{y^{\vee\ell}\vee (\Hf{k}{H}y^{\vee k})} .\]

We may then characterise the $\Hf{k}{H}$ in the following way:
\begin{thm} \label{T:n-vers}
With $\Psi_{R,m}$ from \refeq{eq:n-deg-pol-q}, the stationarity condition
\refeq{eq:iIV-DDl-n} becomes, by the chain rule, for any $m\in\D{N}_0$
\begin{equation}  \label{eq:cond-H}
\forall \ell=0,\dotsc,m:\quad \sum_{k=0}^m \Hf{k}{H}\ipj{y^{\vee(\ell+k)}} = 
      \ipj{R(x)\otimes y^{\vee\ell}}.
\end{equation}
The Hankel operator matrix $(\ipj{y^{\vee(\ell+k)}})_{\ell,k} $ in the
linear equations \refeq{eq:cond-H} is symmetric and positive
semi-definite, hence the system \refeq{eq:cond-H} has a solution,
unique in case the operator matrix is actually definite.
\end{thm}
\begin{proof}
The relation \refeq{eq:cond-H} is the result of straightforward
application of the chain rule to the \refeq{eq:iIV-DDl-n}.

The symmetry of the operator matrix is obvious---the $\ipj{y^{\vee k}}$ are the
coefficients---and positive semi-definiteness follows
easily from the fact that it is the gradient of the functional in
\refeq{eq:iV-DD-n}, which is convex.
\end{proof}
Observe that the operator matrix is independent of the RV $R(x)$ for
which the ncomputation is performed.  Only the right hand side is
influenced by $R(x)$.

The system of operator equations \refeq{eq:cond-H} may be written in
more detailed form as:
\begin{alignat*}{6}
&\ell = 0:\; \Hf{0}{H} &\dotsb  &+ \Hf{k}{H}\ipj{y^{\vee k}}
   &\dotsb &+ \Hf{m}{H}\ipj{y^{\vee m}}  =& \ipj{R(x)}, \\
&\ell = 1: \; \Hf{0}{H}\ipj{y} &\dotsb  &+ 
   \Hf{k}{H}\ipj{y^{\vee (1+k)}} &\dotsb &+ \Hf{m}{H}\ipj{y^{\vee (1+m)}}
     =& \ipj{R(x) \otimes y}, \\
&\vdots &\dotso   & & \vdots & & \vdots \\
&\ell = m:\; \Hf{0}{H}\ipj{y^{\vee m}} &\dotsb  &+ 
    \Hf{k}{H}\ipj{y^{\vee (k+m)}} &\dotsb &+  \Hf{m}{H}\ipj{y^{\vee 2m}} 
    =& \ipj{R(x) \otimes y^{\vee m}}.
\end{alignat*}

Using \emph{`symbolic index'} notation a la Penrose --- the reader may
just think of indices in a finite dimensional space with orthonormal
basis --- the system \refeq{eq:cond-H} can be given yet another form:
denote in symbolic index notation $R(x) = (R^\imath), y = (y^\jmath)$,
and $\Hf{k}{H} =
(\tensor*[^k]{H}{^\imath_{\jmath_1}_{\dotso}_{\jmath_k}})$, then
\refeq{eq:cond-H} becomes, with the use of the Einstein convention of
summation (a tensor contraction) over repeated indices, and with the
symmetry explicitly indicated:
\begin{multline}  \label{eq:symbolic}
\forall \ell=0,\dotsc,m; \;  \jmath_1 \le \ldots \le \jmath_{\ell}\le \ldots \le
 \jmath_{\ell + k}\le \ldots \le \jmath_{\ell+m}:\\
\ipj{y^{\jmath_1}\dotsm y^{\jmath_\ell}} \,
   (\tensor*[^0]{H}{^\imath}) + \dotsb +
  \ipj{y^{\jmath_1}\dotsm y^{\jmath_{\ell+1}}\dotsm y^{\jmath_{\ell+k}}} \,
     (\tensor*[^k]{H}{^\imath_{\jmath_{\ell+1}}_{\dotso}_{\jmath_{\ell+k}}}) +\\ 
     \dotsb +  \ipj{y^{\jmath_1}\dotsm y^{\jmath_{\ell+1}}
     \dotsm y^{\jmath_{\ell+m}}} \,
   (\tensor*[^m]{H}{^\imath_{\jmath_{\ell+1}}_{\dotso}_{\jmath_{\ell+m}}}) =
  \ipj{R^\imath y^{\jmath_1}\dotsm y^{\jmath_\ell}}.
\end{multline}
We see in this representation that the matrix does \emph{not} depend
on $\imath$---it is identically \emph{block diagonal} after appropriate
reordering, which makes the solution of \refeq{eq:cond-H} or
\refeq{eq:symbolic} much easier.

Some special cases are: for $m=0$---\emph{constant} functions.  One does
not use any information from the measurement --- and from
\refeq{eq:cond-H} or \refeq{eq:symbolic} one has
\[ \Phi_{R,0}(R(x)) = \psi_{R,0}(y) = \Hf{0}{H} = \ipj{R} = \EXP{R} =
\bar{R}. \]
Without any information, the conditional expectation is equal to the
unconditional expectation.  The update corresponding to
\refeq{eq:orth-upd} --- actually \refeq{eq:orth-upd-g} as we are
approximating the map $\Phi_R$ by $g_R = \Phi_{R,0}$ --- then becomes
$R_a \approx R_{a,0} = R(x_f) = R_f$, as $R_\infty = 0$ in this case;
the assimilated quantity stays equal to the forecast.  This was to be
expected, is not of much practical use, but is a consistency check.

The case $m=1$ in \refeq{eq:cond-H} or \refeq{eq:symbolic} is more interesting,
allowing up to \emph{linear} terms:
\begin{alignat*}{5}
&\Hf{0}{H} &+ &\Hf{1}{H}\ipj{y}  &= &\ipj{R(x)} =  \bar{R} \\
&\Hf{0}{H}\ipj{y} &+ &\Hf{1}{H}\ipj{y \vee y} &= &\ipj{R(x) \otimes y}.
\end{alignat*}
Remembering that $C_{Ry} = \ipj{R(x) \otimes y}-\ipj{R}\otimes\ipj{y}$ and
analogous for  $C_{yy}$, one obtains by tensor multiplication
with $\ipj{R(x)}$ and symbolic Gaussian elimination
\begin{alignat*}{3}
\Hf{0}{H} &= &\ipj{R} &- \Hf{1}{H}\ipj{y} =  \bar{R}  - \Hf{1}{H}\bar{y}\\
\Hf{1}{H}(\ipj{y \vee y}-\ipj{y}\vee\ipj{y})= \Hf{1}{H}C_{yy}  &= 
    &\ipj{R(x) \otimes y} &- \ipj{R}\otimes\ipj{y} = C_{Ry}.
\end{alignat*}
This gives 
\begin{align} \label{eq:LBU-n1-1H}
\Hf{1}{H} &= C_{Ry}C_{yy}^{-1} =: K\\
\Hf{0}{H} &= \bar{R} - K \bar{y}. \label{eq:LBU-n1-0H}
\end{align}
where $K$ in \refeq{eq:LBU-n1-1H} is the well-known \emph{Kalman} gain operator 
\citep{Kalman}, so that finally
\begin{equation}   \label{eq:cond-ex-n1}
\Phi_{R,1}(R(x)) = \psi_{R,1}(y) = \Hf{0}{H} + \Hf{1}{H}y = 
      \bar{R} + C_{Ry}C_{yy}^{-1}(y - \bar{y}) = \bar{R} + K(y - \bar{y}).
\end{equation}

The update corresponding to \refeq{eq:orth-upd} --- again actually
\refeq{eq:orth-upd-g} as we are approximating the map $\Phi_R$ by
$g_R = \Phi_{R,1}$ --- then becomes
\begin{multline} \label{eq:LBU-n1}
  R_a \approx R_{a,1} = R(x_f) + \left( (\bar{R} + K(\hat{y}-\bar{y}))
  - (\bar{R} + K(y(x_f)-\bar{y}))\right) = \\
   R_f + K(\hat{y} - y(x_f)) = R_f + R_{\infty,1}.
\end{multline}
This may be called a \emph{linear} Bayesian update (LBU), and is
similar to the `Bayes linear' approach \citep{Goldstein2007}.  It is
important to see \refeq{eq:LBU-n1} as a symbolic expression,
especially the inverse $C_{yy}^{-1}$ indicated in
\refeq{eq:cond-ex-n1} should not really be computed, especially when
$C_{yy}$ is ill-conditioned or close to singular.  The inverse can in
that case be replaced by the \emph{pseudo-inverse}, or rather the
computation of $K$, which is in linear algebra terms a
\emph{least-squares} approximation, should be done with orthogonal
transformations and not by elimination.  We will not dwell on these
well-known matters here.  It is also obvious that the constant term in
\refeq{eq:cond-ex-n1} --- or even \refeq{eq:n-deg-pol-q} for that
matter --- is of no consequence for the update filter, as it cancels
out.

The case $m=2$ can still be solved symbolically, the system to be
solved is from \refeq{eq:cond-H} or \refeq{eq:symbolic}:
\begin{alignat*}{7}
&\Hf{0}{H} &+ &\Hf{1}{H}\ipj{y} &+ &\Hf{2}{H}\ipj{y^{\vee 2}}  &= &\ipj{R} \\
&\Hf{0}{H}\ipj{y} &+ &\Hf{1}{H}\ipj{y^{\vee 2}} &+ &\Hf{2}{H}\ipj{y^{\vee 3}}&
= &\ipj{R \otimes y}\\
&\Hf{0}{H}\ipj{y^{\vee 2}} &+ &\Hf{1}{H}\ipj{y^{\vee 3}} &+ 
&\Hf{2}{H}\ipj{y^{\vee 4}}&= &\ipj{R \otimes y^{\vee 2}}.
\end{alignat*}
After some symbolic elimination steps one obtains
\begin{alignat*}{7}
&\Hf{0}{H} &+ &\Hf{1}{H}\ipj{y} &+ &\Hf{2}{H}\ipj{y^{\vee 2}}  &= &\bar{R} \\
&0 &+ &\Hf{1}{H} &+ &\Hf{2}{H}\; \vek{F} &= &K\\
&0 &+ &0 &+ &\Hf{2}{H}\; \tnb{G} &= &\tns{E},
\end{alignat*}
with the Kalman gain operator $K\in(\C{R}\otimes\C{Y})^*$ from \refeq{eq:LBU-n1-1H},
the third order tensors $\vek{F}\in(\C{Y}^{\otimes 3})^*$ given in \refeq{eq:NLBU-n2-F},
and $\tns{E}\in (\C{R}\otimes\C{Y}^{\otimes 2})^*$ given in \refeq{eq:NLBU-n2-E},
and the fourth order tensor $\tnb{G}\in(\C{Y}^{\otimes 4})^*$ given
in \refeq{eq:NLBU-n2-G}:
\begin{align} \label{eq:NLBU-n2-F}
\vek{F} &= \left(\ipj{y^{\vee 3}} - \ipj{y^{\vee 2}} \vee \ipj{y}\right) 
 C_{yy}^{-1}, \\  \label{eq:NLBU-n2-E}
 \tns{E} &= \ipj{R \otimes y^{\vee 2}} - \bar{R}\otimes\ipj{y^{\vee 2}} -
      K  \left(\ipj{y^{\vee 3}} - 
     \ipj{y} \vee \ipj{y^{\vee 2}} \right)\\  \label{eq:NLBU-n2-G}
\tnb{G} &= \left(\ipj{y^{\vee 4}}-\ipj{y^{\vee 2}}^{\vee 2}\right) - 
     \vek{F} \cdot \left(\ipj{y^{\vee 3}} - 
     \ipj{y} \vee \ipj{y^{\vee 2}} \right),
\end{align}
where the single central dot `$\cdot$' denotes as usual a contraction
over the appropriate indices, and a colon `:' a double contraction.
From this one easily obtains the solution
\begin{align} \label{eq:NLBU-n2-2H}
 \Hf{2}{H}&= \tns{E}:\tnb{G}^{-1} \\  \label{eq:NLBU-n2-1H}
 \Hf{1}{H} &= K -  \Hf{2}{H}  \vek{F} \\  \label{eq:NLBU-n2-0H}
 \Hf{0}{H} &= \bar{R} - (K-\Hf{1}{H})  \bar{y} - \Hf{2}{H}  \ipj{y^{\vee 2}} =
 \bar{R} - \Hf{2}{H}(\vek{F} \cdot \bar{y}  + \ipj{y^{\vee 2}}).
\end{align}

The update corresponding to \refeq{eq:orth-upd} --- again actually
\refeq{eq:orth-upd-g} as we are approximating the map $\Phi_R$ now by
$g_R = \Phi_{R,2}$ --- then becomes
\begin{multline} \label{eq:QBU-n2}
  R_a \approx R_{a,2} = R(x_f) + \left( (\Hf{2}{H}\hat{y}^{\vee 2} +
    \Hf{1}{H}\hat{y})  - (\Hf{2}{H}y(x_f)^{\vee 2} +
    \Hf{1}{H} y(x_f)) \right) = \\
   R_f + \left( \tns{E}:\tnb{G}^{-1}:\left(\hat{y}^{\vee 2} -y(x_f)^{\vee 2}\right) +
     (K-\tns{E}:\tnb{G}^{-1}:\vek{F})(\hat{y}-y(x_f)) \right)\\  = R_f + R_{\infty,2}.
\end{multline}
This may be called a \emph{quadratic} Bayesian update (QBU), and it is
clearly an extension of \refeq{eq:LBU-n1}.

\subsubsection{The Gauss-Markov-Kalman filter} \label{SSS:GMK-filter}
The $m=1$ version of \refT{n-vers} is well-known for the special case $R(x):=x$,
and we rephrase this generalisation of the well-known \emph{Gauss-Markov}
theorem from \citep{Luenberger1969}  Chapter 4.6, Theorem 3:
\begin{prop} \label{prop:n-1-vers} 
  The update $x_{a,1}$, minimising $\|x_f - \cdot\|^2_{\E{X}}$ over
  all elements generated by affine mappings (the up to $m=1$ case of
  \refT{n-vers}) of the measurement $\hat{y}$ with predicted
  measurement $y(x_f)$ is given
\begin{equation}   \label{eq:iVII}
  x_{a,1} = x_f + K(\hat{y} - y(x_f)),
\end{equation}
where the operator $K$ is the Kalman gain from  \refeq{eq:LBU-n1-1H}
and \refeq{eq:LBU-n1}. 
\end{prop}
The \refeq{eq:iVII} is reminiscent --- actually an extension --- not
only of the well-known \emph{Gauss-Markov} theorem
\citep{Luenberger1969}, but also of the \emph{Kalman} filter
\citep{Kalman} \citep{Papoulis1998/107}, so that we propose to call
\refeq{eq:iVII} the \tbf{Gauss-Markov-Kalman} (GMK) filter (GMKF).

We point out that $x_{a,1}, x_f$, and $y(x_f)$ are RVs, i.e.\
\refeq{eq:iVII} is an equation in $\E{X} = \C{X} \otimes \C{S}$
between RVs, whereas the traditional Kalman filter is an equation in
$\C{X}$.  If the mean is taken in \refeq{eq:iVII}, one obtains the
familiar Kalman filter formula \cite{Kalman} for the update of the
mean, and one may show \citep{opBvrAlHgm12} that \refeq{eq:iVII} also
contains the Kalman update for the covariance by computing
\refeq{eq:orth-upd-x-var} for this case, which gives the familiar
result of Kalman, i.e.\ the Kalman filter is a low-order part of
\refeq{eq:iVII}.

The computational strategy for a typical filter is now to replace and
approximate the---only abstractly given---computation of $x_a$
\refeq{eq:orth-upd-x} by the practically possible calculation of
$x_{a,m}$ as in \refeq{eq:n-deg-pol-q}.  This means that we
approximate $x_a$ by $x_{a,m}$ by using $\E{X}_m \subseteq
\E{X}_\infty$, and rely on \refT{Cea-Ln}.  This corresponds to some loss
of information from the measurement as one uses a smaller subspace for
the projection, but yields a manageable computation.  If the
assumptions of \refT{Cea-Ln} are satisfied, then one can expect for $m$
large enough that the terms in \refeq{eq:n-deg-pol-q} converge to
zero, thus providing an error indicator on when a sufficient accuracy
has been reached.

\subsubsection{Approximation by general functions} \label{SSS:approx-Phi-gen}
The derivation in \refSSS{approx-Phi-poly} was for the special case
where polynomials are used to find a subspace $L_{0,m}(\C{Y},\C{X})$
for the approximation.  It had the advantage of showing the connection
to the `Bayes linear' approach \citep{Goldstein2007}, to the
Gauss-Markov theorem \citep{Luenberger1969}, and to the \emph{Kalman}
filter \citep{Kalman} \citep{Papoulis1998/107}, giving in
\refeq{eq:iVII} of Proposition~\ref{prop:n-1-vers} the
\emph{Gauss-Markov-Kalman} filter (GMKF).

But for a more general approach not limited to polynomials, we proceed
similarly as in \refeq{eq:iIVxxN}, but now concretely assume a set of
linearly independent functions, not necessarily orthonormal, 
\begin{equation}  \label{eq:basis-L0}
  \C{B} := \{\psi_\alpha \; | \; \alpha \in \C{A}, \; \psi_\alpha \in L_0(\C{Y}); \;
   \psi_\alpha(\hat{h}(x,\vepsilon v)) \in \C{S}\}  \subseteq \C{S}_\infty   
\end{equation}
where $\C{A}$ is some countable index set.  Assume now that
\[ \E{S}_\infty = \overline{\spn}\; \C{B}, \] 
i.e.\ $\C{B}$ is a Hilbert basis of $\C{S}_\infty$, again a
consistency condition.

Denote by $\C{A}_k$ a finite part of $\C{A}$ of cardinality $k$, such
that $\C{A}_k \subset \C{A}_\ell$ for $k<\ell$ and $\bigcup_k \C{A}_k
=\C{A}$, and set
\begin{equation}  \label{eq:def-Rn}
\E{R}_k :=  \C{R} \otimes \C{S}_k \subseteq \E{R}_\infty ,
\end{equation}
where the finite dimensional and hence closed subspaces $\C{S}_k$ are
given by
\begin{equation}  \label{eq:def-Sn}
  \C{S}_k := \spn \{\psi_\alpha \; | \; \alpha \in
      \C{A}_k, \; \psi_\alpha \in \C{B} \} \subseteq \C{S} .
\end{equation}
Observe that the spaces $\E{R}_k$ from \refeq{eq:def-Rn} are linearly closed
according to Proposition~\ref{prop:fst-res-lcl}.

\refT{cond-expect-orthog-n} and \refT{Cea-Ln} apply in this case.
For a RV $R(x) \in \E{R}$ we make the following `ansatz' for the
optimal map $\Phi_{R,k}$ such that $P_{\E{R}_k}(R(x)) =
\Phi_{R,k}(\hat{h}(x, \vepsilon v))$:
\begin{equation}  \label{eq:ansatz-psi}
   \Phi_{R,k}(y) = \sum_{\alpha \in \C{A}_k} v_\alpha \psi_\alpha(y),
\end{equation}
with as yet unknown coefficients $v_\alpha \in \C{R}$.  This is a
normal \emph{Galerkin}-ansatz, and \refeq{eq:iIV-DDl-n} from
\refT{cond-expect-orthog-n} can be used to determine these
coefficients.

Take $\C{Z}_k := \D{R}^{\C{A}_k}$ with canonical basis
$\{\vek{e}_\alpha \; | \; \alpha \in \C{A}_k \}$, and let 
\[ \vek{G}_k
:= (\ip{\psi_\alpha(y(x))}{\psi_\beta(y(x))}_{\C{S}})_{\alpha, \beta
  \in \C{A}_k} \in \E{L}(\C{Z}_k)
\]
be the symmetric positive definite Gram matrix of the basis of
$\C{S}_k$; also set
\begin{align*}
  \tnb{v} &:= \sum_{\alpha \in \C{A}_k} \vek{e}_\alpha \otimes v_\alpha
  \in \C{Z}_k \otimes \C{R}, \\
  \tnb{r} &:= \sum_{\alpha \in \C{A}_k} \vek{e}_\alpha \otimes
  \EXP{\psi_\alpha(y(x)) R(x)} \in \C{Z}_k \otimes \C{R}.
\end{align*}
\begin{thm}
  For any $k \in \D{N}$, the coefficients $\{ v_\alpha\}_{\alpha \in
    \C{A}_k}$ of the optimal map $\Phi_{R,k}$ in \refeq{eq:ansatz-psi}
  are given by the unique solution of the Galerkin equation
  \begin{equation} \label{eq:psi-Galerkin}
    (\vek{G}_k \otimes I_{\C{R}}) \tnb{v} = \tnb{r} .
  \end{equation}
It has the formal solution 
\[ \tnb{v} = (\vek{G}_k \otimes I_{\C{R}})^{-1} \tnb{r} = 
    (\vek{G}_k^{-1} \otimes I_{\C{R}}) \tnb{r} \in \C{Z}_k \otimes \C{R} . \]
\end{thm}
\begin{proof}
  The Galerkin \refeq{eq:psi-Galerkin} is a simple consequence of
  \refeq{eq:iIV-DDl-n} from \refT{cond-expect-orthog-n}.  As the Gram
  matrix $\vek{G}_k$ and the identity $I_{\C{R}}$ on $\C{R}$ are
  positive definite, so is the tensor operator $(\vek{G}_k \otimes
  I_{\C{R}})$, with inverse $(\vek{G}_k^{-1} \otimes I_{\C{R}})$.
\end{proof}
As in \refeq{eq:symbolic}, the block structure of the equations is
clearly visible.  Hence, to solve \refeq{eq:psi-Galerkin}, one only has
to deal with the `small' matrix $\vek{G}_n$.

The update corresponding to \refeq{eq:orth-upd} --- again actually
\refeq{eq:orth-upd-g} as we are approximating the map $\Phi_R$ now by
a new map $g_R = \Phi_{R,k}$ --- then becomes
\begin{equation} \label{eq:GBU-ng}
  R_a \approx R_{a,k} = R(x_f) + \left( \Phi_{R,k}(\hat{y}) - 
  \Phi_{R,k}(y(x_f)) \right)  = R_f +  R_{\infty,k}.
\end{equation}
This may be called a `general Bayesian update'.  Applying
\refeq{eq:GBU-ng} now again to the special case $R(x):=x$, one obtains
a possibly nonlinear filter based on the basis $\C{B}$:
\begin{equation} \label{eq:GBU-xg}
  x_a \approx x_{a,k} =  x_f + \left( \Phi_{x,k}(\hat{y}) - 
  \Phi_{x,k}(y(x_f)) \right)  = x_f +  x_{\infty,k}.
\end{equation}
In case the $\C{Y}^* \subseteq \spn \{ \psi_\alpha\}_{\alpha \in
  \C{A}_k}$, i.e.\ the basis generates all the linear functions on
$\C{Y}$, this is a true extension of the Kalman filter.

%  $Log: char-rv.tex,v $
%  Revision 3.7  2015/10/22 17:32:55  matthies
%  adjusting lines
%
%  Revision 3.6  2015/10/20 18:58:46  matthies
%  some corrections
%
%  Revision 3.5  2015/10/20 09:06:16  matthies
%  some corrections
%
%  Revision 3.4  2015/10/19 01:26:39  matthies
%  really finished
%
%  Revision 3.3  2015/10/17 20:11:02  matthies
%  first half finished
%
%  Revision 3.2  2015/10/16 12:55:54  matthies
%  started earnestly
%
%  Revision 3.1  2015/10/15 20:19:20  matthies
%  section and sub-sections completed
%
%  Revision 2.1  2015/10/15 20:16:50  matthies
%  section and sub-sections
%
%  Revision 1.1  2015/10/15 20:12:52  matthies
%  file template
%
%
%
%

%%% Local Variables: 
%%% mode: latex
%%% TeX-master: "../NonLinBU"
%%% End: 

% RCSID:       $Id: num-realis.tex,v 3.5 2015/10/22 17:33:17 matthies Exp $
% Author:      $Author: matthies $
% Contact:     wire@tu-bs.de
% =================================
%% texfile{
%%  AUTHOR    = "$Author: matthies $",
%%  VERSION   = "$Revision: 3.5 $",
%%  DATE      = "$Date: 2015/10/22 17:33:17 $",
%%  FILENAME  = "$RCSfile: num-realis.tex,v $"}
%
% =================================

\section{Numerical realisation} \label{S:num-real}
In the instances where we want to employ the theory detailed in the
previous \refS{bayes} and \refS{char-rv}, the spaces $\C{U}$ and
$\C{Q}$ and hence $\C{X}$ are usually infinite dimensional, as is the
space $\C{S} = L_2(\Omega)$.  For an actual computation they all have
to be discretised or approximated by finite dimensional subspaces.

In our examples we will chose finite element discretisations for
$\C{U}$, $\C{Q}$, and hence $\C{X}$, and corresponding subspaces.
Hence let $\C{X}_M := \text{span }\{\vrho_m\ : m=1,\dots,M\} \subset
\C{X}$ be an $M$-dimensional subspace with basis
$\{\vrho_m\}_{m=1}^M$.  An element of $\C{X}_M$ will be represented by
the vector $\vek{x}=[x^1, \dots, x^M]^T \in \D{R}^M$ such that
$\sum^M_{m=1} x^m \vrho_m \in \C{X}_M$.  To avoid a profusion of
notations, the corresponding random vector in $\D{R}^M \otimes \C{S}$
--- a mapping $\Omega \to  \D{R}^M \cong \C{X}_M$ --- will also be
denoted by $\vek{x}$, as the meaning will be clear from the context.

The norm $\nd{\vek{x}}_{M}$ one has to take on
$\D{R}^M$ results from the inner product
$\bkt{\vek{x}_1}{\vek{x}_2}_{M} := \vek{x}_1^T \vek{Q}\vek{x}_2$ with
$\vek{Q} = \left(\bkt{\vrho_m}{\vrho_n}_{\C{X}}\right)$, the Gram
matrix of the basis.  We will later choose an orthonormal basis, so
that $\vek{Q} = \vek{I}$ is the identity matrix.  Similarly, on
$\E{X}_M = \D{R}^M \otimes \C{S}$ the inner product is
$\bkd{\vek{x}_1}{\vek{x}_2}_{\E{X}_M} := \EXP{\bkt{\vek{x}_1}{\vek{x}_2}_{M}}$.

The space of possible measurements $\C{Y}$ can usually be taken to be
finite dimensional, otherwise we take similarly as before a
$R$-dimensional subspace $\C{Y}_R$, whose elements are similarly
represented by a vector of coefficients $\vek{y} \in \D{R}^R$.  For
the discretised version of the RV $y(x_f) = y(\hat{h}(x_f, \vepsilon
v))$ we will often use the shorthand $\vek{y}_f := \vek{y}(\vek{x}_f)
= \vek{y}(\hat{h}(\vek{x}_f, \vepsilon \vek{v}))$.

As some of the most efficient ways of doing the update are linear
filters based on the general idea of orthogonal decomposition ---
\refeq{eq:orth-upd} in \refSS{char-p_rv-filt} --- applied to the mean
--- \refeq{eq:orth-upd-x} in \refSSS{filt-mean} --- but in the
modified form \refeq{eq:orth-upd-g} where $g$ is a linear map, and
especially the optimal linear map of the Gauss-Markov-Kalman (GMK)
filter \refeq{eq:iVII}, we start from Proposition~\ref{prop:n-1-vers}
in \refSSS{GMK-filter}.  For other approximations the finite
dimensional discretisation would be largely analogous.

On $\D{R}^M$, representing $\C{X}_M$, the Kalman gain operator in
Proposition~\ref{prop:n-1-vers} in \refeq{eq:iVII} becomes a matrix
$\vek{K}\in \D{R}^{M \times R}$.  Then the update corresponding to
\refeq{eq:iVII} is
\begin{equation} \label{eq:iIX} 
  \vek{x}_a = \vek{x}_f + \vek{K}(\vhat{y} - \vek{y}(\vek{x}_f)),
  \text{ with } \vek{K} =  \vek{C}_{xy}\, \vek{C}_{yy}^{-1}.
\end{equation}
Here the covariances are $\vek{C}_{xy} := \EXP{\tilde{\vek{x}}_f\;
  \vtil{y}(\vek{x}_f)}$, and similarly for $\vek{C}_{yy}$.  Often the
measurement error $v$ in the measurement model $\tilde{h}(x_f,
\vepsilon v) = h(x_f) + \vepsilon S_y(x_f) v$ is independent of
$\vek{x}$ --- actually \emph{uncorrelated} would be sufficient, i.e.\
$\vek{C}_{x v}=\vek{0}$ --- hence, assuming that $S_y$ does not depend
on $x$, $\vek{C}_{xx} = \vek{C}_{hh} + \vepsilon² \vek{S}_y
\vek{C}_{vv} \vek{S}_y^T$ and $\vek{C}_{xy} = \vek{C}_{xh}$, where $h
= h(x_f)$.

% We once more recall our comments in
% \refSS{approx-cond-expect} following \refeq{eq:LBU-n1} regarding the
% inverse which also appears in \refeq{eq:iIX}.  Recall that usually the
% error model involves a regular covariance
% $\vek{C}_{\vepsilon,\vepsilon}$, so that $\vek{C}_{z,z} =
% \vek{C}_{y,y} + \vek{C}_{\vepsilon,\vepsilon}$ is at least
% theoretically regular.

It is important to emphasise that the theory presented in the forgoing
\refS{bayes} and \refS{char-rv} is independent of any discretisation
of the underlying spaces.  But one usually can still not numerically
compute with objects like $\vek{x}\in \E{X}_M = \D{R}^M \otimes
\C{S}$, as $\C{S} = L_2(\Omega)$ is normally an infinite dimensional
space, and has to be discretised.  One well-known possibility are
samples, i.e.\ the RV $\vek{x}(\omega)$ is represented by its value at
certain points $\omega_z$, and the points usually come from some
quadrature rule.  The well-known Monte Carlo (MC) method uses random
samples, the quasi-Monte Carlo (QMC) method uses low discrepancy
samples, and other rules like sparse grids (Smolyak rule) are
possible.  Using MC samples in the context of the linear update
\refeq{eq:iVII} is known as the \emph{Ensemble Kalman Filter} (EnKF),
see \citep{BvrAkJsOpHgm11} for a general overview in this context, and
\citep{Evensen2009}, \citep{Evensen2009a} for a thorough description
and analysis.  This method is conceptually fairly simple and is
currently a favourite for problems where the computation of the
predicted measurement $\vek{y}(\vek{x}_f(\omega_z))$ is difficult or
expensive.  It needs far fewer samples for meaningful results than
MCMC, but on the other hand it uses the linear approximation inherent
in \refeq{eq:iIX}.

Here we want to use so-called \emph{functional} or \emph{spectral}
approximations, so similarly as for $\C{X}_M$, we pick a finite set of
linearly independent vectors in $\C{S}$.  As $\C{S} = L_2(\Omega)$,
these abstract vectors are in fact RVs with finite variance.  Here we
will use the best known example, namely \emph{Wiener}'s
\emph{polynomial chaos} expansion (PCE) as basis \citep{Wiener1938},
\citep{ghanemSpanos91}, \citep{holdenEtAl96}, \citep{Janson1997},
\citep{malliavin97}, \citep{matthies6}, this allows us to use
\refeq{eq:iIX} without sampling, see \citep{BvrAkJsOpHgm11},
\citep{opBvrAlHgm12}, \citep{bvrAlOpHgm12-a}, \citep{boulder:2011},
\citep{OpBrHgm12}, and also \citep{saadGhn:2009},
\citep{Blanchard2010a}.

The PCE is an expansion in multivariate \emph{Hermite polynomials}
\citep{ghanemSpanos91}, \citep{holdenEtAl96}, \citep{Janson1997},
\citep{malliavin97}, \citep{matthies6}; we denote by
$H_{\vek{\alpha}}(\vek{\theta}) = \prod_{k \in \D{N}}
h_{\alpha_k}(\theta_k) \in \C{S}$ the multivariate polynomial in
standard and independent Gaussian RVs $\vek{\theta}(\omega) =
(\theta_1(\omega), \dots, \theta_k(\omega), \dots)_{k\in \D{N}}$, where
$h_j$ is the usual uni-variate Hermite polynomial, and $\vek{\alpha} =
(\alpha_1, \dots, \alpha_k, \dots)_{k\in \D{N}}\in
\C{N}:=\D{N}_0^{(\D{N})}$ is a multi-index of generally infinite
length but with only finitely many entries non-zero.  As $h_0 \equiv
1$, the infinite product is effectively finite and always
well-defined.

The \emph{Cameron-Martin} theorem assures us \citep{holdenEtAl96},
\citep{malliavin97}, \citep{Janson1997} that the set of these
polynomials is dense in $\C{S} = L_2(\Omega)$, and in fact
$\{H_{\vek{\alpha}}/\sqrt{(\vek{\alpha} !)} \}_{\vek{\alpha} \in
  \C{N}}$ is a complete orthonormal system (CONS), where $\vek{\alpha}
! := \prod_{k \in \D{N}} (\alpha_k !)$ is the product of the
individual factorials, also well-defined as except for finitely many
$k$ one has $\alpha_k ! = 0! = 1$.  So one may write $\vek{x}(\omega) =
\sum_{\vek{\alpha}\in \C{N}} \vek{x}^{\vek{\alpha}} H_{\vek{\alpha}}
(\vek{\theta}(\omega))$ with $\vek{x}^{\vek{\alpha}} \in \D{R}^M$, and
similarly for $\vek{y}$ and all other RVs.  In this way the RVs are
expressed as functions of other, known RVs $\vek{\theta}$---hence the
name \emph{functional} approximation---and not through samples.

The space $\C{S}$ may now be discretised by taking a finite subset $\C{J}
\subset \C{N}$ of size $J = \ns{\C{J}}$, and setting $\C{S}_J = \spn
\{H_{\vek{\alpha}}\,:\, \vek{\alpha} \in \C{J} \} \subset \C{S}$.  The
orthogonal projection $P_J$ onto $\C{S}_J$ is then simply
\begin{equation}  \label{eq:proj-J}
P_J: \C{X}_M \otimes \C{S} \ni
\sum_{\vek{\alpha}\in \C{N}} \vek{x}^{\vek{\alpha}} H_{\vek{\alpha}} \mapsto
\sum_{\vek{\alpha}\in \C{J}} \vek{x}^{\vek{\alpha}} H_{\vek{\alpha}} 
\in \C{X}_M \otimes \C{S}_J.
\end{equation}
Taking \refeq{eq:iIX}, one may rewrite it as
\begin{eqnarray}  \label{eq:proj-lin-f1}
  \vek{x}_a &=& \vek{x}_f + \vek{K}(\vhat{y} - \vek{y}_f) =\\
  \sum_{\vek{\alpha}\in \C{N}} \vek{x}_a^{\vek{\alpha}} H_{\vek{\alpha}}(\vek{\theta}) &=& 
  \sum_{\vek{\alpha}\in \C{N}} \left(\vek{x}_f^{\vek{\alpha}} + \vek{K}\left( 
  \vhat{y}^{\vek{\alpha}}-\vek{y}^{\vek{\alpha}}_f\right)\right)
  H_{\vek{\alpha}}(\vek{\theta}). \label{eq:proj-lin-f2}
\end{eqnarray}
Observe, that as the measurement or observation $\vhat{y}$ is a
constant, one has in \refeq{eq:proj-lin-f2} that only $\vhat{y}^{0} =
\vhat{y}$, all other coefficients $\vhat{y}^{\vek{\alpha}} = \vek{0}$
for $\vek{\alpha} \neq \vek{0}$.

Projecting both sides of \refeq{eq:proj-lin-f2} onto $\C{X}_M \otimes
\C{S}_J$ is very simple and results in
\begin{equation} \label{eq:proj-lin-J}
  \sum_{\vek{\alpha}\in \C{J}} \vek{q}_a^{\vek{\alpha}} H_{\vek{\alpha}} = 
  \sum_{\vek{\alpha}\in \C{J}} \left(\vek{q}_f^{\vek{\alpha}} + \vek{K}\left( 
  \vek{z}^{\vek{\alpha}}-\vek{y}^{\vek{\alpha}}_f\right)\right)H_{\vek{\alpha}}.
\end{equation}
Obviously the projection $P_J$ commutes with the Kalman operator $K$ and
hence with its finite dimensional analogue $\vek{K}$.  One may actually
concisely write \refeq{eq:proj-lin-J} as
\begin{equation} \label{eq:proj-comm-K}
  P_J \vek{x}_a = P_J \vek{x}_f + P_J \vek{K}(\vhat{y} - \vek{y}_f) =
  P_J\vek{x}_f + \vek{K}(P_J\vhat{y} - P_J\vek{y}_f).
\end{equation}

Elements of the discretised space $\E{X}_{M,J} = \C{X}_M \otimes
\C{S}_J \subset \E{X}$ thus may be written fully expanded as 
$\sum_{m=1}^M \sum_{\vek{\alpha}\in \C{J}} x^{\vek{\alpha},m} \vrho_m 
H_{\vek{\alpha}}$.  The tensor representation is
$\tnb{x} := \sum_{\vek{\alpha}\in \C{J}} 
\vek{x}^{\vek{\alpha}} \otimes \vek{e}^{\vek{\alpha}}$, where the 
$\{ \vek{e}^{\vek{\alpha}} \}$ are the canonical basis in $\D{R}^J$,
and may be used to express \refeq{eq:proj-lin-J} or
\refeq{eq:proj-comm-K} succinctly as
\begin{equation}  \label{eq:proj-t}
 \tnb{x}_a = \tnb{x}_f+ \tnb{K}(\That{y}-\tnb{y}_f),
\end{equation}
again an equation between the tensor representations of some RVs,
where $\tnb{K} = \vek{K} \otimes \vek{I}$, with $\vek{K}$ from
\refeq{eq:iIX}.  Hence the update equation is naturally in a
tensorised form.  This is how the update can finally be computed in
the PCE representation without any sampling \citep{BvrAkJsOpHgm11},
\citep{opBvrAlHgm12}, \citep{bvrAlOpHgm12-a}, \citep{boulder:2011}.
Analogous statements hold for the forms of the update
\refeq{eq:n-deg-pol} with higher order terms $n>1$, and do not have to
be repeated here.  Let us remark that these updates go very seamlessly
with very efficient methods for sparse or low-rank approximation of
tensors, c.f.\ the monograph \citep{Hackbusch_tensor} and the
literature therein.  These methods are PCE-forms of the Bayesian
update, and in particular the \refeq{eq:proj-t}, because of its formal
affinity to the Kalman filter (KF), may be called the polynomial chaos
expansion based Kalman filter (PCEKF).

It remains to say how to compute the terms $\Hf{k}{H}$ in the update equation 
\refeq{eq:n-deg-pol}---or rather the terms in the defining \refeq{eq:cond-H}
in \refT{n-vers}---in this approach.  Given the PCEs of the RVs, this is actually
quite simple as any moment can be computed directly from the PCE
\citep{matthies6}, \citep{opBvrAlHgm12}, \citep{bvrAlOpHgm12-a}. 
A typical term $\ipj{y^{\vee k}} = \ipj{\text{Sym}(y^{\otimes k})}=
\text{Sym}(\ipj{y^{\otimes k}})$
in the operator matrix \refeq{eq:cond-H}, where $\vek{y}=\sum_\alpha\vek{y}^{\vek{\alpha}}
H_{\vek{\alpha}}(\vek{\theta})$, may be computed through
\begin{multline}  \label{eq:typ-zk}
  \ipj{\vek{y}^{\otimes k}} = \EXP{\bigotimes_{i=1}^k \sum_{\vek{\alpha}_i} \left( 
  \vek{y}^{\vek{\alpha}_i} H_{\vek{\alpha}_i}\right)} = \\
   \EXP{\sum_{\vek{\alpha}_1, \dots, \vek{\alpha}_k} \bigotimes_{i=1}^k 
   \vek{y}^{\vek{\alpha}_i} \prod_{i=1}^k H_{\vek{\alpha}_i}}  = 
   \sum_{\vek{\alpha}_1, \dots, \vek{\alpha}_k} \bigotimes_{i=1}^k \vek{y}^{\vek{\alpha}_i}
   \;\EXP{\prod_{i=1}^k H_{\vek{\alpha}_i}}
\end{multline}
As here the $H_{\vek{\alpha}}$ are \emph{polynomials}, the last
expectation in \refeq{eq:typ-zk} is finally over products of powers of
pairwise independent normalised Gaussian variables, which actually may
be done analytically \citep{holdenEtAl96}, \citep{malliavin97},
\citep{Janson1997}.  But some simplifications come from remembering
that $\vek{y}^0=\EXP{\vek{y}} = \bar{\vek{y}}$, $H_{\vek{0}}\equiv 1$,
the orthogonality relation $\bkt{H_{\vek{\alpha}}}{H_{\vek{\beta}}} =
\delta_{\vek{\alpha},\vek{\beta}}\, \vek{\alpha}!$, and that the
Hermite polynomials are an \emph{algebra}.  Hence
$H_{\vek{\alpha}}H_{\vek{\beta}} = \sum_{\vek{\gamma}}
c^{\vek{\gamma}}_{\vek{\alpha},\vek{\beta}}H_{\vek{\gamma}}$, where
the \emph{structure} coefficients
$c^{\vek{\gamma}}_{\vek{\alpha},\vek{\beta}}$ are known analytically
\citep{malliavin97}, \citep{matthies6}, \citep{opBvrAlHgm12},
\citep{bvrAlOpHgm12-a}.

Similarly, for a RV $R=R(x)$, for a typical right-hand-side term
$\ipj{R(x)\otimes y^{\vee k}} = \ipj{R\otimes \text{Sym}(y^{\otimes
    k})}$ in \refeq{eq:cond-H} with $\vek{R}=\sum_\beta
\vek{R}^{\vek{\beta}} H_{\vek{\beta}}(\vek{\theta})$ one has
\begin{equation}  \label{eq:typ-qzk}
   \ipj{R\otimes \text{Sym}(y^{\otimes k})} =
     \sum_{\vek{\beta}, \vek{\alpha}_1, \dots, \vek{\alpha}_k} \vek{R} \otimes
     \text{Sym}\left(\bigotimes_{i=1}^k \vek{y}^{\vek{\alpha}_i}\right)
   \;\EXP{H_{\vek{\beta}} \, \prod_{i=1}^k H_{\vek{\alpha}_i}}.
\end{equation}
As these relations may seem a bit involved --- they are actually just an
intricate combination of \emph{known} terms --- we show here how simple
they become for the case of the covariance needed in the linear update
formula \refeq{eq:iVII} or rather \refeq{eq:iIX}:
\begin{eqnarray}  \label{eq:cov-PCE-1}
   \vek{C}_{yy} &=%= \EXP{\tilde{\vekyz}} \otimes \tilde{\vek{y}}} 
   \sum_{\vek{\alpha}\in \C{N}, \vek{\alpha} \ne 0} (\vek{\alpha} !)\;
    \vek{y}^{\vek{\alpha}}\otimes \vek{y}^{\vek{\alpha}} 
   &\approx %\vek{C}_{P_J y} = \EXP{(P_J\vek{y}) \otimes (P_J\vek{y})} =
   \sum_{\vek{\alpha}\in \C{J}, \vek{\alpha} \ne 0} (\vek{\alpha} !)\;
   \vek{y}^{\vek{\alpha}}\otimes \vek{y}^{\vek{\alpha}},\\
  \vek{C}_{xy} &= %\EXP{\tilde{\vek{q}} \otimes \tilde{\vek{y}}} =
   \sum_{\vek{\alpha}\in \C{N}, \vek{\alpha} \ne 0} (\vek{\alpha} !)\;
    \vek{x}^{\vek{\alpha}}\otimes \vek{y}^{\vek{\alpha}} 
  &\approx \sum_{\vek{\alpha}\in \C{J}, \vek{\alpha} \ne 0} (\vek{\alpha} !)\;
   \vek{x}^{\vek{\alpha}}\otimes \vek{y}^{\vek{\alpha}}. 
    \label{eq:cov-PCE-2}
\end{eqnarray}

Looking for example at \refeq{eq:iIX} and our setup as explained in \refS{intro},
we see that the coefficients of $\vek{y}(\vek{x}_f)=\sum_\alpha\vek{y}_f^{\vek{\alpha}}
H_{\vek{\alpha}}$ have to be computed from those of
$\vek{x}_f=\sum_\beta\vek{x}_f^{\vek{\beta}} H_{\vek{\beta}}$.  This propagation
of uncertainty through the system is known as \emph{uncertainty quantification} (UQ),
e.g.\ \citep{matthies6} and the references therein.  For the sake of brevity,
we will not touch further on this subject, which nevertheless is the bedrock
on which the whole computational procedure is built.

We next concentrate in \refS{bayes-lin} on examples of updating with $\psi_{m}$ 
for the case $m=1$ in \refeq{eq:n-deg-pol}, whereas in \refS{bayes-non-lin}
an example for the case $m=2$ in \refeq{eq:n-deg-pol} will be shown.

%%  Marcov chain Monte Carlo
%\input{\thetext/mcmc}

%%  Ensemble Kalman Filter
%\input{\thetext/enkf}

%%  Projection onto PCE --- linear
%\input{\thetext/pce-bayes}

%  $Log: num-realis.tex,v $
%  Revision 3.5  2015/10/22 17:33:17  matthies
%  adjusting lines
%
%  Revision 3.4  2015/10/20 21:39:50  matthies
%  spell-check complete
%
%  Revision 3.3  2015/10/20 18:59:13  matthies
%  essentially finished
%
%  Revision 3.2  2015/10/20 09:06:52  matthies
%  adaptations to new variables
%
%  Revision 3.1  2015/10/05 11:25:32  matthies
%  started changes for new chapter for Adnan, Sarajevo, book - up to now only Greel letters
%
%  Revision 2.1  2013/12/17 22:31:03  hgm
%  arXiv version
%
%  Revision 1.2  2013/12/15 00:46:37  hgm
%  almost finished, Sat night
%
%  Revision 1.1  2013/12/14 16:46:01  hgm
%  initial check in from 12_Sarajevo
%
%
%

%%% Local Variables: 
%%% mode: latex
%%% TeX-master: "../NonLinBU"
%%% End: 

% RCSID:       $Id: bayes-lin.tex,v 3.4 2015/10/22 17:33:35 matthies Exp $
% Author:      $Author: matthies $
% Contact:     wire@tu-bs.de
% =================================
%% texfile{
%%  AUTHOR    = "$Author: matthies $",
%%  VERSION   = "$Revision: 3.4 $",
%%  DATE      = "$Date: 2015/10/22 17:33:35 $",
%%  FILENAME  = "$RCSfile: bayes-lin.tex,v $"}
%
% =================================

\section{The linear Bayesian update} \label{S:bayes-lin}
All the examples in this \refS{bayes-lin} have been computed with the case
$m=1$ of up to linear terms in \refeq{eq:n-deg-pol}, i.e.\ this is the
LBU with PCEKF.
\begin{figure}[!ht]
\centering
 \includegraphics[width=0.8\textwidth,height=0.35\textheight]{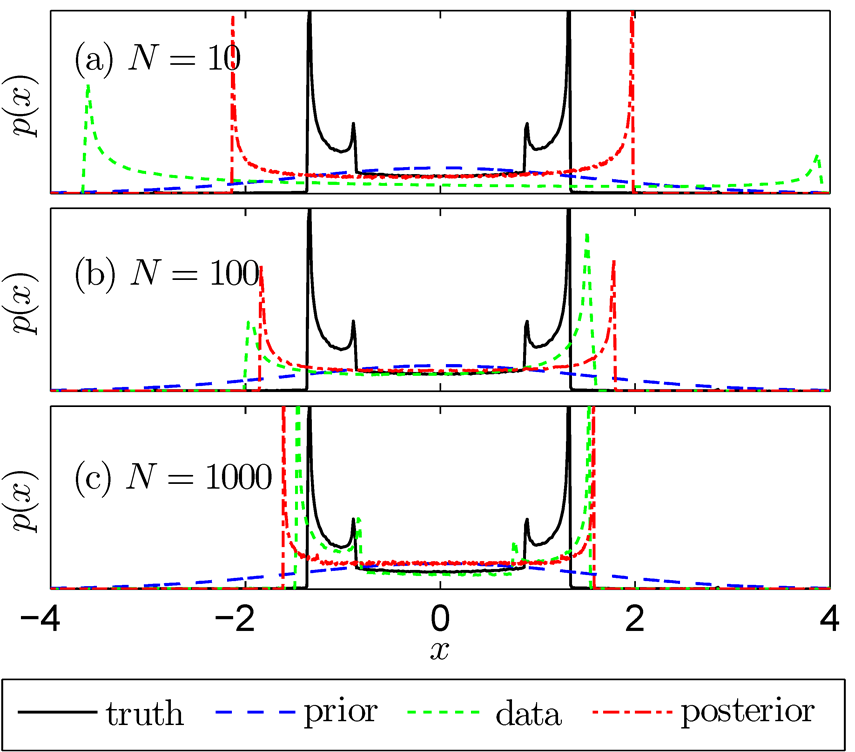}
 \caption{pdfs for linear Bayesian update (LBU), from \citep{opBvrAlHgm12}}
\label{F:exp-B-1}
\end{figure}
As the traditional Kalman filter is highly geared towards Gaussian
distributions \citep{Kalman}, and also its Monte Carlo variant EnKF
which was mentioned in \refS{num-real} tilts towards Gaussianity,
we start with a case---already described in \citep{opBvrAlHgm12}---where
the the quantity to be identified has a strongly
non-Gaussian distribution, shown in black---the `truth'---in \refig{exp-B-1}.
The operator describing the system is the identity---we compute the quantity
directly, but there is a Gaussian measurement error.  The `truth' was
represented as a $12^{\text{th}}$ degree PCE.
We use the methods as described in \refS{num-real}, and here in particular
the \refeq{eq:iIX} and \refeq{eq:proj-t}, the PCEKF.

The update is repeated several times (here ten times) with new
measure\-ments---see \refig{exp-B-1}.  The task is here to identify the
distribution labelled as `truth' with ten updates of $N$ samples
(where $N=10, 100, 1000$ was used), and we start with a very broad
Gaussian prior (in blue).  Here we see the ability of the polynomial
based LBU, the PCEKF, to identify highly non-Gaussian distributions,
the posterior is shown in red and the pdf estimated from the samples
in green; for further details see \citep{opBvrAlHgm12}.

\begin{figure}[!ht]
\centering
 \includegraphics[width=0.9\textwidth,height=0.35\textheight]{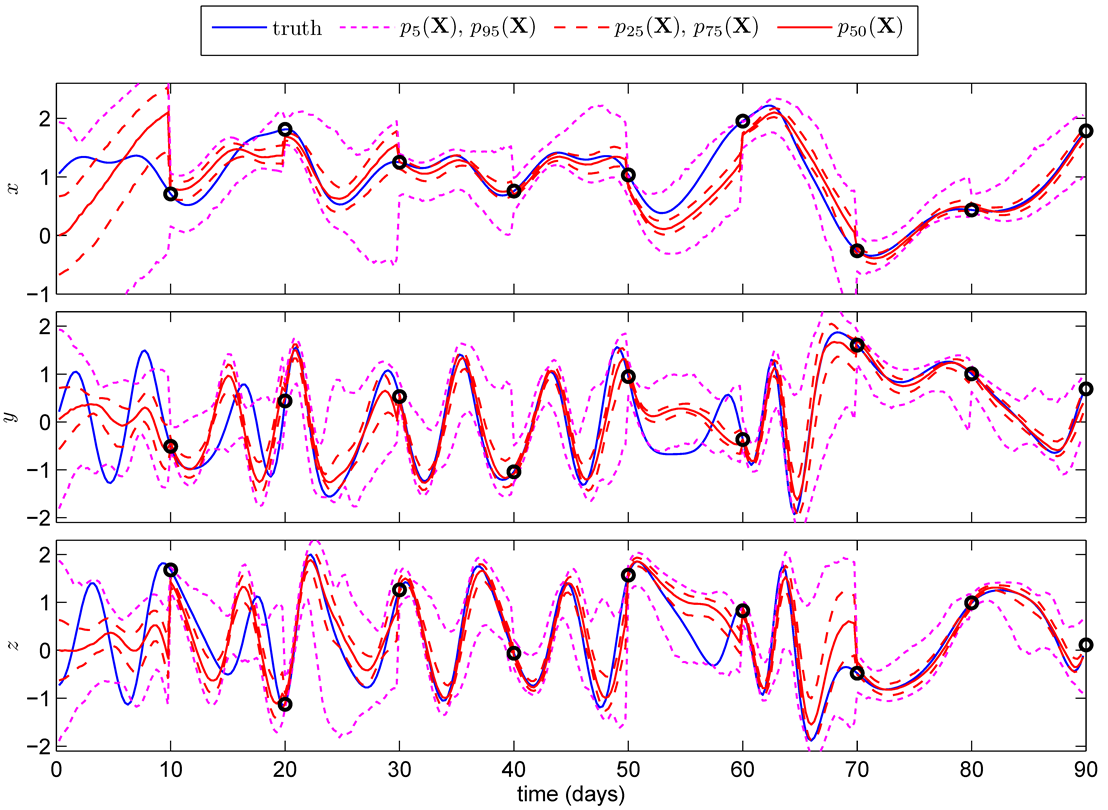}
 \caption{Time evolution of Lorenz-84 state and uncertainty with the LBU, from \citep{opBvrAlHgm12}}
\label{F:exp-B-2}
\end{figure}
The next example is also from \citep{opBvrAlHgm12}, where the system
is the well-known Lorenz-84 chaotic model, a system of three nonlinear
ordinary differential equations operating in the chaotic regime. This
is truly an example along the description of \refeq{eq:dyn-l} and
\refeq{eq:dyn-ml} in \refSS{bayes-setting}.  Remember that this was
originally a model to describe the evolution of some amplitudes of a
spherical harmonic expansion of variables describing world climate.
As the original scaling of the variables has been kept, the time axis
in \refig{exp-B-2} is in \emph{days}.  Every ten days a noisy
measurement is performed and the state description is updated.  In
between the state description evolves according to the chaotic dynamic
of the system.  One may observe from \refig{exp-B-2} how the
uncertainty---the width of the distribution as given by the quantile
lines---shrinks every time a measurement is performed, and then
increases again due to the chaotic and hence noisy dynamics.  Of
course, we did not really measure world climate, but rather simulated
the `truth' as well, i.e.\ a \emph{virtual} experiment, like the
others to follow.  More details may be found in \citep{opBvrAlHgm12}
and the references therein.

\begin{figure}[!ht]
\centering
\begin{minipage}{.4\textwidth}
  \centering
  \includegraphics[width=.99\linewidth]{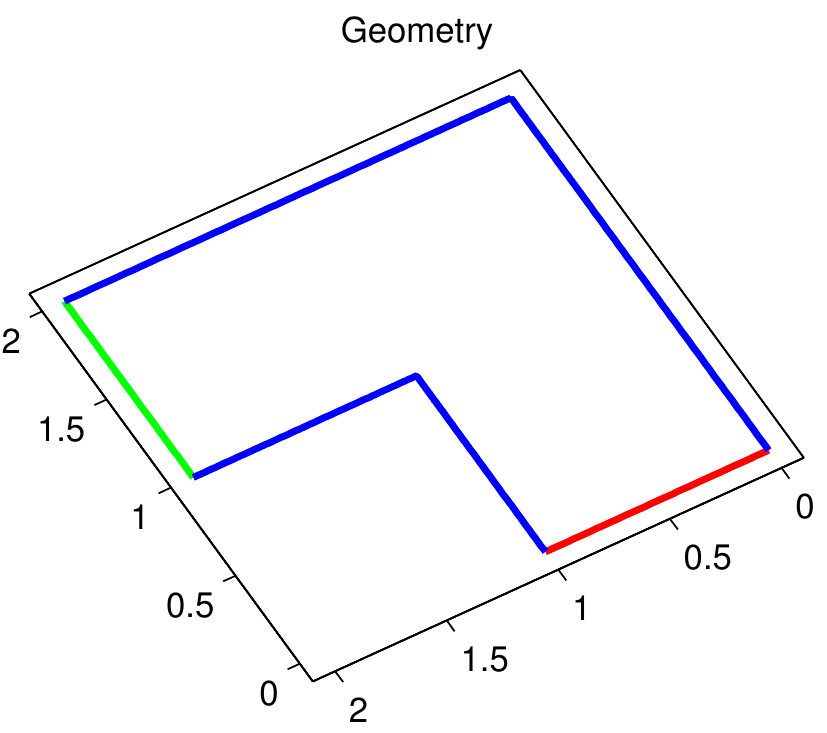}
%  \captionof{figure}{Diffusion domain, from \citep{bvrAlOpHgm12-a}}
  \caption{Diffusion domain, from \citep{bvrAlOpHgm12-a}}
  \label{F:exp-B-3}
\end{minipage}%
\begin{minipage}{.6\textwidth}
  \centering
  \includegraphics[width=.99\linewidth]{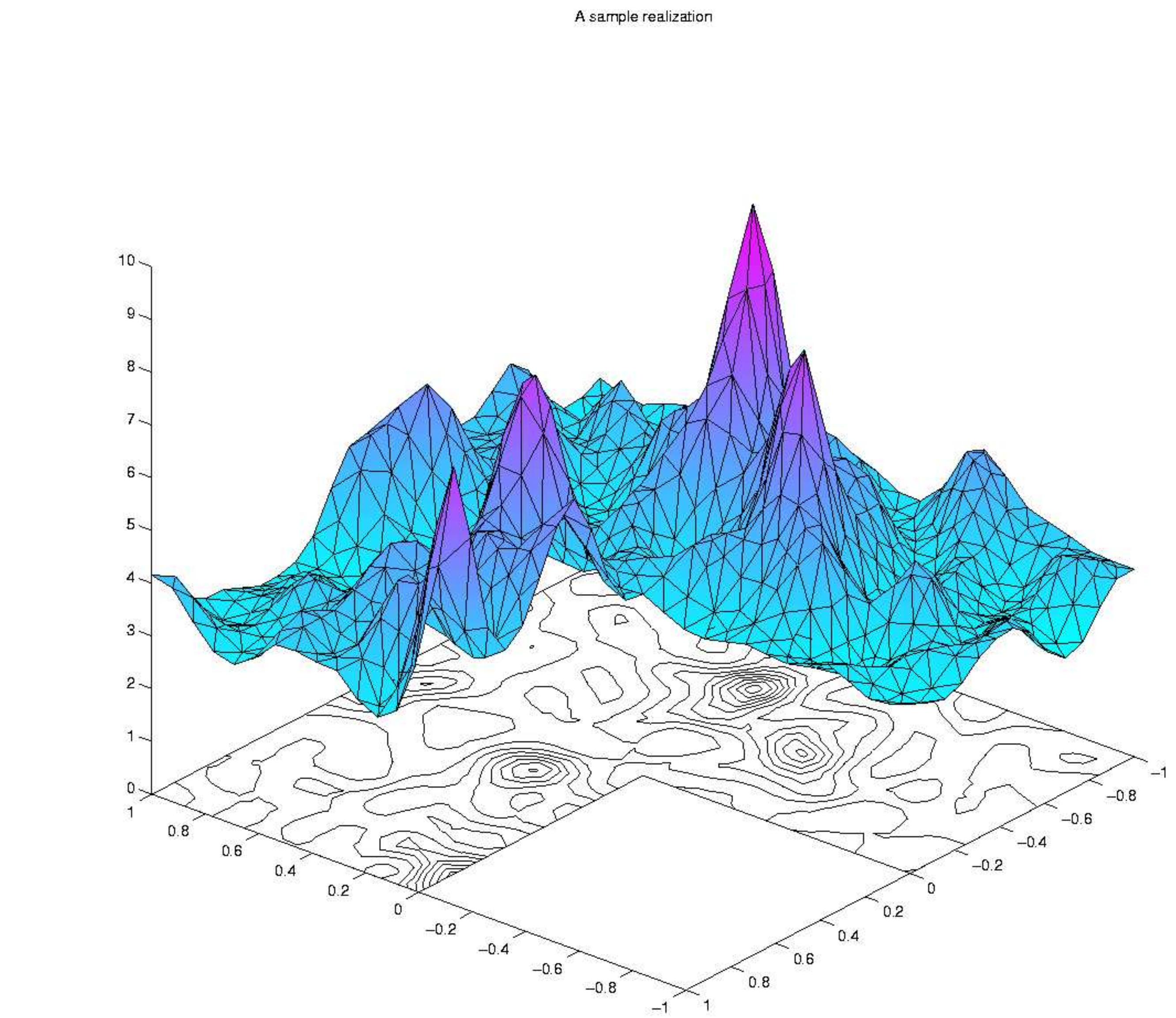}
%  \captionof{figure}{Conductivity field, from \citep{bvrAlOpHgm12-a}}  
  \caption{Conductivity field, from \citep{bvrAlOpHgm12-a}}  
  \label{F:exp-B-4}
\end{minipage}
\end{figure}
From \citep{bvrAlOpHgm12-a} we take the example shown in
\refig{exp-B-3}, a linear stationary diffusion equation on an L-shaped
plane domain as alluded to in \refS{intro}.  The diffusion coefficient
$\kappa$ in \refeq{eq:I-c} is to be identified.  As argued in \citep{BvrAkJsOpHgm11},
it is better to work with $q = \log \kappa$ as the diffusion coefficient has
to be positive, but the results are shown
in terms of $\kappa$.

% \begin{figure}[!ht]
% \centering
% \begin{minipage}{0.5\textwidth}
%   \centering
%   \includegraphics[width=0.7\linewidth]{./figures/first_mesh}
%   \captionof{figure}{447 measurement patches, from \citep{bvrAlOpHgm12-a}}
%   \label{F:exp-B-5}
% \end{minipage}%
% \begin{minipage}{0.5\textwidth}
%   \centering
%   \includegraphics[width=0.7\linewidth]{./figures/fourth_mesh}
%   \captionof{figure}{10 measurement patches, from \citep{bvrAlOpHgm12-a}}  
%   \label{F:exp-B-6}
% \end{minipage}
% \end{figure}
One possible realisation of the diffusion coefficient
is shown in \refig{exp-B-4}.  More realistically, one should assume that
$\kappa$ is a symmetric positive definite tensor field, unless one knows that
the diffusion is \emph{isotropic}.  Also in this case one should do the updating
on the logarithm.  For the sake of simplicity we stay with the scalar case,
as there is no principal novelty in the non-isotropic case.
The virtual experiments use different right-hand-sides $f$ in \refeq{eq:I-c},
and the measurement is the observation of the solution $u$ averaged over little patches.
%two of these arrangements are shown in \refig{exp-B-5} and \refig{exp-B-6}.

\begin{figure}[!ht]
\centering
\begin{minipage}{0.48\textwidth}
  \centering
  \includegraphics[width=0.99\linewidth]{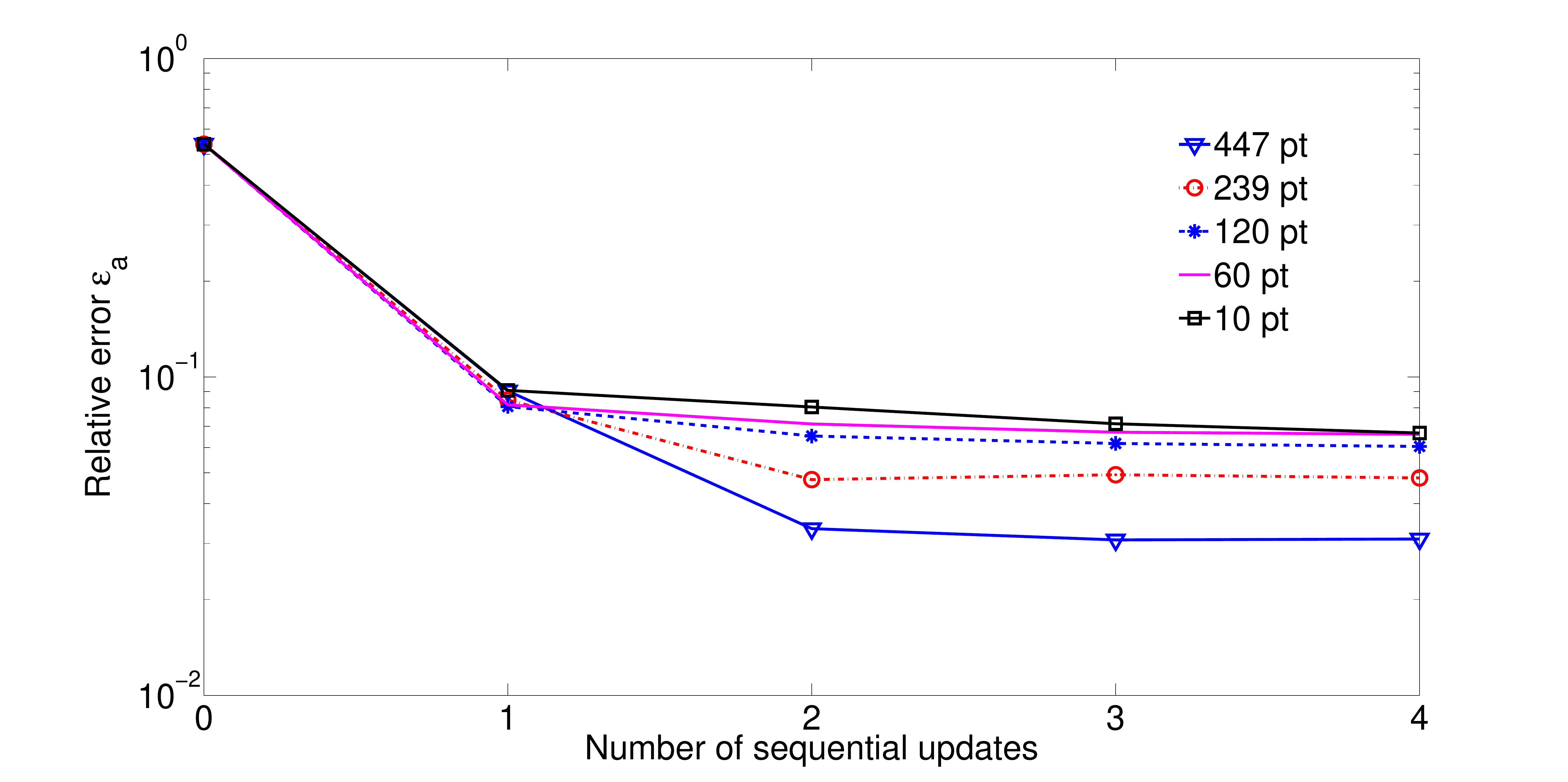}
%  \captionof{figure}{Convergence of identification, from \citep{bvrAlOpHgm12-a}}
  \caption{Convergence of identification, from \citep{bvrAlOpHgm12-a}}
  \label{F:exp-B-7}
\end{minipage}%
\hfill
\begin{minipage}{0.48\textwidth}
  \centering
  \includegraphics[width=0.99\linewidth,height=0.17\textheight]{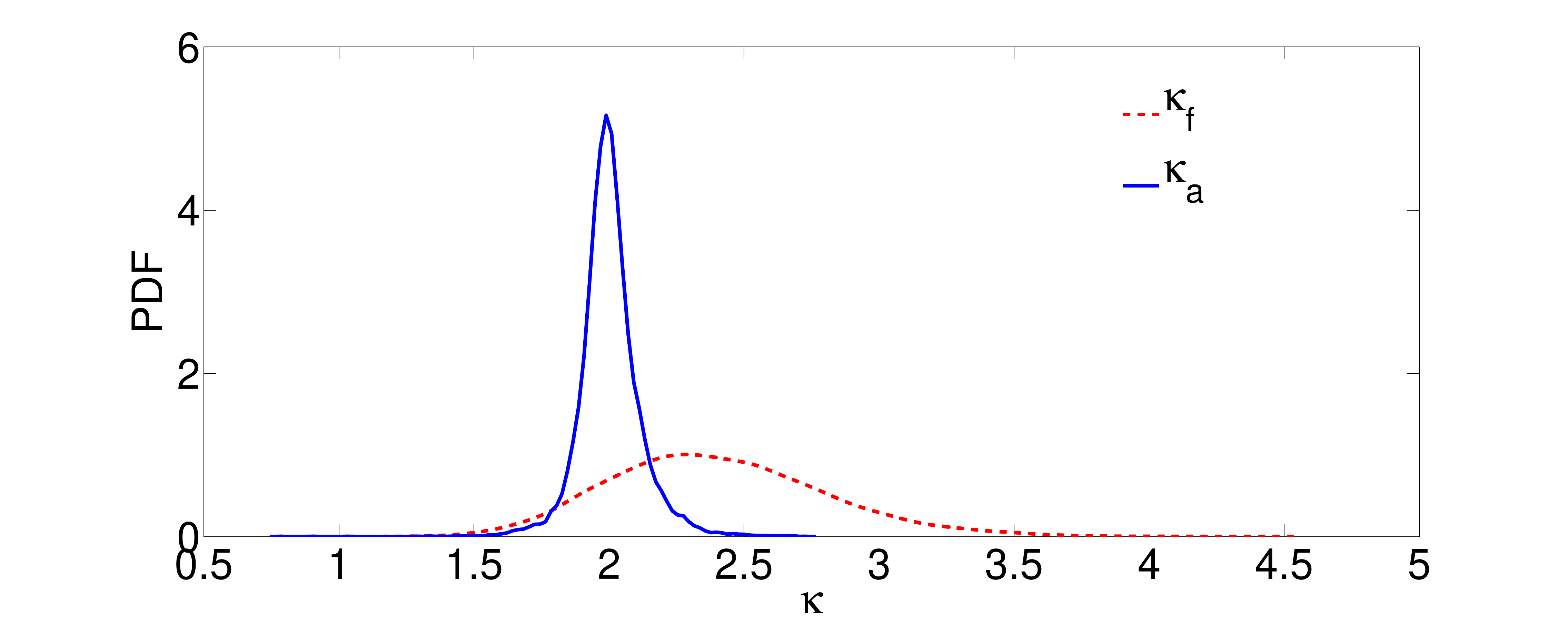}
%  \captionof{figure}{Prior and posterior, from \citep{bvrAlOpHgm12-a}}  
  \caption{Prior and posterior, from \citep{bvrAlOpHgm12-a}}  
  \label{F:exp-B-8}
\end{minipage}
\end{figure}
In \refig{exp-B-7} one may observe the decrease of the error with successive
updates, but due to measurement error and insufficient information from just
a few patches, the curves level off, leaving some residual uncertainty.
The pdfs of the diffusion coefficient at some point in the domain before
and after the updating is shown in \refig{exp-B-8}, the `true' value at
that point was $\kappa=2$.
Further details can be found in \citep{bvrAlOpHgm12-a}.

\section{The nonlinear Bayesian update} \label{S:bayes-non-lin}
%
% \begin{figure}[!ht]
% \centering
%  \includegraphics[width=0.9\textwidth,height=0.2\textheight]{./figures/apri_apost_comparison_after1update}
%  \caption{Linear measurement: prior and posterior after one update}
% \label{F:exp-NB-5}
% \end{figure}
In this Section we want to show a computation with the case $m=2$ of
up to quadratic terms in $\psi_m$ in \refeq{eq:n-deg-pol}.  We go back
to the example of the chaotic Lorentz-84 \citep{opBvrAlHgm12} model
already shown in \refS{bayes-lin}, from \refeq{eq:dyn-l} and
\refeq{eq:dyn-ml} in \refSS{bayes-setting}.  This kind of experiment
has several advantages but at the same time also challenges for
identification procedures: it has only a three-dimensional state
space, these are the uncertain `parameters', i.e.\ $\vek{x} =
(x_1,x_2,x_3) = (x, y ,z) \in\C{X}=\D{R}^3$, the corresponding
operator $A$ resp.\ $f$ in the abstract \refeq{eq:I} resp.\
\refeq{eq:dyn-l} is sufficiently nonlinear to make the problem
difficult, and adding to this we operate the equation in its chaotic
regime, so that new uncertainty from the numerical computation is
added between measurements.

\begin{figure}[!ht]
\centering
 \includegraphics[width=0.9\textwidth,height=0.2\textheight]{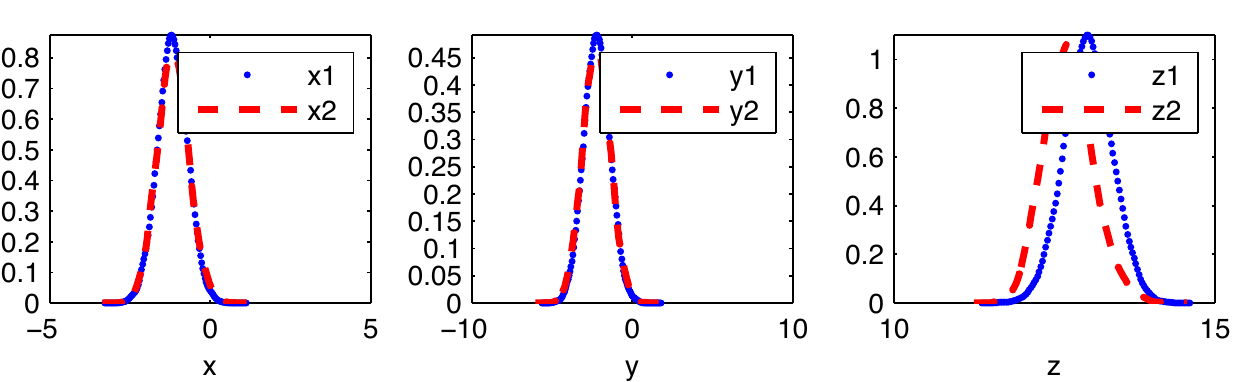}
 \caption{Linear measurement: Comparison posterior for LBU ($m=1$) and 
 QBU ($m=2$) after one update}
\label{F:exp-NB-5.5}
\end{figure}
As a first set of experiments we take the measurement operator to be
linear in $\vek{x}$, i.e.\ we can observe the \emph{whole} state
directly.  At the moment we consider updates after each day---whereas
in \refS{bayes-lin} the updates were performed every 10 days.  The
update is done once with the linear Bayesian update (LBU), and again
with a \emph{quadratic} nonlinear BU (QBU) with $m=2$.  The results
for the posterior pdfs are given in \refig{exp-NB-5.5}, where the
linear update is dotted in blue, and the full red line is the
quadratic QBU; there is hardly any difference between the two, most
probably indicating that the LBU is already very accurate.
% \begin{figure}[!ht]
% \centering
%  \includegraphics[width=0.9\textwidth,height=0.2\textheight]{./figures/apost_comparison_p1_p2_after2nd_update}
%  \caption{Linear measurement: Comparison posterior for LBU ($m=1$) and 
%  QBU ($m=2$) after second update}
% \label{F:exp-NB-6}
% \end{figure}

As the differences between LBU and QBU were small --- we take this as
an indication that the LBU is not too inaccurate an approximation to
the conditional expectation --- we change the experiment and take a
nonlinear measurement function, which is now cubic: $h(\vek{x})
=(x^3,y^3,z^3)$.  We now observe larger differences between LBU and
QBU.
%\begin{figure}[!ht]
%\centering
% \includegraphics[width=0.9\textwidth,height=0.35\textheight]{./figures/apri_apost_pdfs_comparison}
% \caption{cubic: apri\_apost\_pdfs\_comparison}
%\label{F:exp-NB-7}
%\end{figure}
%
%\begin{figure}[!ht]
%\centering
% \includegraphics[width=0.9\textwidth,height=0.2\textheight]{./figures/apri_apost_comparison}
% \caption{cubic: apri\_apost\_comparison}
%\label{F:exp-NB-8}
%\end{figure}
\begin{figure}[!ht]
\centering
 \includegraphics[width=0.9\textwidth,height=0.2\textheight]{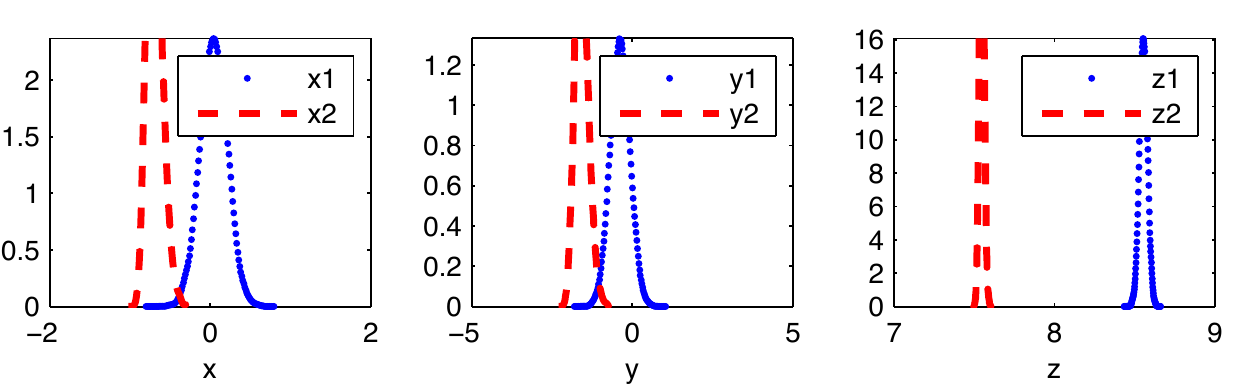}
 \caption{Cubic measurement: Comparison posterior for LBU ($m=1$) and 
 QBU ($m=2$) after one update}
\label{F:exp-NB-9}
\end{figure}

These differences in posterior pdfs after one update may be gleaned 
from \refig{exp-NB-9}, and they are indeed larger
than in the linear case \refig{exp-NB-5.5}, due to the strongly nonlinear
measurement operator, showing that the QBU may provide much more
accurate tracking of the state, especially for non-linear observation
operators.
% \begin{figure}[!ht]
% \centering
%  \includegraphics[width=0.9\textwidth,height=0.3\textheight]{./figures/3updates_x2signx}
%  \caption{Partial state trajectory with uncertainty and three updates}
% \label{F:exp-NB-9.5}
% \end{figure}

% As the cubic is quite a strong nonlinearity, we performed a set of experiments
% where the measurement function is $h(\vek{x}) = (x |x|,y |y|,z |z|)$;
% only a quadratic nonlinearity..
% \begin{figure}[!ht]
% \centering
%  \includegraphics[width=0.9\textwidth,height=0.2\textheight]{./figures/apri_apost_comparison-2}
%  \caption{Quadratic measurement: Comparison posterior for LBU ($m=1$) and 
%  QBU ($m=2$) after one update}
% \label{F:exp-NB-10}
% \end{figure}
%
% The results for the $2$-nd update are displayed for the posterior pdfs
% in \refig{exp-NB-11} for LBU and QBU for one component, and the difference are
% considerable again.
% \begin{figure}[!ht]
% \centering
% \includegraphics[width=0.9\textwidth,height=0.2\textheight]{./figures/p1_p2_after_2ndupdate}
% \caption{$x^2 \sign(x)$: p1\_p2\_after\_2ndupdate}
% \label{F:exp-NB-11}
% \end{figure}

\begin{figure}[!ht]
\centering
\begin{minipage}{0.47\textwidth}
  \centering
  \includegraphics[width=0.99\linewidth,height=0.26\textheight]{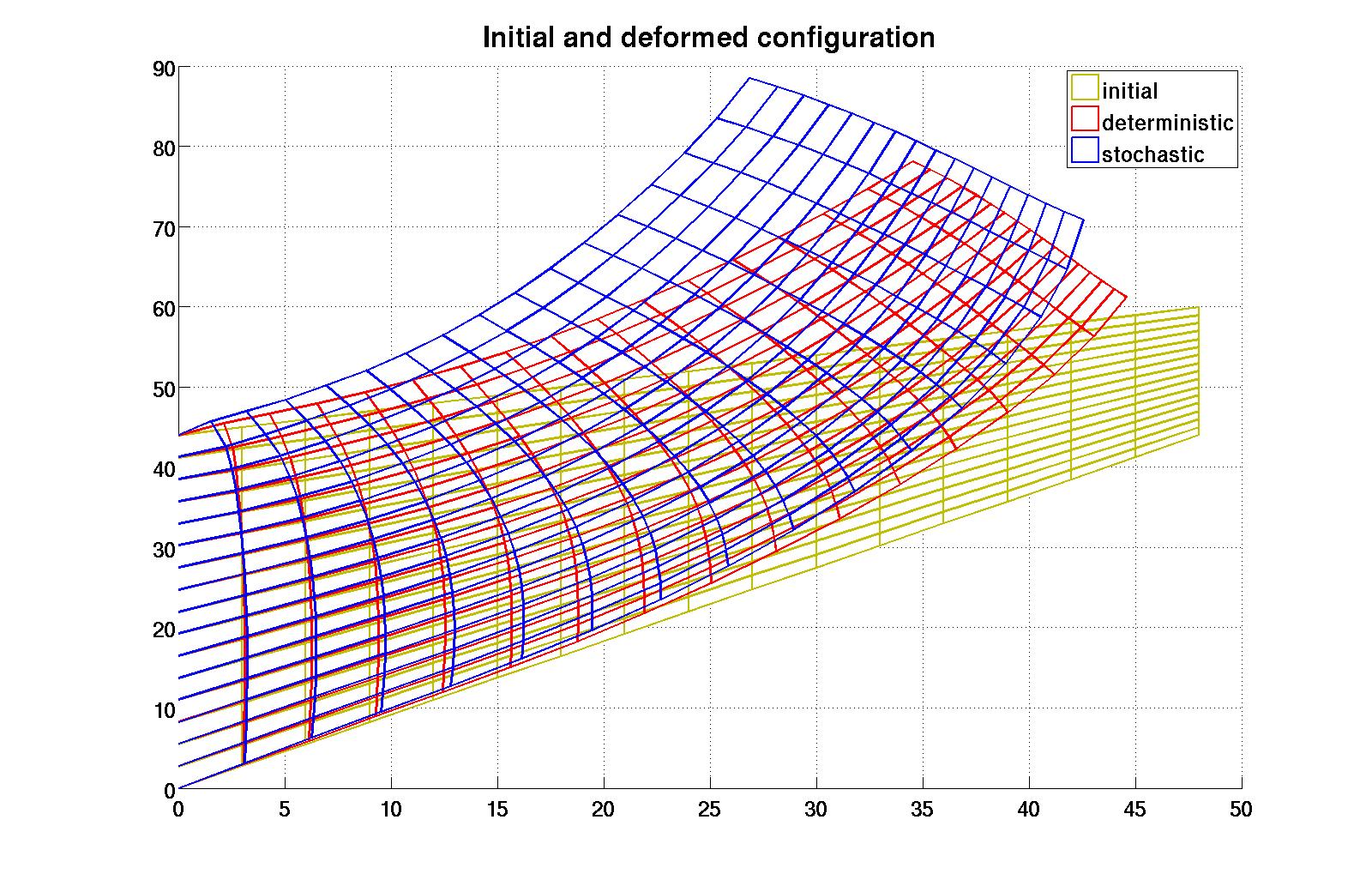}
%  \captionof{figure}{Deformations, from \citep{BvrAkJsOpHgm11}, \citep{rosic2013hgm}}
  \caption{Deformations, from \citep{BvrAkJsOpHgm11}, \citep{rosic2013hgm}}
  \label{F:exp-B-9}
\end{minipage}%
\hfill
\begin{minipage}{0.47\textwidth}
  \centering
  \includegraphics[width=0.99\linewidth]{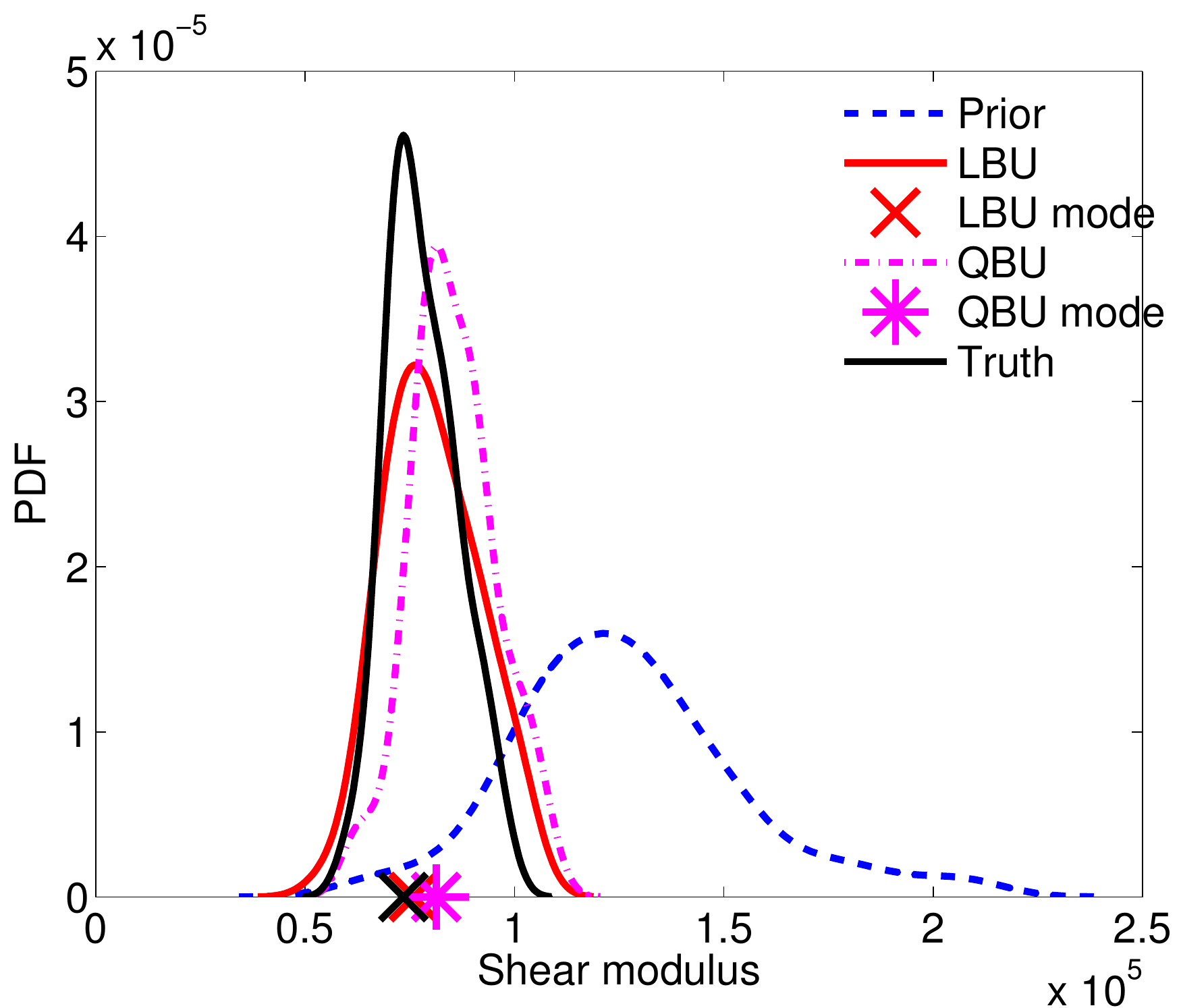}
%  \captionof{figure}{LBU and QBU for the shear modulus}
  \caption{LBU and QBU for the shear modulus}
  \label{F:exp-B-10}
\end{minipage}
\end{figure}
As a last example we take a strongly nonlinear and also
non-smooth situation, namely elasto-plasticity with linear hardening
and large deformations and a \emph{Kirchhoff-St.\ Venant} elastic
material law \citep{BvrAkJsOpHgm11}, \citep{rosic2013hgm}.  This
example is known as \emph{Cook's membrane}, and is shown in
\refig{exp-B-9} with the undeformed mesh (initial), the deformed one
obtained by computing with average values of the elasticity and
plasticity material constants (deterministic), and finally the average
result from a stochastic forward calculation of the probabilistic
model (stochastic), which is described by a variational inequality
\citep{rosic2013hgm}.

The shear modulus $G$, a random field and not a deterministic value in
this case, has to be identified, which is made more difficult by the
non-smooth non-linearity.  In \refig{exp-B-10} one may see the `true'
distribution at one point in the domain in an unbroken black line,
with the mode --- the maximum of the pdf --- marked by a black cross
on the abscissa, whereas the prior is shown in a dotted blue line.
The pdf of the LBU is shown in an unbroken red line, with its mode
marked by a red cross, and the pdf of the QBU is shown in a broken
purple line with its mode marked by an asterisk.  Again we see a
difference between the LBU and the QBU.  But here a curious thing
happens; the mode of the LBU-posterior is actually closer to the mode
of the `truth' than the mode of the QBU-posterior.  This means that
somehow the QBU takes the prior more into account than the LBU, which
is a kind of overshooting which has been observed at other occasions.
On the other hand the pdf of the QBU is narrower --- has less
uncertainty --- than the pdf of the LBU.

%  $Log: bayes-non-lin.tex,v $
%  Revision 3.3  2015/10/22 17:33:49  matthies
%  adjusting lines
%
%  Revision 3.2  2015/10/21 20:25:44  matthies
%  finished
%
%  Revision 3.1  2015/10/21 12:39:43  matthies
%  some corrections for new paper
%
%  Revision 2.1  2013/12/17 22:31:18  hgm
%  arXiv version
%
%  Revision 1.4  2013/12/15 16:44:32  hgm
%  final
%
%  Revision 1.3  2013/12/15 00:46:19  hgm
%  just started, Sat night
%
%  Revision 1.2  2013/12/14 16:47:27  hgm
%  not yet started, corrected eq: no
%
%  Revision 1.1  2013/12/11 23:29:54  hgm
%  initial check in, not final
%
%
%

%%% Local Variables: 
%%% mode: latex
%%% TeX-master: "../NonLinBU"
%%% End: 

% RCSID:       $Id: conclusion.tex,v 3.1 2015/10/21 20:32:26 matthies Exp $
% Author:      $Author: matthies $
% Contact:     wire@tu-bs.de
% =================================
%% texfile{
%%  AUTHOR    = "$Author: matthies $",
%%  VERSION   = "$Revision: 3.1 $",
%%  DATE      = "$Date: 2015/10/21 20:32:26 $",
%%  FILENAME  = "$RCSfile: conclusion.tex,v $"}
%
% =================================

\section{Conclusion} \label{S:concl}
The connection between inverse problems and uncertainty quantification
was shown.  An abstract model of a system was introduced, together
with a measurement operator, which provides a possibility to predict
--- in a probabilistic sense --- a measurement.  The framework chosen
is that of Bayesian analysis, where uncertain quantities are modelled
as random variables.  New information leads to an update of the
probabilistic description via Bayes's rule.

After elaborating on the --- often not well-known --- connection
between conditional probabilities as in Bayes's rule and conditional
expectation, we set out to compute and --- necessarily --- approximate
the conditional expectation.  As a polynomial approximation was
chosen, there is the choice up to which degree one should go.  The
case with up to linear terms --- the linear Bayesian update (LBU) ---
is best known and intimately connected with the well-known Kalman
filter.  We call this update the Gauss-Markov-Kalman filter.  In
addition, we show how to compute approximations of higher order, in
particular the quadratic Bayesian update (QBU).

There are several possibilities on how one may choose a numerical
realisation of these theoretical concepts, and we decided on
functional or spectral approximations.  It turns out that this
approach goes very well with recent very efficient approximation
methods building on separated or so-called low-rank tensor
approximations.

Starting with the linear Bayesian update, a series of examples of
increasing complexity is shown.  The method works well in all cases.
Some examples are then chosen to show the nonlinear or rather
quadratic Bayesian update, where we go up to quadratic terms.  A
series of experiments is chosen with different measurement operators,
which have quite a marked influence on whether the linear and
quadratic update are close to each other.

%  $Log: conclusion.tex,v $
%  Revision 3.1  2015/10/21 20:32:26  matthies
%  finished for book chapter
%
%  Revision 2.1  2013/12/17 22:31:39  hgm
%  arXiv version
%
%  Revision 1.1  2013/12/11 23:30:53  hgm
%  initial check in, not final
%
%
%
%

%%% Local Variables: 
%%% mode: latex
%%% TeX-master: "../NonLinBU"
%%% End: 

% ============================================================================
% Bibliography
% The BibTeX files come from my external common bibtex repository
% multiple .bib-files are simply concatenated (without whitespace!)
%\appendix
%\cleardoublepage 

%\begin{thebibliography}{99}

%\bibliography{\thebib/jabbrevlong,\thebib/matthies_BU_paper-1,\thebib/phys_D,\thebib/fa,\thebib/risk,\thebib/fuq-new,\thebib/highdim,\thebib/filtering}

\begin{thebibliography}{10}

\bibitem{Blanchard2010a}
E.~D. Blanchard, A.~Sandu, and C.~Sandu, \emph{A polynomial chaos-based
  {K}alman filter approach for parameter estimation of mechanical systems},
  Journal of Dynamic Systems, Measurement, and Control \textbf{132} (2010),
  no.~6, 061404, \href {http://dx.doi.org/10.1115/1.4002481}
  {\path{doi:10.1115/1.4002481}}.

\bibitem{Bobrowski2006/087}
A.~Bobrowski, \emph{Functional analysis for probability and stochastic
  processes}, Cambridge University Press, Cambridge, 2005.

\bibitem{bosq2000}
D.~Bosq, \emph{Linear processes in function spaces. theory and applications.},
  Lecture Notes in Statistics, vol. 149, Springer, Berlin, 2000.

\bibitem{bosq2007}
\bysame, \emph{General linear processes in {H}ilbert spaces and prediction},
  Journal of Statistical Planning and Inference \textbf{137} (2007), 879--894,
  \href {http://dx.doi.org/10.1016/j.jspi.2006.06.014}
  {\path{doi:10.1016/j.jspi.2006.06.014}}.

\bibitem{Engl2000}
H.~W. Engl, M.~Hanke, and A.~Neubauer, \emph{Regularization of inverse
  problems}, Kluwer, Dordrecht, 2000.

\bibitem{Evensen2009}
G.~Evensen, \emph{Data assimilation --- the ensemble {Kalman} filter},
  Springer, Berlin, 2009.

\bibitem{Evensen2009a}
\bysame, \emph{The ensemble {Kalman} filter for combined state and parameter
  estimation}, IEEE Control Systems Magazine \textbf{29} (2009), 82--104, \href
  {http://dx.doi.org/10.1109/MCS.2009.932223}
  {\path{doi:10.1109/MCS.2009.932223}}.

\bibitem{GalvisSarkis:2012}
J.~Galvis and M.~Sarkis, \emph{Regularity results for the ordinary product
  stochastic pressure equation}, SIAM Journal on Mathematical Analysis
  \textbf{44} (2012), 2637--2665, \href {http://dx.doi.org/10.1137/110826904}
  {\path{doi:10.1137/110826904}}.

\bibitem{Gamerman06}
D.~Gamerman and H.~F. Lopes, \emph{{M}arkov {C}hain {M}onte {C}arlo: Stochastic
  simulation for {B}ayesian inference}, Chapman \& Hall, Boca Raton, FL, 2006.

\bibitem{ghanemSpanos91}
R.~Ghanem and P.~D. Spanos, \emph{Stochastic finite elements---a spectral
  approach}, Springer, Berlin, 1991.

\bibitem{Goldstein2007}
M.~Goldstein and D.~Wooff, \emph{Bayes linear statistics---theory and methods},
  Wiley Series in Probability and Statistics, John Wiley \& Sons, Chichester,
  2007.

\bibitem{Hackbusch_tensor}
W.~Hackbusch, \emph{Tensor spaces and numerical tensor calculus}, Springer,
  Berlin, 2012.

\bibitem{hida}
T.~Hida, H.~H. Kuo, J.~Potthoff, and L.~Streit, \emph{White noise---an infinite
  dimensional calculus}, Kluwer, Dordrecht, 1999.

\bibitem{holdenEtAl96}
H.~Holden, B.~{\O}ksendal, J.~Ub{\o}e, and T.-S. Zhang, \emph{Stochastic
  partial differential equations}, Birkh\"auser, Basel, 1996.

\bibitem{Janson1997}
S.~Janson, \emph{{G}aussian {H}ilbert spaces}, Cambridge Tracts in Mathematics,
  129, Cambridge University Press, Cambridge, 1997.

\bibitem{jaynes03}
E.~T. Jaynes, \emph{Probability theory, the logic of science}, Cambridge
  University Press, Cambridge, 2003.

\bibitem{Kalman}
R.~E. K\'alm\'an, \emph{A new approach to linear filtering and prediction
  problems}, Transactions of the ASME---J. of Basic Engineering (Series D)
  \textbf{82} (1960), 35--45.

\bibitem{Kucherova10}
A.~Ku\v{c}erov\'{a} and H.~G. Matthies, \emph{Uncertainty updating in the
  description of heterogeneous materials}, Technische Mechanik \textbf{30}
  (2010), no.~1--3, 211--226.

\bibitem{LawLitHGM15}
K.~H.~J. Law, A.~Litvinenko, and H.~G. Matthies, \emph{Nonlinear evolution,
  observation, and update}, 2015.

\bibitem{Luenberger1969}
D.~G. Luenberger, \emph{Optimization by vector space methods}, John Wiley \&
  Sons, Chichester, 1969.

\bibitem{Madras-Fields:2002}
N.~Madras, \emph{Lectures on {M}onte {C}arlo methods}, American Mathematical
  Society, Providence, RI, 2002.

\bibitem{malliavin97}
P.~Malliavin, \emph{Stochastic analysis}, Springer, Berlin, 1997.

\bibitem{Marzouk2007}
Y.~M. Marzouk, H.~N. Najm, and L.~A. Rahn, \emph{Stochastic spectral methods
  for efficient {Bayesian} solution of inverse problems}, Journal of
  Computational Physics \textbf{224} (2007), no.~2, 560--586, \href
  {http://dx.doi.org/10.1016/j.jcp.2006.10.010}
  {\path{doi:10.1016/j.jcp.2006.10.010}}.

\bibitem{matthies6}
H.~G. Matthies, \emph{Uncertainty quantification with stochastic finite
  elements}, Encyclopaedia of Computational Mechanics (E.~Stein, R.~de~Borst,
  and T.~J.~R. Hughes, eds.), John Wiley \& Sons, Chichester, 2007, \href
  {http://dx.doi.org/10.1002/0470091355.ecm071}
  {\path{doi:10.1002/0470091355.ecm071}}.

\bibitem{matthiesKeese05cmame}
H.~G. Matthies and A.~Keese, \emph{Galerkin methods for linear and nonlinear
  elliptic stochastic partial differential equations}, Computer Methods in
  Applied Mechanics and Engineering \textbf{194} (2005), no.~12-16, 1295--1331.
  \MR{MR2121216 (2005j:65146)}

\bibitem{boulder:2011}
H.~G. Matthies, A.~Litvinenko, O.~Pajonk, B.~V. Rosi\'c, and E.~Zander,
  \emph{Parametric and uncertainty computations with tensor product
  representations}, Uncertainty Quantification in Scientific Computing (Berlin)
  (A.~Dienstfrey and R.~Boisvert, eds.), IFIP Advances in Information and
  Communication Technology, vol. 377, Springer, 2012, pp.~139--150, \href
  {http://dx.doi.org/10.1007/978-3-642-32677-6}
  {\path{doi:10.1007/978-3-642-32677-6}}.

\bibitem{moselhyYMarz:2011}
T.~A. Moselhy and Y.~M. Marzouk, \emph{{B}ayesian inference with optimal maps},
  Journal of Computational Physics \textbf{231} (2012), 7815--7850, \href
  {http://dx.doi.org/10.1016/j.jcp.2012.07.022}
  {\path{doi:10.1016/j.jcp.2012.07.022}}.

\bibitem{opBvrAlHgm12}
O.~Pajonk, B.~V. Rosi\'c, A.~Litvinenko, and H.~G. Matthies, \emph{A
  deterministic filter for non-{G}aussian {B}ayesian estimation ---
  applications to dynamical system estimation with noisy measurements}, Physica
  D \textbf{241} (2012), 775--788, \href
  {http://dx.doi.org/10.1016/j.physd.2012.01.001}
  {\path{doi:10.1016/j.physd.2012.01.001}}.

\bibitem{OpBrHgm12}
O.~Pajonk, B.~V. Rosi\'c, and H.~G. Matthies, \emph{Sampling-free linear
  {B}ayesian updating of model state and parameters using a square root
  approach}, Computers and Geosciences \textbf{55} (2013), 70--83, \href
  {http://dx.doi.org/10.1016/j.cageo.2012.05.017}
  {\path{doi:10.1016/j.cageo.2012.05.017}}.

\bibitem{Papoulis1998/107}
A.~Papoulis, \emph{Probability, random variables, and stochastic processes},
  third ed., McGraw-Hill Series in Electrical Engineering, McGraw-Hill, New
  York, 1991.

\bibitem{ParnoTMYM:2015-arxiv}
M.~Parno, T.~Moselhy, and Y.~Marzouk, \emph{A multiscale strategy for
  {B}ayesian inference using transport maps}, arXiv:1507.07024v1 [stat:CO],
  2015, Available from: \url{http://arxiv.org/abs/1507.07024}.

\bibitem{rao2005}
M.~M. Rao, \emph{Conditional measures and applications}, CRC Press, Boca Raton,
  FL, 2005.

\bibitem{Roman_Sarkis_06}
L.~Roman and M.~Sarkis, \emph{Stochastic {G}alerkin method for elliptic
  {SPDE}s: A white noise approach}, Discrete Cont. Dyn. Syst. Ser. B \textbf{6}
  (2006), 941--955.

\bibitem{BvrAkJsOpHgm11}
B.~V. Rosi\'c, A.~Ku\v{c}erov\'a, J.~S\'ykora, O.~Pajonk, A.~Litvinenko, and
  H.~G. Matthies, \emph{Parameter identification in a probabilistic setting},
  Engineering Structures \textbf{50} (2013), 179--196, \href
  {http://dx.doi.org/10.1016/j.engstruct.2012.12.029}
  {\path{doi:10.1016/j.engstruct.2012.12.029}}.

\bibitem{bvrAlOpHgm12-a}
B.~V. Rosi\'c, A.~Litvinenko, O.~Pajonk, and H.~G. Matthies,
  \emph{Sampling-free linear {B}ayesian update of polynomial chaos
  representations}, Journal of Computational Physics \textbf{231} (2012),
  5761--5787, \href {http://dx.doi.org/10.1016/j.jcp.2012.04.044}
  {\path{doi:10.1016/j.jcp.2012.04.044}}.

\bibitem{rosic2013hgm}
B.~V. Rosi{\'c} and H.~G. Matthies, \emph{Identification of properties of
  stochastic elastoplastic systems}, Computational Methods in Stochastic
  Dynamics (Berlin) (M.~Papadrakakis, G.~Stefanou, and V.~Papadopoulos, eds.),
  Computational Methods in Applied Sciences, vol.~26, Springer, 2013,
  pp.~237--253, \href {http://dx.doi.org/10.1007/978-94-007-5134-7\_14}
  {\path{doi:10.1007/978-94-007-5134-7\_14}}.

\bibitem{saadGhn:2009}
G.~Saad and R.~Ghanem, \emph{Characterization of reservoir simulation models
  using a polynomial chaos-based ensemble {K}alman filter}, Water Resources
  Research \textbf{45} (2009), W04417, \href
  {http://dx.doi.org/10.1029/2008WR007148} {\path{doi:10.1029/2008WR007148}}.

\bibitem{Sanz-Alonso_Stuart:2014}
D.~Sanz-Alonso and A.~M. Stuart, \emph{Long-time asymptotics of the filtering
  distribution for partially observed chaotic dynamical systems},
  arXiv:1411.6510v1 [math.DS], 2014, Available from:
  \url{http://arxiv.org/abs/1411.6510}.

\bibitem{Segal1978}
I.~E. Segal and R.~A. Kunze, \emph{Integrals and operators}, Springer, Berlin,
  1978.

\bibitem{Stuart2010}
A.~M. Stuart, \emph{Inverse problems: A {Bayesian} perspective}, Acta Numerica
  \textbf{19} (2010), 451--559, \href
  {http://dx.doi.org/10.1017/S0962492910000061}
  {\path{doi:10.1017/S0962492910000061}}.

\bibitem{Tarantola2004}
A.~Tarantola, \emph{Inverse problem theory and methods for model parameter
  estimation}, SIAM, Philadelphia, PA, 2004.

\bibitem{Tarn-Rasis:1976}
T.-J. Tarn and Y.~Rasis, \emph{Observers for nonlinear stochastic systems},
  IEEE Transactions on Automatic Control \textbf{21} (1976), 441--448.

\bibitem{Wiener1938}
N.~Wiener, \emph{The homogeneous chaos}, American Journal of Mathematics
  \textbf{60} (1938), no.~4, 897--936.

\bibitem{xiuKarniadakis02a}
D.~Xiu and G.~E. Karniadakis, \emph{The {W}iener-{A}skey polynomial chaos for
  stochastic differential equations}, SIAM Journal of Scientific Computing
  \textbf{24} (2002), 619--644.

\end{thebibliography}

\providecommand{\bysame}{\leavevmode\hbox to3em{\hrulefill}\thinspace}
\providecommand{\MR}{\relax\ifhmode\unskip\space\fi MR }
% \MRhref is called by the amsart/book/proc definition of \MR.
\providecommand{\MRhref}[2]{%
  \href{http://www.ams.org/mathscinet-getitem?mr=#1}{#2}
}
\providecommand{\href}[2]{#2}

%\end{thebibliography}

{ %\color{gray9}
   \tiny
       \texttt{\RCSId} 
%   \begin{verbatim}
%    $Id: NonLinBU.tex,v 3.9 2015/10/28 15:22:19 matthies Exp $
%   \end{verbatim}
   }

%\clearpage
%\printindex

%\clearpage
%\include{kber}

\end{document}